%% file: L_p-final_arxive.tex
\documentclass[preprint,11pt]{imsart}

\RequirePackage[OT1]{fontenc}
\RequirePackage{amsthm,amsmath}
\RequirePackage[colorlinks,citecolor=blue,urlcolor=blue]{hyperref}

\usepackage{a4wide,xcolor,eucal,enumerate,mathrsfs,graphics,caption,subfig,booktabs,verbatim}
\usepackage{pdfsync}
\usepackage{amsmath,amssymb,epsfig,amsthm}
\usepackage[latin1]{inputenc}
\usepackage{dsfont}
\usepackage{enumitem}

\startlocaldefs
\theoremstyle{definition}
\newtheorem{de}{Definition}[section]

\theoremstyle{plain}
\newtheorem{theo}[de]{Theorem}
\newtheorem{lemma}[de]{Lemma}
\newtheorem{prop}[de]{Proposition}
\newtheorem{cor}[de]{Corollary}
\theoremstyle{remark}
\newtheorem{re}[de]{Remark}

\newcommand{\R}{\mathbb{R}}

\newcommand{\p}{\mathbb{P}}

\newcommand{\E}{\mathbb{E}}

\newcommand{\B}{\mathcal{B}}

\newcommand{\A}{\mathcal{A}}
\newcommand{\1}{\mathds{1}}

\newcommand{\dd}{\mathrm{d}}
\newcommand{\CM}{\mathrm{CM}}
\newcommand{\DM}{\mathrm{D}}
\newcommand{\sgn}{\mathop{\mathrm{sgn}}}

\endlocaldefs

\begin{document}

\input L_p-frontmatter_AOS_RL.tex
\section{Introduction}
The property of monotonicity plays an important role when dealing with survival data or regression relationships. 
For example, it is often natural to assume that increasing a factor~$X$ has a positive (negative) effect on a response~$Y$ or that the risk for an event to happen  is increasing (decreasing) over time.
In situations like these, incorporating monotonicity constraints in the estimation procedure leads to more accurate results. 
The first non-parametric monotone estimators were introduced in~\cite{grenander1956}, \cite{brunk1958}, and~\cite{marshallproschan1965}, 
concerning the estimation of a {monotone probability density, regression function, and failure rate.}
These estimators are all piecewise constant functions that exhibit a non-normal limit distribution at rate~$n^{1/3}$.

On the other hand, under some more regularity assumptions on the function of interest, smooth {non-parametric} estimators can be used to achieve a faster rate of convergence to a Gaussian distributional law. Typically, these estimators are constructed by combining an isotonization step with a smoothing step. Estimators constructed by smoothing followed by an isotonization step have been considered
in~\cite{chenglin1981}, \cite{wright1982}, \cite{friedmantibshirani1984}, and~\cite{ramsay1998}, for the regression setting,
in~\cite{vdvaart-vdlaan2003} for estimating a monotone density,
and in~\cite{eggermont-lariccia2000}, who consider maximum smoothed likelihood estimators for monotone densities.
Methods that interchange the smoothing step and the isotonization step,
can be found in~\cite{mukerjee1988}, \cite{DGL13}, and~\cite{lopuhaa-mustaSN2017}, who study kernel smoothed isotonic estimators.
{Comparisons between isotonized smooth estimators and smoothed isotonic estimators are made in~\cite{mammen1991}, \cite{GJW10} and~\cite{GJ13}.}

A lot of attention has been given in the literature to the pointwise asymptotic behavior of {smooth estimators and monotone estimators, separately.}
However, for example for goodness of fit tests, global errors of estimates are needed instead of pointwise results. 
For the Grenander estimator of a monotone density, a central limit theorem for the $L_1$-error was {formulated in~\cite{groeneboom1985} and proven rigorously in~\cite{groeneboom-hooghiemstra-lopuhaa1999}. }
A similar result was established in~\cite{durot2002} for the regression context. 
Extensions to general $L_p$-errors can be found in~\cite{kulikov-lopuhaa2005} and in~\cite{durot2007}, where the latter provides a unified approach that applies to a variety of statistical models. 
On the other hand, central limit theorems for regular kernel density estimators have been obtained in~\cite{CH88} and~\cite{CGH91}.

{In this paper we investigate the $L_p$-error of smooth isotonic estimators obtained by kernel smoothing the Grenander-type estimator or by isotonizing the ordinary kernel estimator.
We consider the same general setup as in~\cite{durot2007}, which includes estimation of a probability density, a regression function, or a failure rate under monotonicity constraints 
(see Section 3 in~\cite{durot2007} for more details on these models). 
An essential assumption in this setup is that the observed process of interest can be 
approximated by a Brownian motion or a Brownian bridge. 
Our main results are central limit theorems for the $L_p$-error of smooth isotonic estimators for a 
monotone function on a compact interval.
However, since the behavior of these estimators is closely related to the behavior of 
ordinary kernel estimators, we first establish a central limit theorem for the $L_p$-error of ordinary kernel estimators for a monotone function on a compact interval.
This extends the work by~\cite{CH88} on the $L_p$-error of densities that are smooth on the whole real line,
but is also of interest by itself.
The fact that we no longer have a smooth function on the whole real line, leads to boundary effects. 
Unexpectedly, different from~\cite{CH88}, 
we find that the limit variance of the $L_p$-error changes,
depending on whether the approximating process is a Brownian motion or a Brownian bridge. 
Such a phenomenon has also not been observed in other isotonic problems, 
where a similar embedding assumption was made.
Usually, both approximations lead to the same asymptotic results (e.g., see ~\cite{durot2007} and~\cite{kulikov-lopuhaa2005}).}

After establishing a central limit theorem for the $L_p$-error of {ordinary kernel estimators,} 
we transfer this result to the smoothed Grenander estimator.
The key ingredient here is the behavior of the process obtained as the difference between a naive estimator and its least concave majorant.
For this we use results from~\cite{lopuhaa-mustaSPL2017}.
As an intermediate result, we show that the $L_p$-distance between the 
{smoothed Grenander-type estimator and the ordinary} kernel estimator converges at rate $n^{2/3}$ to 
some functional of two-sided Brownian motion minus a parabolic drift.

The situation for the isotonized kernel estimator is much easier, because it can be shown that this estimator coincides with the {ordinary} kernel estimator on large intervals in the interior of the support, 
with probability {tending} to one.
However, since the isotonization step is performed last, the estimator is inconsistent at the boundaries. 
{For this reason}, we can only obtain a central limit theorem for the $L_p$-error on a sub-interval {that} approaches the whole support, as $n$ converges to infinity.
Finally, the results on the $L_p$-error can be applied immediately to obtain {a central limit theorem}
for the Hellinger loss.

The paper is organized as follows. 
In Section~\ref{sec:notation} we describe the model, the assumptions and fix some notation that will be used throughout the paper. 
A central limit theorem for the {$L_p$-error} of the kernel estimator is obtained in Section~\ref{sec:kernel}. 
This result is used in Section~\ref{sec:SG} and~\ref{sec:GS} {to obtain} the limit distribution of the $L_p$-error of the SG and GS estimators. 
Section~\ref{sec:hellinger} is dedicated to {corresponding} asymptotics for the Hellinger distance. 
{In Section~\ref{sec:testing} we provide a possible application of our results by considering a test for monotonicity.
Details of some of the proofs are delayed to Section~\ref{sec:proofs} and to additional technicalities have been put in the supplemental material in~\cite{LM_L_p_2018}.}

\section{Assumptions and notations}
\label{sec:notation}
Consider estimating a function $\lambda:\,[0,1]\to\R$ subject to the constraint  that it is non-increasing. Suppose that on the basis of $n$ observations we  have at hand a cadlag step estimator~$\Lambda_n$ of
\[
\Lambda(t)=\int_0^t\lambda(u)\,\dd u, \quad t\in[0,1].
\]
A typical example is the estimation of a monotone density $\lambda$ on a compact interval by means of the empirical cumulative distribution function ${\Lambda_n}$.
Hereafter $M_n$ denotes the process $M_n=\Lambda_n-\Lambda$, $\mu$ is a measure on the Borel sets of $\R$, {and
\begin{equation}
\label{def:kernel}
\parbox{0.9\textwidth}{$k$ is a twice differentiable symmetric probability density with support $[-1,1]$.}
\end{equation}}%
The rescaled kernel is defined as $k_b(u)=b^{-1}k\left(u/b\right)$ where the bandwidth $b=b_n\to 0$, as~$n\to\infty$.
In the sequel we will make use of the following assumptions.
\begin{enumerate}[label={(A\arabic*)}]
	\item
	$\lambda$ is {decreasing} and twice continuously differentiable on $[0,1]$ with $\inf_t|\lambda'(t)|>0$.
	\item
	Let $B_n$ be either a Brownian motion or a Brownian bridge.
	There exists $q>5/2$, $C_q>0$, $L:[0,1]\to\R$ and versions of $M_n$ and $B_n$ such that
	\[
	\p\left(n^{1-1/q}\sup_{t\in[0,1]}\left|M_n(t)-n^{-1/2}B_n\circ L(t)\right|>x \right)\leq C_qx^{-q}
	\]
	for all $x\in(0,n]$. Moreover, $L$ is increasing and twice differentiable on $[0,1]$ with $\sup_t|L''(t)|<\infty$ and $\inf_t|L'(t)|>0.$
	\item $\dd\mu(t)=w(t)\,\dd t,$ where $w(t)\geq 0$ is continuous on $[0,1]$.
\end{enumerate}
In particular, the approximation of the process $M_n$ by a Gaussian process, as in assumption~(A2), is required  also in~\cite{durot2007}. It corresponds to a general setting which includes estimation of a probability density, regression function or a failure rate under monotonicity constraints (see Section 3 in~\cite{durot2007} for more details on {these} models).

First we introduce some notation.
{We partly adopt the one used in~\cite{CH88} and briefly explain their appearance.
Let $\tilde{\lambda}_n^s$ be the standard kernel estimator of $\lambda$, i.e.
\begin{equation}
\label{def:kernel_est}
\tilde{\lambda}_n^s(t)=\int_{t-b}^{t+b} k_b(t-u)\,\mathrm{d}\Lambda_n(u),
\quad
{\text{for }t\in[b,1-b].}
\end{equation}
As usual we decompose into a random term and a bias term:
\begin{equation}
\label{eq:decompose}
(nb)^{1/2}
\left(\tilde{\lambda}_n^s(t)
-
\lambda(t)
\right)
=
(nb)^{1/2}
\int k_b(t-u)\,\dd(\Lambda_n-\Lambda)(u)+g_{(n)}(t)
\end{equation}
where}
\begin{equation}
\label{eqn:def-g_n}
g_{(n)}(t)=(nb)^{1/2}\left(\lambda_{(n)}(t)-\lambda(t) \right),\qquad \lambda_{(n)}(t)=\int k_b(t-u)\lambda(u)\,\dd u.
\end{equation}
{When $nb^5\to C_0>0$, then $g_{(n)}(t)$ converges to
\begin{equation}
\label{eqn:def-g(u)}
g(t)=\frac{1}{2}C_0\lambda''(t)\int k(y)y^2\,\dd y.
\end{equation}	
After separating the bias term, the first term on the right hand side
of~\eqref{eq:decompose} involves an integral of $k_b(t-u)$ with respect to the process $M_n$.
Due to~(A2), this integral will be approximated by an integral with respect to a Gaussian process.
For this reason, the limiting moments of the {$L_p$-error} involve integrals with respect to
Gaussian densities, such as
\begin{equation}
\label{eqn:psi}
\begin{split}
\phi(x)&=
(2\pi)^{-1/2}\exp(-x^2/2),\\
\psi(u,x,y)
&=
\frac{1}{2\pi\sqrt{1-u^2}}\exp\left(-\frac{x^2-2uxy+y^2}{2(1-u^2)} \right)
=
\frac{1}{\sqrt{1-u^2}}
\phi\left(\frac{x-uy}{\sqrt{1-u^2}}\right)\phi(y),
\end{split}
\end{equation}
and a Taylor expansion of $k_b(t-u)$ yields the following constants involving the kernel {function}:}
\begin{equation}
\label{eqn:r(s)}
D^2=\int k(y)^2\,\dd y,\qquad r(s)=\frac{\int k(z)k(s+z)\,\dd z}{\int k^2(z)\,\dd z}.
\end{equation}
{For example, the limiting means of the {$L_p$-error} and a truncated version are given by:
\begin{equation}
\label{eqn:def-m_n}
\begin{split}
m_n(p)
&=
\int_{\R}\int_0^{1}\left|\sqrt{L'(t)}Dx+g_{(n)}(t) \right|^p w(t)\phi(x)\,\dd t\,\dd x,\\
m_n^c(p)
&=
\int_{\R}\int_b^{1-b}\left|\sqrt{L'(t)}Dx+g_{(n)}(t) \right|^p w(t)\phi(x)\,\dd t\,\dd x,
\end{split}
\end{equation}
where $D$ and $g_{(n)}$ are defined in~\eqref{eqn:r(s)} and~\eqref{eqn:def-g_n}.
Depending on the rate at which $b\to 0$, the limiting variance of the {$L_p$-error}
has a different form.
When $nb^5\to0$, the limiting variance turns out to be}
\begin{equation}
\label{eqn:def-sigma}
\sigma^2(p)=\sigma_1D^{2p}\int_{0}^{1}\left|L'(u)\right|^{p}w(u)^2\,\dd u,
\end{equation}
{where}
\begin{equation}
\label{eqn:sigma_1}
\sigma_1=\int_{\R}\left\{\int_{\R}\int_{\R}|xy|^p\psi(r(s),x,y)\,\dd x\,\dd y-\int_{\R}\int_{\R}|xy|^p\phi(x)\phi(y)\,\dd x\,\dd y\right\} \,\dd s,
\end{equation}
{with $\sigma_1$ representing $p$-th moments of bivariate Gaussian vectors,
where $D$, $\psi$, and $\phi$ are defined in~\eqref{eqn:r(s)} and~\eqref{eqn:psi}.
When $nb^5\to C_0>0$ and~$B_n$ in~(A2) is a Brownian motion, the limiting variance of the {$L_p$-error} is}
\begin{equation}
\label{eqn:def-theta}
\begin{split}
\theta^2(p)
&=
\int_0^1\int_{\R^3}\left|g(u)^2+g(u)(x+y)\sqrt{L'(u)}D+D^2L'(u)xy \right|^p\\
&\qquad\qquad\qquad
w^2(u)
\Big(
\psi(r(s),x,y) -\phi(x)\phi(y)
\Big)\,\dd s\,\dd y\,\dd x\,\dd u,
\end{split}
\end{equation}
{where $g$, $D$, $\psi$, and $\phi$ are defined in~\eqref{eqn:def-g(u)}, \eqref{eqn:r(s)} and~\eqref{eqn:psi},
whereas, if $B_n$ in~(A2) is a Brownian bridge, the limiting variance is slightly different,
\begin{equation}
\label{eqn:def-theta-tilde}
\tilde{\theta}^2(p)=\theta^2(p)-\frac{\theta_1^2(p)}{D^2L(1)},
\end{equation}
with}
\begin{equation}
\label{eqn:def-theta1}
\theta_1(p)=\int_0^{1}\int_{\R}\left|\sqrt{L'(t)}D x+g(t) \right|^p x\phi(x)\,\dd x\,\sqrt{L'(t)}w(t)\dd t.
\end{equation}
{Finally,} the following inequality will be used throughout this paper:
\begin{equation}
\label{eqn:L_p-inequality}
\begin{split}
&
\int_A^B \left||q(t)|^p-|h(t)|^p \right| \,\dd \mu(t)
\leq p 2^{p-1}\int_A^B \left|q(t)-h(t) \right|^p \,\dd \mu(t)\\
&\qquad\qquad\qquad
+p 2^{p-1}\left(\int_A^B \left|h(t) \right|^p \,\dd \mu(t)\right)^{1-1/p}\left(\int_A^B \left|q(t)-h(t) \right|^p \,\dd \mu(t)\right)^{1/p},
\end{split}
\end{equation}
where $p\in[1,\infty)$,  $-\infty\leq A<B\leq \infty$ and $q,h\in L_p(A,B)$.

\section{Kernel estimator of a decreasing function}
\label{sec:kernel}
We extend the results of \cite{CH88} and \cite{CGH91} to the case of {a} kernel estimator of a decreasing function with compact support.
Note that, since the function of interest {cannot} be twice differentiable {on} $\R$ (not even continuous), the kernel estimator is inconsistent at zero and one.
Moreover we show that the contribution of the boundaries to the $L_p$-error is not negligible, so in order to avoid the $L_p$-distance to explode we have to restrict ourselves to the interval $[b,1-b]$ {or apply} some boundary correction.

\subsection{A modified $L_p$-distance of the standard kernel estimator}
\label{subsec:Lp kernel}
Let $\tilde{\lambda}_n^s$ be the standard kernel estimator of $\lambda$ {defined in~\eqref{def:kernel_est}.}
In order to avoid boundary problems, we start by finding the asymptotic distribution of a modification of the $L_p$-distance
\begin{equation}
\label{def:Jnc}
J^c_n(p)=\int_b^{1-b}\left|\tilde{\lambda}^s_n(t)-\lambda(t)\right|^p\,\dd\mu(t),
\end{equation}
instead of
\begin{equation}
\label{def:Jn}
J_n(p)=\int_0^1\left|\tilde{\lambda}^s_n(t)-\lambda(t)\right|^p\,\dd\mu(t).
\end{equation}
\begin{theo}
\label{theo:as.distribution_J_n}
Assume that (A1)-(A3) hold.
{Let $k$ satisfy~\eqref{def:kernel}} and let $J^c_n$ be defined in~\eqref{def:Jnc}.
Suppose~$p\geq 1$ and $nb\to\infty$.
\begin{enumerate}[label={\roman*)}]
\item If $nb^5\to0$, {then}
\[
(b\sigma^2(p))^{-1/2}\left\{(nb)^{p/2}J^c_n(p)-m_n^c(p) \right\}\xrightarrow{d} N(0,1);
\]
\item If $nb^5\to C_0^2>0$, and $B_n$ in Assumption $(A2)$ is a Brownian motion,
{then}
\[
(b\theta^2(p))^{-1/2}\left\{(nb)^{p/2}J_n^c(p)-m_n^c(p) \right\}\xrightarrow{d} N(0,1);
\]
\item If $nb^5\to C_0^2>0$, and $B_n$ in Assumption $(A2)$ is a Brownian bridge,
{then}
\[
(b\tilde{\theta}^2(p))^{-1/2}\left\{(nb)^{p/2}J_n^c(p)-m_n^c(p) \right\}\xrightarrow{d} N(0,1),
\]
\end{enumerate}
{where
$m_n^c(p)$, $\sigma^2(p)$, $\theta^2(p)$, $\tilde\theta^2(p)$ are defined in
\eqref{eqn:def-m_n}, \eqref{eqn:def-sigma}, \eqref{eqn:def-theta}, and \eqref{eqn:def-theta-tilde}, respectively.}
\end{theo}
The proof goes along the same lines as in {the one for} the case of the $L_p$-norms {for} kernel density estimators {on the whole real line} (see \cite{CH88} and \cite{CGH91}).
The main idea is that by means of assumption~(A2), it is sufficient to prove the central limit theorem for the approximating process.
{When $B_n$ in (A2) is a Brownian motion},
the latter one can be obtained by a big-blocks-small-blocks procedure using the independence of the increments of the Brownian motion.
{When $B_n$ in (A2) is a Brownian bridge, we can still obtain a central limit theorem, but the limiting variance turns out to be different.
The latter result differs from what} is stated in~\cite{CH88}.
In \cite{CH88}, {the complete proof for both Brownian motion and Brownian bridge, is only given} for the case $nb^5\to 0$,
and it is shown that the random variables obtained by using the Brownian motion and the Brownian bridge as approximating processes are asymptotically equivalent (see their Lemma 6).
In fact, when dealing with a Brownian bridge, the rescaled $L_p$-error is asymptotically equivalent to the $L_p$-error that corresponds to the Brownian motion process plus an additional term which is equal to $CW(L(1))$,
for a constant $C$ proportional on $\theta_1(p)$ defined in~\eqref{eqn:def-theta1}.
When the bandwidth is small, {i.e., $nb^5\to0$}, the bias term $g(t)$ in the definition of~$\theta_1(p)$ disappears.
Hence, by the symmetry property of the standard normal density, $\theta_1(p)=0$ and as a consequence $C=0$.
This means that the additional term resulting from the fact that we are dealing with a Brownian bridge converges to zero.
For details, see the proof of Lemma \ref{lemma:Gamma2}.
When $nb^5\to C_0^2>0$, only a sketch of the proof is given in \cite{CH88} {for~$B_n$ being a Brownian motion}
and it is claimed that again the limit distribution would be the same for {$B_n$ being a Brownian bridge.}
However, {in out setting we find that the limit variances are different.}
\begin{proof}[Proof of Theorem~\ref{theo:as.distribution_J_n}]
From the definition of $J^c_n(p)$ we have
\[
(nb)^{p/2}J^c_n(p)
=
\int_b^{1-b}\left|(nb)^{1/2}\int k_b(t-u)\,\dd(\Lambda_n-\Lambda)(u)+g_{(n)}(t) \right|^p\,\dd\mu(t).
\]
Let $(W_t)_{t\in\R}$ be a Wiener process and define
\begin{equation}
\label{def:Gamma1}
\Gamma^{(1)}_n(t)=\int k\left(\frac{t-u}{b}\right)\,\dd W(L(u)),
\end{equation}
Hence, if $B_n$ in assumption (A2) is a Brownian motion, {then according to~\eqref{eqn:L_p-inequality},}
{\[
\begin{split}
&
\left|(nb)^{p/2}J^c_n(p)-\int_b^{1-b}\left|b^{-1/2}\Gamma^{(1)}_n(t)+g_{(n)}(t)\right|^p\,\dd \mu(t)\right|\\
&\qquad\leq
p2^{p-1}b^{-p/2}\int_b^{1-b}\left|\int k\left(\frac{t-u}{b}\right)\,\dd (B_n\circ L(u)-n^{1/2}M_n(u))\right|^p\,\dd \mu(t)\\
&\qquad\qquad+
p2^{p-1}\left(b^{-p/2}\int_b^{1-b}\left|\int k\left(\frac{t-u}{b}\right)\,\dd (B_n\circ L-n^{1/2}M_n)(u)\right|^p\,\dd \mu(t)\right)^{1/p}\cdot\\
&\qquad\qquad\qquad\qquad\qquad\qquad\cdot
\left( \int_b^{1-b}\left|b^{-1/2}\Gamma^{(1)}_n(t)+g_{(n)}(t)\right|^p\,\dd \mu(t)\right)^{1-1/p}
\end{split}
\]}
{We can write}
\begin{equation}
\label{eqn:embedding_approximation}
\begin{split}
\left|\int k\left(\frac{t-u}{b}\right)\,\dd (B_n\circ L-n^{1/2}M_n)(u) \right|
&=
\left|\int_{-1}^1 k(y)\,\dd (B_n\circ L-n^{1/2}M_n)(t-by) \right|\\
&=
\left|\int_{-1}^1 (B_n\circ L-n^{1/2}M_n)(t-by)\,\dd k(y) \right|\\
&\leq
C\sup_{t\in[0,1]}\left|B_n\circ L(t)-n^{1/2}M_n(t)\right|.
\end{split}
\end{equation}
{According to assumption (A2), the right hand side of~\eqref{eqn:embedding_approximation}
is of the order $O_P(n^{-1/2+1/q})$, and because} $b^{-1/2}O_P(n^{-1/2+1/q}){=(nb^5)^{3/10}O_P(n^{-2/5+1/q})}=o_P(1)$ we derive that
\[
\left|(nb)^{p/2}
{J_n^c(p)}-\int_b^{1-b}\left|b^{-1/2}\Gamma^{(1)}_n(t)+g_{(n)}(t)\right|^p\,\dd \mu(t)\right|=o_P(1).
\]
As a result, the statement follows from {the fact that
\[
(b\sigma^2(p))^{-1/2}\left\{\int_b^{1-b}\left|b^{-1/2}\Gamma^{(1)}_n(t)+g_{(n)}(t) \right|^p\,\dd \mu(t)-m_n^c(p) \right\}\xrightarrow{d} N(0,1),
\]
where $g_{(n)}$ and $m_n^c(p)$ are defined in~\eqref{eqn:def-g_n} and~\eqref{eqn:def-m_n}, respectively.
This result is a generalization of Lemmas~1-5 in~\cite{CH88} and the proof goes in the same way.
However, for completeness we give all the details in the supplementary material.
See Lemma~\ref{lemma:Gamma1} in~\cite{LM_L_p_2018}.
}
	
{Finally}, if $B_n$ is a Brownian bridge on $[0,L(1)]$, we use the representation {$B_n(t)=W(t)-tW(L(1))/L(1)$.}
By replacing $\Gamma^{(1)}_n$ with
\begin{equation}
\label{def:Gamma2}
\Gamma^{(2)}_n(t)=\int k\left(\frac{t-u}{b} \right)\,\dd \left(W(L(u))-\frac{L(u)}{L(1)}W(L(1))\right)
\end{equation}
in the previous reasoning, the statement follows from Lemma~\ref{lemma:Gamma2}.
\end{proof}
{When $nb^4\to0$, the centering constant $m_n(p)$ can be replaced by a quantity that does not depend on~$n$.}
\begin{theo}
Assume that (A1)-(A3) hold.
{Let $k$ satisfy~\eqref{def:kernel}} and let $J^c_n$ be defined in~\eqref{def:Jnc}.
Suppose~$p\geq 1$ and $nb\to\infty$, such that $nb^4\to0$.
Then
\[
(b\sigma^2(p))^{-1/2}\left\{(nb)^{p/2}J^c_n(p)-m(p) \right\}\xrightarrow{d} N(0,1),
\]
where {$\sigma^2(p)$ is defined in~\eqref{eqn:def-sigma} and}
\[
m(p)=
{\int_{\R}
|x|^p\phi(x)\,\dd x}
\left(\int k^2(t)\,\dd t\right)^{p/2}\int_0^1 |L'(t)|^{p/2}\,\dd \mu(t).
\]
\end{theo}
\begin{proof}
{The statement} follows from Theorem~\ref{theo:as.distribution_J_n}, if $|m_n^c(p)-m(p)|=o(b^{1/2})$.
First we note that $\int_0^b |L'(t)|^{p/2}\,\dd \mu(t)=o(b^{1/2})$ and
${\int_{1-b}^1} |L'(t)|^{p/2}\,\dd \mu(t)=o(b^{1/2})$.
Moreover, {according to~\eqref{eqn:L_p-inequality}}, for each $x\in\R$, we have
\[
\begin{split}
&\int_b^{1-b}\left|\left|\sqrt{L'(t)}Dx+g_{(n)}(t)\right|^p-\left|\sqrt{L'(t)}Dx \right|^p \right|\,\dd \mu(t)\\
&\quad \leq p2^{p-1}\int_b^{1-b}\left|g_{(n)}(t)\right|^p\,\dd \mu(t)\\
&\qquad +p2^{p-1}\left(\int_b^{1-b}\left|\sqrt{L'(t)}Dx\right|^p\,\dd \mu(t)\right)^{1-1/p}\left(\int_b^{1-b}\left|g_{(n)}(t)\right|^p\,\dd \mu(t) \right)^{1/p},
\end{split}
\]
{where $g_{(n)}(t)$ is defined in~\eqref{eqn:def-g_n}.}
Hence, it suffices to prove
\[
b^{-p/2}\int_b^{1-b}\left|g_{(n)}(t)\right|^p\,\dd \mu(t)=o(1).
\]
{This follows, since} $\sup_{t\in[0,1]}\left|g_{(n)}(t)\right|=O((nb)^{1/2}b^2)$ and $b^{-p/2}(nb)^{p/2}b^{2p}=(nb^4)^{p/2}\to 0.$
\end{proof}

\subsection{Boundary problems of the standard kernel estimator}

We show that, actually, we {cannot} extend the results of Theorem~\ref{theo:as.distribution_J_n} to the whole interval $[0,1]$,
because then the inconsistency at the boundaries dominates the $L_p$-error.
A similar phenomenon was also observed in the case of the Grenander-type estimator (see~\cite{durot2007} and~\cite{kulikov-lopuhaa2005}), but only for $p\geq 2.5$.
In our case the contribution of the boundaries to the $L_p$-error is not negligible for all $p\geq 1$.
{This} {mainly has} to do with the fact that the {functions $g_{(n)}$,} defined in~\eqref{eqn:def-g_n}, converge to infinity.
{As} a result, all the previous theory, which relies on the fact that $g_{(n)}=O(1)$ does not hold.
{For example}, for $t\in[0,b)$, we have
{\begin{equation}
\label{eqn:bias_g_n}
\begin{split}
g_{(n)}(t)
&=(nb)^{1/2}\int_0^{t+b} k_b(t-u)\,\dd \Lambda(u)-\lambda(t)\\
&=
(nb)^{1/2}\int_{-1}^{t/b}k(y)[\lambda(t-by)-\lambda(t)]\,\dd y-(nb)^{1/2}\lambda(t)\int_{t/b}^{1}k(y)\,\dd y\\
&=
(nb)^{1/2}
\left\{
\int_{-1}^{t/b}k(y)[\lambda(t-by)-\lambda(t)]\,\dd y-\lambda(t)\int_{t/b}^{1}k(y)\,\dd y
\right\}.
\end{split}
\end{equation}
For the first term within the brackets, we have
\begin{equation}
\label{eq:order term1}
\left|\int_{-1}^{t/b}k(y)[\lambda(t-by)-\lambda(t)]\,\dd y\right|\leq b\sup_{t\in[0,1]}|\lambda'(t)|\left|\int_{-1}^{t/b}k(y)y\,\dd y\right|=O(b),
\end{equation}
whereas for any $0<c<1$ and $t\in[0,cb]$,
\begin{equation}
\label{eq:order term2}
0<
\inf_{t\in[0,1]}\lambda(t)\int_{c}^{1}k(y)\,\dd y
\leq
\lambda(t)\int_{t/b}^{1}k(y)\,\dd y
\leq \lambda(0).
\end{equation}
Because $nb\to\infty$, this would mean that
\begin{equation}
\label{eq:gn negative}
\sup_{t\in [0,cb]}g_{(n)}(t)
\to-\infty.
\end{equation}
}
What would solve the problem is {to assume} that $\lambda$ is twice differentiable as a function defined {on} $\R$ (see \cite{CH88} and \cite{CGH91}).
This is not the case, because here we are considering a function which is positive and decreasing on $[0,1]$ and usually is zero outside this interval.
This means that as a function {on} $\R$, $\lambda$ is not monotone {anymore} and has at least one discontinuity point.

The following {results indicate} that inconsistency at the boundaries dominates the $L_p$-error, i.e., the expectation and the variance of the integral close to the end points of the support converge to infinity.
We cannot even approach the boundaries at a rate faster than $b$ (as in the case of the Grenander-type estimator),
because the kernel estimator is inconsistent on the whole interval $[0,b)$ (and $(1-b,1]$).
\begin{prop}
\label{prop:boundaries}
Assume that (A1)-(A3) hold
{and let $\tilde{\lambda}^s_n$ be defined in~\eqref{def:kernel_est}.
Let $k$ satisfy~\eqref{def:kernel}}.
Suppose that~$p\geq 1$ and $nb\to\infty$.
\begin{enumerate}[label={\roman*)}]
\item
{When $nb^3\to \infty$, then for each $p\geq 1$,}
\[
(nb)^{p/2}\E\left[\int_0^{b}\left|\tilde{\lambda}^s_n(t)-\lambda(t)\right|^p\,\dd\mu(t)\right]\to\infty;
\]

\item If $bn^{1-1/p}\to 0$, then
\[
b^{-1/2}\left\{\int_0^b(nb)^{p/2}\left|\tilde{\lambda}^s_n(t)-\lambda(t)\right|^p\,\dd \mu(t)-\int_0^b\left| g_{(n)}(t)\right|^p\,\dd \mu(t) \right\}\to 0,
\]
{where $g_{(n)}$ is defined in~\eqref{eqn:def-g_n};}
\item Let
\begin{equation}
\label{def:Y_n}
Y_n(t)=b^{1/2}\int_0^{t+b}k_b(t-u)\,\dd B_n(L(u)),\qquad t\in[0,b].
\end{equation}
If  {$b^{-1}n^{-1+1/q}=O(1)$} and $b^{p-1}n^{p-2+2/q}\to 0$, then
\begin{equation}
\label{eqn:approx_boundaries}
b^{-1/2}\left|\int_0^b(nb)^{p/2}\left|\tilde{\lambda}^s_n(t)-\lambda(t)\right|^p\,\dd \mu(t)-\int_0^b\left|Y_n(t)+g_{(n)}(t)\right|^p\,\dd \mu(t) \right|\to 0,
\end{equation}
{in probability} and {when $bn^{1-1/p}\to\infty$, then for all $0<c<1$,}
\[
b^{-1}\mathrm{Var}\left(\int_0^{cb}\left|{Y_n(t)+g_{(n)}(t)}\right|^p\,\dd \mu(t) \right)\to\infty,
\]
{where $g_{(n)}$ is defined in~\eqref{eqn:def-g_n}.}
\end{enumerate}
The previous results also hold if we consider the integral on $(1-b,1]$ instead of $[0,b)$.
\end{prop}
The proof can be found in {the supplemental material~\cite{LM_L_p_2018}.}
\begin{re}
Note that, {if $b\sim n^{-\alpha}$, for some $0<\alpha<1$, then for $\alpha<1/3$,
Proposition~\ref{prop:boundaries}(i) shows that for all $p\geq 1$,
the expectation of the boundary regions in the {$L_p$-error} tends to infinity.
This holds in particular for the optimal choice $\alpha=1/5$.
For} $p<1/(1-\alpha)$, Proposition~\ref{prop:boundaries}(ii) allows us to include the boundary regions in the central limit theorem for the $L_p$-error of the kernel estimator,
\[
(b\sigma^2(p))^{-1/2}\left\{(nb)^{p/2}J_n(p)-\bar{m}_n(p) \right\}\xrightarrow{d} N(0,1),
\]
with {$J_n(p)$ defined in~\eqref{def:Jn} and}
$\bar{m}_n(p)={\int_0^1\left|g_{(n)}(t)\right|^p\,\dd \mu(t)}$.
However, the bias term $\bar{m}_n(p)$ is not bounded {anymore}.
On the other hand, if $p>1/(1-\alpha)$, Proposition~\ref{prop:boundaries}{(iii) shows that
the boundary regions in the {$L_p$-error} behave asymptotically as random variables whose variance tends to
infinity.}
\end{re}
\begin{re}
The choice of the measure $\mu$ instead of the Lebesgue measure, in~\cite{CH88} and~\cite{CGH91}, is motivated by the fact that,
for a particular $\mu(t)=w(t)\dd t$, the normalizing constants $m(p)$ and $\sigma(p)$ in the CLT will not depend on the unknown function.
In our case, a proper choice {for}~$\mu$  can {also be used} to get rid of the boundary problems.
This happens when $\mu$ puts less mass on the boundary regions in order to compensate the inconsistency of the kernel estimator.
For example, if $\mu(t)=t^{2p}(1-t)^{2p}\dd t$,
then
\[
\int_0^b|g_{(n)}(t)|^p\,\dd\mu(t)+\int_{1-b}^1|g_{(n)}(t)|^p\,\dd\mu(t)\to 0
\]
and, as a result, Theorem \ref{theo:as.distribution_J_n} {also holds} if we replace $J_n^c(p)$ with $J_n(p)$, {defined in~\eqref{def:Jn}.}
\end{re}

\subsection{Kernel estimator with boundary correction}

One way to overcome the inconsistency problems of the standard kernel estimator is to apply some boundary correction. Let now $\hat{\lambda}_n^{s}$ be the 'corrected' kernel estimator of $\lambda$, i.e.
\begin{equation}
\label{def:corrected_kernel_est}
\hat{\lambda}_n^s(x)=\int_{x-b}^{x+b} k^{(x)}_b(x-u)\,\mathrm{d}\Lambda_n(u),
\quad
{\text{for }x\in[0,1],}
\end{equation}
{where} $k^{(x)}_b(u)$ {denotes} the rescaled kernel $b^{-1}k^{(x)}(u/b)$,
{with}
\begin{equation}
\label{eqn:boundary_kernel}
k^{(x)}(u)=\begin{cases}
{\psi_1}\left(\frac{x}{b}\right)k(u)
+
{\psi_2}\left(\frac{x}{b}\right)uk(u) & x\in[0,b)\\
k(u) & x\in[b,1-b]\\
{\psi_1}\left(\frac{1-x}{b}\right)k(u)
-
{\psi_2}\left(\frac{1-x}{b}\right)uk(u) & x\in(1-b,1].
\end{cases}
\end{equation}
For $s\in[-1,1]$, the coefficients ${\psi_1}(s)$, ${\psi_2}(s)$ are determined by
\[
\begin{split}
{\psi_1}(s)\int_{-1}^sk(u)\,\dd u&+{\psi_2}(s)\int_{-1}^s u k(u)\,\dd u=1\\
{\psi_1}(s)\int_{-1}^suk(u)\,\dd u&+{\psi_2}(s)\int_{-1}^s u^2 k(u)\,\dd u=0.
\end{split}
\]
As a result, the boundary corrected kernel satisfies
\begin{equation}
\label{eq:unbiased boundary kernel}
\int_{-1}^{x/b} k^{(x)}(u)\,\dd u=1\qquad\text{and}\qquad \int_{-1}^{x/b} uk^{(x)}(u)\,\dd u=0.
\end{equation}
Moreover, {$\psi_1$ and $\psi_2$} are continuously differentiable (in particular {they} are bounded).
We aim at showing that in this case,
Theorem~\ref{theo:as.distribution_J_n} holds for the $L_p$-error on the whole support, i.e., with $J_n(p)$ instead of $J_n^c(p)$.

Note that boundary corrected kernel estimator coincides with the standard kernel estimator on $[b,1-b]$. Hence the behavior of the $L_p$-error on $[b,1-b]$ will be the same. We just have to deal with the boundary regions $[0,b]$ and $[1-b,1]$.
\begin{prop}
\label{prop:clt_boundaries}
Assume that (A1)-(A3) hold
{and let $\hat{\lambda}^s_n$ be defined in~\eqref{def:corrected_kernel_est}.}
{Let $k$ satisfy~\eqref{def:kernel}} and suppose~$p\geq 1$ and $nb\to\infty$.
Then
\[
b^{-1/2}(nb)^{p/2}\int_0^b\left|\hat{\lambda}^s_n(t)-\lambda(t) \right|^p\dd\mu(t)\to 0.
\]
The previous {result also holds} if we consider the integral on $(1-b,1]$ instead of $[0,b)$.
\end{prop}
The proof can be found in {the supplemental material~\cite{LM_L_p_2018}.}
\begin{cor}
\label{cor:as.distribution_J_n}
Assume that (A1)-(A3) hold
{and let $J_n(p)$ be defined in~\eqref{def:Jn}.
Let $k$ satisfy~\eqref{def:kernel}} and suppose~$p\geq 1$ and $nb\to\infty$.
Then
\begin{enumerate}[label={\roman*)}]
\item if $nb^5\to0$, then it holds
\[
(b\sigma^2(p))^{-1/2}\left\{(nb)^{p/2}J_n(p)-m_n(p) \right\}\xrightarrow{d} N(0,1);
\]
\item If $nb^5\to C_0^2>0$ and $B_n$ in Assumption (A2) is a Brownian motion, then it holds
\[
(b\theta^2(p))^{-1/2}\left\{(nb)^{p/2}J_n(p)-m_n(p) \right\}\xrightarrow{d} N(0,1);
\]
\item If $nb^5\to C_0^2>0$ and $B_n$ in Assumption (A2) is a Brownian bridge, then it holds
\[
(b\tilde\theta^2(p))^{-1/2}\left\{(nb)^{p/2}J_n(p)-m_n(p) \right\}\xrightarrow{d} N(0,1),
\]
\end{enumerate}
where $\sigma^2$, $\theta^2$, $\tilde\theta^2$ and $m_n$ are defined respectively in \eqref{eqn:def-sigma}, \eqref{eqn:def-theta}, \eqref{eqn:def-theta-tilde} and \eqref{eqn:def-m_n}.
\end{cor}
\begin{proof}
It follows from combining Theorem~\ref{theo:as.distribution_J_n} and Proposition~\ref{prop:clt_boundaries},
{together with the fact that}
\[
b^{-1/2}\int_{\R}\int_0^{b}\left|\sqrt{L'(t)}Dx+g_{(n)}(t) \right|^p w(t)\phi(x)\,\dd t\,\dd x\to 0,
\]
{where $D$ and $g_{(n)}$ are defined in~\eqref{eqn:r(s)} and~\eqref{eqn:def-g_n}.}
\end{proof}

\section{Smoothed Grenander-type estimator}
\label{sec:SG}

The smoothed Grenander-type estimator is defined by
\begin{equation}
\label{def:SG}
\tilde{\lambda}^{SG}_n(t)=\int_{0\vee (t-b)}^{1\wedge (t+b)} k^{(t)}_b(t-u)\,\dd \tilde{\Lambda}_n(u),
\qquad
{\text{for }t\in[0,1],}
\end{equation}
where $\tilde{\Lambda}_n$ is the least concave majorant of $\Lambda_n$.
We are interested in the asymptotic distribution of the $L_p$-error of this estimator:
\begin{equation}
\label{def:InSG}
I_n^{SG}(p)=\int_0^1\left|\tilde{\lambda}^{SG}_n(t)-\lambda(t)\right|^p\,\dd\mu(t).
\end{equation}
{We will compare the behavior of the {$L_p$-error} of $\tilde{\lambda}^{SG}_n$
with that of the regular kernel estimator~$\hat{\lambda}^{s}_n$ from~\eqref{def:corrected_kernel_est}.
Because
\[
\tilde{\lambda}^{SG}_n(t)-\hat{\lambda}^{s}_n(t)
=
\int k^{(t)}_b(t-u)\,\dd(\tilde{\Lambda}_n-\Lambda_n)(u),
\]
we will make use of the behavior of $\tilde{\Lambda}_n-\Lambda_n$,
which has been investigated in~\cite{lopuhaa-mustaSPL2017}, extending similar results from~\cite{durot-tocquet2003} and~\cite{kulikov-lopuhaa2008}.
The idea is to represent $\tilde{\Lambda}_n-\Lambda_n$ in terms of
the mapping~$\CM_{I}$ that maps a function $h:\R\to\R$ into the least concave majorant of~$h$ on the interval $I\subset\R$,
or equivalently by the mapping $\DM_h=\CM_{I}h-h$.}

Let $B_n$ be as in assumption (A2) and $\xi_n$ a $N(0,1)$ distributed r.v. independent of $B_n$. Define versions $W_n$ of Brownian motion by
\begin{equation}
W_n(t)=\begin{cases}
B_n(t)+\xi_nt & \text{ if $B_n$ is a Brownian bridge}\\
B_n(t) & \text{ if $B_n$ is a Brownian motion}.
\end{cases}
\end{equation}
Define
\begin{equation}
\label{def:AnW}
\begin{split}
A_n^E
&=
n^{2/3}\left(\CM_{[0,1]}\Lambda_n-\Lambda_n\right)
=
n^{2/3}\DM_{[0,1]}\Lambda_n,
\\
A_n^W
&=
n^{2/3}\left(\CM_{[0,1]}\Lambda^W_n-\Lambda_n^W\right)
=
n^{2/3}\DM_{[0,1]}\Lambda_n^W.
\end{split}
\end{equation}
where
\begin{equation}
\label{def:LambdaW}
\Lambda_n^W(t)=\Lambda(t)+n^{-1/2}W_n(L(t)),
\end{equation}
with $L$ as in Assumption~(A2).
We start with the following result
{on the $L_p$-distance between~$\tilde{\lambda}^{SG}_n$ and $\hat{\lambda}^s_n$.
In order to use results from~\cite{lopuhaa-mustaSPL2017}, we need that $1\leq p<\min(q,2q-7)$,
where $q$ is from Assumption~(A2).
Moreover, in order to obtain suitable approximations in combination with results from~\cite{lopuhaa-mustaSPL2017},
we require additional conditions on the rate at which $1/b$ tends to infinity.
Also see Remark~\ref{rem:rate 1/b}.
For the optimal rate $b\sim n^{-1/5}$, the result in {Theorem~\ref{theo:clt_SG_kernel}} is valid, as long as $p<5$
and $q>9$.
}

\begin{theo}
\label{theo:clt_SG_kernel}
Assume that $(A1)-{(A2)}$ hold {and let $\mu$ be a finite measure on~$(0,1)$.
Let $k$ satisfy~\eqref{def:kernel} and} let $\tilde{\lambda}^{SG}_n$ and $\hat{\lambda}^s_n$ be defined in~\eqref{def:SG}
and~\eqref{def:corrected_kernel_est}, respectively.
{If $1\leq p<min(q,2q-7)$ and {$nb\to\infty$}, such that
$1/b=o\left(n^{1/3-1/q}\right)$,
$1/b=o\left(n^{(q-3)/(6p)}\right)$,
and
$1/b=o\left(n^{1/6+1/(6p)}(\log n)^{-(1/2+1/(2p))}\right)$},  {then}
\[
n^{2/3}\left(\int_{b}^{1-b} \left|\tilde{\lambda}^{SG}_n(t)-\hat{\lambda}^s_n(t)\right|^p\,\dd\mu(t)\right)^{1/p}\xrightarrow{d}
{\alpha_0}[\DM_{\R}Z](0),
\]
where $Z(t)=W(t)-t^2$, {with $W$ being} a two-sided Brownian motion originating from zero, and
\[
{\alpha_0}=\left(\int_0^1\left|\frac{c'_1(t)}{c_1(t)^2}\right|^p\dd\mu(t)\right)^{1/p},\qquad c_1(t)=\left|\frac{\lambda'(t)}{2L'(t)^2}\right|^{1/3}.
\]
\end{theo}
{
\begin{proof}
{We write}
\[
n^{2/3}\left(\int_{b}^{1-b} \left|\tilde{\lambda}^{SG}_n(t)-\hat{\lambda}^s_n(t)\right|^p\,\dd\mu(t)\right)^{1/p}=b^{-1}\left(\int_{b}^{1-b}\left|Y_n(t)\right|^p\,\dd\mu(t) \right)^{1/p},
\]
where
\begin{equation}
\label{def:Yn}
Y_n(t)=bn^{2/3}\left(\int_{t-b}^{t+b} k_b(t-u)\,\dd(\tilde{\Lambda}_n-\Lambda_n)(u)\right),\qquad t\in(b,1-2b).
\end{equation}
We first show that
\begin{equation}
\label{eqn:step2}
b^{-p}\int_{b}^{1-b}\left|Y_n(t)\right|^p\,\dd\mu(t)\xrightarrow{d}{\alpha_0^p}[D_{\R}Z](0)^p,
\end{equation}
and then the result would follow from the continuous mapping theorem.
Note that {integration by parts yields}
\[
Y_n(t)={\frac1b\int_{-1}^1 k'\left(\frac{t-v}{b}\right)A_n^E(v)\,\dd v.}
\]
The proof consists of several succeeding approximations of $A_n^E$.
{For details, see Lemmas~\ref{lem:Yn-Yn1} to~\ref{lem:Yn4-Yn5}.}
First {we replace $A_n^E$} in the previous integral by $A^W_n$.
The approximation of $Y_n(t)$ by
\begin{equation}
\label{def:tilde_Y}
Y_n^{(1)}(t)
=
\frac1b\int_{-1}^1 k'\left(\frac{t-v}{b}\right)A_n^W(v)\,\dd v.
\end{equation}
where $A_n^W$ is defined in~\eqref{def:AnW}, is possible thanks to Assumption~(A2).
According to~\eqref{eqn:L_p-inequality},
\begin{equation}
\label{eq:bound Yn-YnW}
\begin{split}
&\left|\int_{b}^{1-b} |Y_n(t)|^p\,\dd\mu(t)-\int_{b}^{1-b} |{Y_n^{(1)}}(t)|^p\,\dd\mu(t)\right|\\
&\leq p2^{p-1}\int_{b}^{1-b} |Y_n(t)-{Y_n^{(1)}}(t)|^p\,\dd\mu(t)\\
&\quad+p2^{p-1}\left(\int_{b}^{1-b} |Y_n(t)-{Y_n^{(1)}}(t)|^p\,\dd\mu(t)\right)^{1/p}\left( \int_{b}^{1-b} |{Y_n^{(1)}}(t)|^p\,\dd\mu(t)\right)^{1-1/p}.
\end{split}
\end{equation}
According to Lemma~\ref{lem:Yn-Yn1},
$b^{-p}\int_{b}^{1-b} |Y_n(t)-{Y_n^{(1)}}(t)|^p\,\dd\mu(t)
=
o_P(1)$.
Consequently, in view of~\eqref{eq:bound Yn-YnW}, if we show that
\begin{equation}
\label{eqn:step3}
b^{-p}\int_{b}^{1-b} |{Y_n^{(1)}}(t)|^p\,\dd\mu(t)\xrightarrow{d}{\alpha_0}^p[\DM_{\R}Z](0)^p,
\end{equation}
then we obtain
\begin{equation}
\label{eq:approx Yn}
b^{-p}\int_{b}^{1-b} |Y_n(t)|^p\,\dd\mu(t)
=
b^{-p}\int_{b}^{1-b} |{Y_n^{(1)}}(t)|^p\,\dd\mu(t)
+
{o_P(1)},
\end{equation}
and~\eqref{eqn:step2} follows.

In order to prove~\eqref{eqn:step3}, we replace $A^W_n$ by $n^{2/3}\DM_{I_{nv}}\Lambda^W_n$,
i.e., we approximate~${Y_n^{(1)}}$ by
\begin{equation}
\label{def:Yn2}
{Y_n^{(2)}}(t)=\frac{1}{b}\int_{t-b}^{t+b} k'\left(\frac{t-v}{b}\right)n^{2/3}[\DM_{I_{nv}}\Lambda^W_n](v)\,\dd v.
\end{equation}
where $I_{nv}=[0,1]\cap[v-n^{-1/3}\log n,v+n^{-1/3}\log n]$ and $\Lambda^W$ is defined in~\eqref{def:LambdaW}.
From Lemma~\ref{lem:Yn1-Yn2}, we have that
$b^{-p}\int_{b}^{1-b} |Y_n^{(1)}(t)-{Y_n^{(2)}}(t)|^p\,\dd\mu(t)
=
o_P(1)$.
Hence, similar to the argument that leads to~\eqref{eq:approx Yn},
if we show that
\begin{equation}
\label{eqn:step4}
b^{-p}\int_{b}^{1-b} |{Y_n^{(2)}}(t)|^p\,\dd\mu(t)\xrightarrow{d}{\alpha_0}^p[D_{\R}Z](0)^p,
\end{equation}
then, together with~\eqref{eqn:L_p-inequality}, it follows that
\[
b^{-p}\int_{b}^{1-b} |{Y_n^{(1)}}(t)|^p\,\dd\mu(t)
=
b^{-p}\int_{b}^{1-b} |{Y_n^{(2)}}(t)|^p\,\dd\mu(t)+o_P(1).
\]
Consequently, \eqref{eqn:step3} is equivalent to \eqref{eqn:step4}.

In order to prove~\eqref{eqn:step4},
let
\begin{equation}
\label{def:Ynv}
Y_{nv}(s)=n^{1/6}\left[W_n(L(v+n^{-1/3}s))-W_n(L(v)) \right]+\frac{1}{2}\lambda'(v)s^2.
\end{equation}
{Let $H_{nv}=[-n^{1/3}v,n^{1/3}(1-v)]\cap [-\log n,\log n]$ and $\Delta_{nv}=n^{2/3}[\DM_{I_{nv}}\Lambda^W_n](v)-[\DM_{H_{nv}}Y_{nv}](0)$}.
We approximate ${Y_n^{(2)}}$ by
\begin{equation}
\label{def:Yn3}
{Y_n^{(3)}}(t)=\frac{1}{b}\int_{t-b}^{t+b} k'\left(\frac{t-v}{b}\right)[\DM_{H_{nv}}Y_{nv}](0)\,\dd v.
\end{equation}
From Lemma~\ref{lem:Yn2-Yn3}, we have that
$b^{-p}\int_{b}^{1-b} |Y_n^{(2)}(t)-{Y_n^{(3)}}(t)|^p\,\dd\mu(t)
=
o_P(1)$.
Again, similar to the argument that leads to~\eqref{eq:approx Yn},
if we show that
\begin{equation}
\label{eqn:step5}
b^{-p}\int_{b}^{1-b}
|{Y_n^{(3)}}(t)|^p\,\dd\mu(t)\xrightarrow{d}{\alpha_0}^p[\DM_{\R}Z](0)^p.
\end{equation}
then, together with~\eqref{eqn:L_p-inequality},
it follows that
\[
b^{-p}\int_{b}^{1-b}  |{Y_n^{(2)}}(t)|^p\,\dd\mu(t)
=
b^{-p}\int_{b}^{1-b}  |{Y_n^{(3)}}(t)|^p\,\dd\mu(t)+o_P(1),
\]
which would prove~\eqref{eqn:step4}.

We {proceed with} proving~\eqref{eqn:step5}.
Let $W$ be a two sided Brownian motion originating from zero.
We have that
\[
n^{1/6}\left[W_n(L(v+n^{-1/3}s))-W_n(L(v))\right]\overset{d}{=}W\left(n^{1/3}(L(v+n^{-1/3}s)-L(v))\right)
\]
as a process in $s$.
Consequently,
\[
{Y_n^{(3)}(t)}
\stackrel{d}{=}
\frac{1}{b}\int_{t-b}^{t+b} k'\left(\frac{t-v}{b}\right)[\DM_{H_{nv}}\tilde{Y}_{nv}](0)\,\dd v
\]
where
\begin{equation}
\label{def:tilde Ynv}
\tilde{Y}_{nv}(s)=W(n^{1/3}(L(v+n^{-1/3}s)-L(v)))+\frac{1}{2}\lambda'(v)s^2.
\end{equation}
Now define
\begin{equation}
\label{def:Znv}
Z_{nv}(s)=W({L'(v)}s)+\frac{1}{2}\lambda'(v)s^2.
\end{equation}
and
$J_{nv} = [n^{1/3} (L(a_{nv}) - L(v)) /L'(v),
	n^{1/3} (L(b_{nv} ) - L(v)) /L'(v)]$,
where $a_{nv} = \max(0, v - n^{-1/3} \log n)$
and $b_{nv} = \min(1, v + n^{-1/3} \log n)$.
We approximate $\tilde{Y}_{nv}$ by $Z_{nv}$, i.e., we approximate $Y_n^{(3)}$ by
\begin{equation}
\label{def:Yn4}
{Y_n^{(4)}}(t)
=
\frac{1}{b}\int_{t-b}^{t+b} k'\left(\frac{t-v}{b}\right)[\DM_{J_{nv}}Z_{nv}](0)\,\dd v,
\end{equation}
Lemma~\ref{lem:Yn3-Yn4} yields
$b^{-p}\int_{b}^{1-b} |Y_n^{(3)}(t)-{Y_n^{(4)}}(t)|^p\,\dd\mu(t)
=
o_P(1)$.
Once more, similar to the argument that leads to~\eqref{eq:approx Yn},
if we show that
\begin{equation}
\label{eqn:step6}
b^{-p}\int_{b}^{1-b}
|{Y_n^{(4)}}(t)|^p\,\dd\mu(t)\xrightarrow{d}{\alpha_0}^p[D_{\R}Z](0)^p,
\end{equation}
{then}, together with~\eqref{eqn:L_p-inequality}, it follows that
\[
b^{-p}\int_{b}^{1-b}
|{Y_n^{(3)}}(t)|^p\,\dd\mu(t)
=
b^{-p}\int_{b}^{1-b}  |{Y_n^{(4)}}(t)|^p\,\dd\mu(t)+o_P(1),
\]
and as a result, also \eqref{eqn:step5} holds.

As a final step, we prove \eqref{eqn:step6}.
Since
$c_1(v)W\left({L'(v)}c_2(v)s\right)\overset{d}{=}W(s)$ as a process in $s$,
where
\[
c_1(v)=\left(\frac{|\lambda'(v)|}{2{L'(v)}^2}\right)^{1/3},\qquad c_2(v)=\left(\frac{4{L'(v)}}{|\lambda'(v)|^2}\right)^{1/3}
\]
we obtain that
\[
{Y_n^{(4)}(t)}
\overset{d}{=}
\frac{1}{b}\int_{t-b}^{t+b} k'\left(\frac{t-v}{b}\right)\frac{1}{c_1(v)}[\DM_{I_{nv}}Z](0)\,\dd v
\]
where $I_{nv}=c_2(v)^{-1}J_{nv}$ and $Z(t)=W(t)-t^2$.
{We approximate $\DM_{I_{nv}}$ by $\DM_{\R}$, i.e., we approximate $Y_n^{(4)}$ by}
\begin{equation}
\label{def:Yn5}
{Y_n^{(5)}}(t)
=
[\DM_{\R}Z](0)\frac{1}{b}\int_{t-b}^{t+b} k'\left(\frac{t-v}{b}\right)\frac{1}{c_1(v)}\,\dd v.
\end{equation}
It remains to show that
\begin{equation}
\label{eqn:step7}
b^{-p}\int_{b}^{1-b}
|{Y_n^{(5)}}(t)|^p\,\dd\mu(t)\xrightarrow{d}{\alpha_0^p}[\DM_{\R}Z](0)^p,
\end{equation}
because then, it follows that
\[
b^{-p}\int_{b}^{1-b}
|{Y_n^{(4)}}(t)|^p\,\dd\mu(t)
=
b^{-p}\int_{b}^{1-b} |{Y_n^{(5)}}(t)|^p\,\dd\mu(t)+o_P(1)
\]
so that~\eqref{eqn:step6} holds.
Since
\[
\frac{1}{b}\int_{t-b}^{t+b} k'\left(\frac{t-v}{b}\right)\frac{1}{c_1(t)}\,\dd v=\frac{1}{c_1(t)}\int_{-1}^{1} k'\left(y\right)\,\dd y=0.
\]
we can write
\[
\begin{split}
\frac{1}{b}\int_{t-b}^{t+b} k'\left(\frac{t-v}{b}\right)\frac{1}{c_1(v)}\,\dd v
&=\frac{1}{b}\int_{t-b}^{t+b} k'\left(\frac{t-v}{b}\right)\left(\frac{1}{c_1(v)}-\frac{1}{c_1(t)}\right)\,\dd v \\
&=\int_{-1}^{1} k'\left(y\right)\left(\frac{1}{c_1(t-by)}-\frac{1}{c_1(t)}\right)\,\dd y.
\end{split}
\]
Assumptions (A1) and (A2) imply that {$t\mapsto c_1(t)$} is strictly positive and differentiable with bounded derivative, so by a Taylor expansion we get
\[
\int_{-1}^{1} k'\left(y\right)\left(\frac{1}{c_1(t-by)}-\frac{1}{c_1(t)}\right)\,\dd y
=
\frac{c'_1(t)}{c_1(t)^2}b\int_{-1}^{1} k'\left(y\right)y\,\dd y+O(b^2).
\]
Hence,
\[
\begin{split}
b^{-p}\int_{b}^{1-b}|{Y_n^{(5)}}(t)|^p\,\dd\mu(t)
&=
[\DM_{\R}Z](0)^pb^{-p}\int_{b}^{1-b}\left|\frac{{c'_1(t)b}}{c_1(t)^2}  \right|^p\,\dd\mu(t)+o_P(1)\\
&=
{[\DM_{\R}Z](0)^p}\int_{0}^{1}\left|\frac{c'_1(t)}{c_1(t)^2}  \right|^p\,\dd\mu(t)+o_P(1)
\end{split}
\]
which concludes the proof of~\eqref{eqn:step7} and finishes the proof of the theorem.
\end{proof}
}
\begin{re}
\label{rem:rate 1/b}
{Note that the assumption $1/b=o\left(n^{1/6+1/(6p)}{(\log n)^{-(1+1/p)}}\right)$ of the previous theorem puts a restriction on $p$, when $b$ has the optimal rate $n^{-1/5}$.
This is due to the approximation of $Y_n^{(4)}(t)$ by $Y_n^{(5)}(t)$ for $t\in(b,1-b)$.
This restriction on $p$ can be avoided if we consider the $L_p$-error on
the smaller interval $(b+n^{-1/3}\log n,1-b-n^{-1/3}\log n)$.}
\end{re}
\begin{re}
\label{re:boundaries}
For $p>1$, the boundary regions cannot be included in the CLT of Theorem~\ref{theo:clt_SG_kernel}.
{For example, for $t\in(0,b)$, it can be shown that there exists a universal constant $K>0$, such that
\[
n^{2p/3}\int_0^b\left|\tilde{\lambda}_n^{SG}(t)-\tilde{\lambda}_n^s(t)\right|^p\,\dd \mu(t)
>
Kb^{-p+1}[\DM_{\R}Z](0)^p+o_P(b^{-p+1}),
\]
which is not bounded in probability for $p>1$.
{For details see the supplemental material~\cite{LM_L_p_2018}.}
The same result also holds for $t\in(1-b,1)$.

In the special case $p=1$, for $t\in(0,b)$ we have
\[
n^{2/3}
\int_0^b\left|\tilde{\lambda}_n^{SG}(t)-\tilde{\lambda}_n^s(t)\right|\,\dd \mu(t)
=
[\DM_{\R}Z](0)\frac{1}{b}\int_0^b\left|\frac{1}{c_1(t)}\int_{-1}^{t/b} \frac{\dd }{\dd y}k^{(t)}\left(y\right)\,\dd y \right|\,\dd \mu(t)+o_P(1).
\]
If (A3) holds, then
\[
\frac{1}{b}\int_0^b\left|\frac{1}{c_1(t)}\int_{-1}^{t/b} \frac{\dd }{\dd y}k^{(t)}\left(y\right)\,\dd y \right|\,\dd \mu(t)
\to
\frac{w(0)}{c_1(0)}\int_0^1\left|\psi_1\left(y\right)k\left(y\right)+\psi_2\left(y\right)yk\left(y\right)\right|\,\dd y.
\]
Similarly, we can deal with the case $t\in(1-b,1)$.
It follows that
\[
n^{2/3}
\int_0^1\left|\tilde{\lambda}_n^{SG}(t)-\tilde{\lambda}_n^s(t)\right|\,\dd \mu(t)
\xrightarrow{d}
\tilde{\alpha}_0[\DM_{\R}Z](0)
\]
with
\[
\tilde{\alpha}_0
=
\alpha_0+\left(\frac{w(0)}{c_1(0)}+\frac{w(1)}{c_1(1)}\right)\int_0^1\left|\psi_1\left(y\right)k\left(y\right)+\psi_2\left(y\right)yk\left(y\right)\right|\,\dd y.
\]}
\end{re}
{We are now ready to formulate the CLT for the smoothed Grenander-type estimator.
The result will follow from combining Corollary~\ref{cor:as.distribution_J_n} with Theorem~\ref{theo:clt_SG_kernel}.
Because we now deal with the $L_p$-error between $\tilde{\lambda}^{SG}_n$ and $\lambda$,
the contribution of the integrals over the boundary regions~$(0,2b)$ and $(1-2b,1)$ can be shown to be negligible.
This means we no longer need the third requirement in Theorem~\ref{theo:clt_SG_kernel}
on the rate of $1/b$.}
\begin{theo}
\label{theo:as.distribution_I_n}
Assume that $(A1)-(A3)$ hold {and let $k$ satisfy~\eqref{def:kernel}.
Let $I_n^{SG}$ be defined in~\eqref{def:InSG}.
If $1\leq p<min(q,2q-7)$ and {$nb\to\infty$}, such that
	$1/b=o\left(n^{1/3-1/q}\right)$ and $1/b=o\left(n^{(q-3)/(6p)}\right)$.}
\begin{enumerate}[label={\roman*)}]
\item If $nb^5\to 0$, {then}
\[
(b\sigma^2(p))^{-1/2}\left\{(nb)^{p/2} I_n^{SG}(p)-{m}_n(p) \right\}\xrightarrow{d} N(0,1);
\]
\item
If $nb^5\to C_0^2>0$, and $B_n$ in assumption (A2) is a Brownian motion, {then}
\[
(b\theta^2(p))^{-1/2}\left\{(nb)^{p/2} I_n^{SG}(p)-{m}_n(p) \right\}\xrightarrow{d} N(0,1);
\]
\item
If $nb^5\to C_0^2>0$, and $B_n$ in assumption (A2) is a Brownian bridge, {then}
\[
(b\tilde\theta^2(p))^{-1/2}\left\{(nb)^{p/2} I_n^{SG}(p)-{m}_n(p) \right\}\xrightarrow{d} N(0,1),
\]
\end{enumerate}
{where $I_n^{SG}$, ${m_n}$, $\sigma^2$, $\theta^2$, and $\tilde\theta^2$ are defined in
\eqref{def:InSG}, \eqref{eqn:def-m_n}, \eqref{eqn:def-sigma}, \eqref{eqn:def-theta}, and \eqref{eqn:def-theta-tilde}, respectively.}
\end{theo}
\begin{proof}
{Define
\begin{equation}
\label{eqn:def-gamma}
\gamma^2(p)=\begin{cases}
\sigma^2(p)  & \text{if}\quad nb^5\to 0\\
\theta^2(p)  & \text{if}\quad nb^5\to C_0^2.
\end{cases}
\end{equation}}
By Corollary~\ref{cor:as.distribution_J_n}, we already  have that
\[
(b\gamma^2(p))^{-1/2}\left\{(nb)^{p/2} \int_0^1 \left|\hat{\lambda}^{s}_n(t)-\lambda(t)\right|^p\,\dd\mu(t)-{m}_n(p) \right\}\xrightarrow{d} N(0,1),
\]
{for $\hat{\lambda}^{s}_n$ defined in~\eqref{def:corrected_kernel_est}.}
Hence it is sufficient to show that
\[
b^{-1/2}(nb)^{p/2}\left|\int_0^1 \left|\tilde{\lambda}^{SG}_n(t)-\lambda(t)\right|^p\,\dd\mu(t)-\int_0^1 \left|\hat{\lambda}^{s}_n(t)-\lambda(t)\right|^p\,\dd\mu(t) \right|\xrightarrow{\p} 0,
\]
{in all three cases (i)-(iii).}
First we show that
\begin{equation}
\label{eqn:difference_boundaries}
b^{-1/2}(nb)^{p/2}\left|\int_0^{2b} \left|\tilde{\lambda}^{SG}_n(t)-\lambda(t)\right|^p\,\dd\mu(t)-\int_0^{2b} \left|\hat{\lambda}^{s}_n(t)-\lambda(t)\right|^p\,\dd\mu(t) \right|\xrightarrow{\p} 0.
\end{equation}
Indeed, by~\eqref{eqn:L_p-inequality}, we get
\begin{equation}
\begin{split}
&\left|\int_0^{2b} \left|\tilde{\lambda}^{SG}_n(t)-\lambda(t)\right|^p\,\dd\mu(t)-\int_0^{2b} \left|\hat{\lambda}^{s}_n(t)-\lambda(t)\right|^p\,\dd\mu(t) \right|\\
&\leq p2^{p-1}\int_0^{2b} \left|\tilde{\lambda}^{SG}_n(t)-\hat{\lambda}^s_n(t)\right|^p\,\dd\mu(t)\\
&\quad+p2^{p-1}\left(\int_0^{2b} \left|\tilde{\lambda}^{SG}_n(t)-\hat{\lambda}^s_n(t)\right|^p\,\dd\mu(t)\right)^{1/p}\left(\int_0^{2b} \left|\hat{\lambda}^{s}_n(t)-\lambda(t)\right|^p\,\dd\mu(t)  \right)^{1-1/p}.
\end{split}
\end{equation}
Moreover, by integration by parts and the Kiefer-Wolfowitz type of result in Corollary 3.1 in~\cite{DL14}, it follows that
\begin{equation}
\label{eqn:KW-bound}
\begin{split}
\sup_{t\in[0,1]}\left|\tilde{\lambda}^{SG}_n(t)-\hat{\lambda}^{s}_n(t)\right|&=\sup_{t\in[0,1]}\left|\int k^{(t)}_b(t-u)\,\dd(\tilde{\Lambda}_n-\Lambda_n)(u)\right|\\
&\leq
Cb^{-1}\sup_{t\in[0,1]}|\tilde{\Lambda}_n(t)-\Lambda_n(t)|
=
O_P\left(b^{-1}\left(\frac{\log n}{n} \right)^{2/3}\right).
\end{split}
\end{equation}
Hence
\begin{equation}
\label{eqn:boundary_SG_kernel}
\int_0^{2b} \left|\tilde{\lambda}^{SG}_n(t)-\hat{\lambda}^s_n(t)\right|^p\,\dd\mu(t)=O_P\left(b^{1-p}\left(\frac{\log n}{n} \right)^{2p/3}\right).
\end{equation}
Together with Proposition~\ref{prop:clt_boundaries} this implies~\eqref{eqn:difference_boundaries}. Similarly, we also have
\[
b^{-1/2}(nb)^{p/2}\left|\int_{1-2b}^1 \left|\tilde{\lambda}^{SG}_n(t)-\lambda(t)\right|^p\,\dd\mu(t)-\int_{1-2b}^1 \left|\hat{\lambda}^{s}_n(t)-\lambda(t)\right|^p\,\dd\mu(t) \right|\xrightarrow{\p} 0.
\]
Thus, it remains to prove
\begin{equation}
\label{eqn:difference_integrals_1}
b^{-1/2}(nb)^{p/2}\left|\int_{2b}^{1-2b} \left|\tilde{\lambda}^{SG}_n(t)-\lambda(t)\right|^p\,\dd\mu(t)-\int_{2b}^{1-2b} \left|\hat{\lambda}^{s}_n(t)-\lambda(t)\right|^p\,\dd\mu(t) \right|\xrightarrow{\p} 0.
\end{equation}
Again, from~\eqref{eqn:L_p-inequality}, we have
\begin{equation}
\label{eqn:difference_integrals_2}
\begin{split}
&\left|\int_{2b}^{1-2b} \left|\tilde{\lambda}^{SG}_n(t)-\lambda(t)\right|^p\,\dd\mu(t)-\int_{2b}^{1-2b} \left|\hat{\lambda}^{s}_n(t)-\lambda(t)\right|^p\,\dd\mu(t) \right|\\
&\leq p2^{p-1}\int_{2b}^{1-2b} \left|\tilde{\lambda}^{SG}_n(t)-\hat{\lambda}^s_n(t)\right|^p\,\dd\mu(t)\\
&\quad+p2^{p-1}\left(\int_{2b}^{1-2b} \left|\tilde{\lambda}^{SG}_n(t)-\hat{\lambda}^s_n(t)\right|^p\,\dd\mu(t)\right)^{1/p}\left(\int_{2b}^{1-2b} \left|\hat{\lambda}^{s}_n(t)-\lambda(t)\right|^p\,\dd\mu(t)  \right)^{1-1/p}.
\end{split}
\end{equation}
{Because $b^{-1}=o(n^{1/3-1/q})$ implies that $(2b,1-2b)\subset (b+n^{-1/3}\log n,1-b-n^{-1/3}\log n)$,
from Theorem~\ref{theo:clt_SG_kernel}, in particular Remark~\ref{rem:rate 1/b},} we have
\begin{equation}
\label{eqn:difference_SG_kernel}
\int_{2b}^{1-2b}\left|\tilde{\lambda}^{SG}_n(t)-\hat{\lambda}^s_n(t)\right|^p\,\dd\mu(t)
=
{O_P}(n^{-2p/3})
=
o_P(n^{-p/2}).
\end{equation}
Then, ~\eqref{eqn:difference_integrals_1} follows immediately
from~\eqref{eqn:difference_integrals_2} and the fact that,
{according to Theorem~\ref{theo:as.distribution_J_n},}
\[
\int_{2b}^{1-2b} \left|\hat{\lambda}^{s}_n(t)-\lambda(t)\right|^p\,\dd\mu(t)=O_P((nb)^{-p/2}).
\]
{This proves the theorem.}
\end{proof}
\begin{re}
Note that, {if $b=cn^{-\alpha}$, for some $0<\alpha<1$}, the proof is simple and short in case $\alpha<p/(3(1+p))$ because the Kiefer-Wolfowitz type of result in Corollary 3.1 in~\cite{DL14} is sufficient to prove~\eqref{eqn:difference_SG_kernel}.
Indeed, from~\eqref{eqn:KW-bound}, it follows that
\[
\int_{2b}^{1-2b} \left|\tilde{\lambda}^{SG}_n(t)-\hat{\lambda}^s_n(t)\right|^p\,\dd\mu(t)=O_P\left(b^{-p}\left(\frac{\log n}{n} \right)^{2p/3}\right)
=
o_P\left({b^{1/2}\left(nb\right)^{-p/2}}\right).
\]
However, this assumption on $\alpha$ is quite restrictive because for example if $\alpha=1/5$ then the theorem holds only for $p>3/2$ (not for the $L_1$-{loss}) and if $\alpha=1/4$ then the theorem holds only for $p>3$.
\end{re}

\section{Isotonized kernel estimator}
\label{sec:GS}
The {isotonized kernel} estimator is defined as follows.
First, we smooth the piecewise constant estimator $\Lambda_n$ by means of a boundary corrected kernel function, i.e.,
let
\begin{equation}
\label{eqn:smoothed_cumulative}
\Lambda^s_n(t)=\int_{(t-b)\vee 0}^{(t+b)\wedge 1 } k^{(t)}_b(t-u)\Lambda_n(u)\,\dd u,
{\quad\text{for }t\in[0,1],}
\end{equation}
where $k^{(t)}_b(u)$ defined as in~\eqref{eqn:boundary_kernel}.
Next, we define a continuous monotone estimator $\tilde{\lambda}^{GS}_n$ of $\lambda$  as the left-hand slope of the least concave majorant
$\widehat{\Lambda}^s_n$ of $\Lambda^s_n$ on $[0,1]$.
{In this way we define a sort of Grenander estimator based on a smoothed naive estimator for $\Lambda$.
For this reason we use the superscript $GS$.}

We are interested in the asymptotic distribution of the $L_p$-error of this estimator:
\[
I_n^{GS}(p)=\int_0^1\left|\tilde{\lambda}^{GS}_n(t)-\lambda(t)\right|^p\,\dd\mu(t).
\]
It follows from Lemma 1 in~\cite{GJ10} (in the case of a decreasing function), that $\tilde{\lambda}_n^{GS}$ is continuous and is the unique minimizer of
\[
\psi(\lambda)=\frac{1}{2}
\int_0^{1}
\left(
\lambda(t)-\tilde{\lambda}_n^s(t)
\right)^2\,\mathrm{d}t
\]
over all nonincreasing functions $\lambda$, where
$\tilde{\lambda}_n^s(t)=\mathrm{d}\Lambda_n^s(t)/\mathrm{d}t$.
This {suggests $\tilde{\lambda}_n^s(t)$} as a naive estimator for $\lambda_0(t)$.
Note that, for $t\in[b,1-b]$, {from} integration by parts we get
\begin{equation}
\label{eqn:naive_kernel}
\tilde{\lambda}_n^s(t)=\frac{1}{b^2}\int_{t-b}^{t+b} k'\left(\frac{t-u}{b}\right)\Lambda_n(u)\,\mathrm{d}u=\int_{t-b}^{t+b} k_b(t-u)\,\mathrm{d}\Lambda_n(u),
\end{equation}
i.e., $\tilde{\lambda}_n^s$ coincides with the usual kernel estimator of $\lambda$ on the interval $[b,1-b]$.

{Let $0<\gamma<1$.
It can be shown that
{\begin{equation}
\label{eq:kernel=GS}
\p
(
\tilde{\lambda}_n^{s}(t)=\tilde{\lambda}_n^{GS}(t)\text{ for all } t\in[b^{\gamma},1-b^{\gamma}]
)\to 1.
\end{equation}}%
See Corollary~\ref{cor:asymptotic_equivalence} in {the supplemental material~\cite{LM_L_p_2018}.}
Hence, their $L_p$-error between $\tilde{\lambda}^{GS}_n$ and $\tilde{\lambda}_n^s$} will exhibit the same behavior in the limit.
Note that this holds for every $\gamma<1$, which means that the interval we are considering is approaching $(b,1-b)$.
Consider a modified $L_p$-error of the {isotonized kernel} estimator defined by
\begin{equation}
\label{def:InGSc}
I^{GS,c}_{n,\gamma}(p)=\int_{b^{\gamma}}^{1-b^{\gamma}}\left|\tilde{\lambda}_n^{GS}(t)-\lambda(t)\right|^p\,\dd \mu(t).
\end{equation}
{We then have the following result.}
\begin{theo}
Assume that (A1)-(A3) {hold
	and let $I^{GS,c}_{n,\gamma}(p)$ be defined in~\eqref{def:InGSc}.
Let $k$ satisfy~\eqref{def:kernel} and let} $L$ be as in Assumption (A2).
Assume $b\to 0$ and $1/b=o(n^{1/4})$ {and let $1/2<\gamma<1$.}
\begin{enumerate}[label={\roman*)}]
\item If $nb^5\to 0$, {then}
\[
(b\sigma^2(p))^{-1/2}\left\{(nb)^{p/2} I^{GS,c}_{n,\gamma}(p)-m_{n}(p) \right\}\xrightarrow{d} N(0,1);
\]
\item If $nb^5\to C_0^2>0$ and $B_n$ in assumption (A2) is a Brownian motion, {then}
\[
(b\theta^2(p))^{-1/2}\left\{(nb)^{p/2} I^{GS,c}_{n,\gamma}(p)-m_{n}(p) \right\}\xrightarrow{d} N(0,1);
\]
\item If $nb^5\to C_0^2>0$ and $B_n$ in assumption (A2) is a Brownian bridge, {then}
\[
(b\tilde\theta^2(p))^{-1/2}\left\{(nb)^{p/2} I^{GS,c}_{n,\gamma}(p)-m_{n}(p) \right\}\xrightarrow{d} N(0,1),
\]
\end{enumerate}
where $\sigma^2$, $\theta^2$, $\tilde\theta^2$ and $m_n$ are defined respectively in \eqref{eqn:def-sigma}, \eqref{eqn:def-theta}, \eqref{eqn:def-theta-tilde} and \eqref{eqn:def-m_n}.
\end{theo}
\begin{proof}
It follows from Theorem~\ref{theo:as.distribution_J_n} {and~\eqref{eq:kernel=GS}.} 
Note that the results of Theorem~\ref{theo:as.distribution_J_n} do not change if we consider the interval $[b^\gamma,1-b^{\gamma}]$ instead of $[b,1-b]$ and that $b^{-1/2}|m_n^c(p)-m_n(p)|\to 0.$
\end{proof}

\section{Hellinger error}
\label{sec:hellinger}
In this section we {investigate} the global behavior of estimators by means of a weighted Hellinger distance
\begin{equation}
\label{def:Hellinger}
H(\hat{\lambda}_n,\lambda)=\left(\frac{1}{2}\int_0^1\left(\sqrt{\hat{\lambda}_n(t)}-\sqrt{\lambda(t)}\right)^2\,\dd\mu(t) \right)^{1/2},
\end{equation}
where $\hat{\lambda}_n$ is the estimator at hand.
{This metric is convenient in maximum likelihood problems, which goes back to~\cite{lecam1970, lecam1973, birgemassart1993}. Consistency in Hellinger distance of shape constrained maximum likelihood estimators has been investigated in~\cite{palwoodroofe2007}, \cite{sereginwellner2010}, and~\cite{dosswellner2016}, whereas rates on Hellinger risk measures have been obtained in~\cite{sereginwellner2010},
\cite{kimsamworth2016}, and \cite{kimguntuboyinasamworth2016}. The first central limit theorem type of result for the Hellinger distance was presented in \cite{LM_Hellinger} for Grenander type estimators of a monotone function. We deal with the smooth (isotonic) estimators following the same approach.}

Note that, for the Hellinger distance to be well defined we need to assume that $\lambda$ takes only positive values.
We follow the same line of argument as in~\cite{LM_Hellinger}.
{We first establish that
\[
\int_0^1\left(\sqrt{\hat{\lambda}_n^s(t)}-\sqrt{\lambda(t)}\right)^2\,\dd\mu(t)=\int_0^1\left(\hat{\lambda}_n^s(t)-\lambda(t)\right)^2(4\lambda(t))^{-1}\,\dd\mu(t) +O_P\left((nb)^{-3/2}\right),
\]
which shows that the squared Hellinger loss can be approximated by a weighted squared $L_2$-distance.
For details, see} Lemma~\ref{lemma:Hellinger_approximation} in {the supplemental material~\cite{LM_L_p_2018}},
which is the {corresponding} version of Lemma~2.1 in~\cite{LM_Hellinger}.
Hence, a central limit theorem for squared the Hellinger loss follows directly from the central limit theorem for the weighted $L_2$-distance
(see Theorem~\ref{theo:clt_Hellinger_squared} in {the supplemental material~\cite{LM_L_p_2018}}, which corresponds to Theorem~3.1 in~\cite{LM_Hellinger}).
An application of the delta method {will} then lead to the following result.
\begin{theo}
\label{theo:clt_Hellinger}
Assume (A1)-(A3) hold.
{Let $\tilde\lambda_n^s$ be defined in~\eqref{def:kernel_est}, with $k$ satisfying~\eqref{def:kernel},
and let $H$ be defined in~\eqref{def:Hellinger}.}
{Suppose that $nb\to\infty$} and that $\lambda$ is strictly positive.
\begin{enumerate}[label={\roman*)}]
\item If $nb^5\to0$, {then}
\[
\left(b\frac{{\tau^2(2)}}{8{\mu_n(2)}}\right)^{-1/2}\left\{(nb)^{1/2}H(\hat{\lambda}_n^s,\lambda)-2^{-1/2}{\mu_n(2)^{1/2}} \right\}\xrightarrow{d} N(0,1).
\]
\item If $nb^5\to C_0^2>0$ and $B_n$ in Assumption (A2) is a Brownian motion, {then}
\[
\left(b\frac{{\kappa^2}(2)}{8{\mu_n(2)}}\right)^{-1/2}\left\{(nb)^{1/2}H(\hat{\lambda}_n^s,\lambda)-2^{-1/2}{\mu_n(2)^{1/2}} \right\}\xrightarrow{d} N(0,1),
\]
\item If $nb^5\to C_0^2>0$ and $B_n$ in Assumption (A2) is a Brownian bridge, {then}
\[
\left(b\frac{{\tilde\kappa^2}(2)}{8{\mu_n(2)}}\right)^{-1/2}\left\{(nb)^{1/2}H(\hat{\lambda}_n^s,\lambda)-2^{-1/2}{\mu_n(2)^{1/2}} \right\}\xrightarrow{d} N(0,1),
\]
\end{enumerate}
where ${\tau^2}$, ${\kappa^2}$, ${\tilde\kappa^2}$ and ${\mu_n}$ are defined
as in \eqref{eqn:def-sigma}, \eqref{eqn:def-theta}, \eqref{eqn:def-theta-tilde} and \eqref{eqn:def-m_n}, respectively,
by replacing~$w(t)$ with $w(t)(4\lambda(t))^{-1}$.
{\begin{enumerate}
\item[(iv)]
Under the conditions of Theorem~\ref{theo:as.distribution_I_n}, results (i)-(iii) also hold
when replacing $\hat{\lambda}_n^s$ by the smoothed Grenander-type estimator $\tilde{\lambda}_n^{SG}$,
defined in~\eqref{def:SG}.
\end{enumerate}}
\end{theo}
{\begin{proof}
The proof consists of an application of the delta-method in combination with Theorem~\ref{theo:clt_Hellinger_squared} in {the supplemental material~\cite{LM_L_p_2018}}.
According to part~(i) of Theorem~\ref{theo:clt_Hellinger_squared} in~\cite{LM_L_p_2018},
\[
b^{-1/2}\left(2nbH(\hat{\lambda}_n^s,\lambda)-\mu_n(2)\right)
\xrightarrow{d} Z
\]
where $Z$ is a mean zero normal random variable with variance $\tau^2(2)$.
Therefore, in order to obtain part~(i) of Theorem~\ref{theo:clt_Hellinger}, we apply the delta method with the mapping $\phi(x)=2^{-1/2}x^{1/2}$.
Parts~(ii)-(iv) are obtained in the same way.
\end{proof}
To be complete, note that from Corollary~\ref{cor:asymptotic_equivalence},
the previous central limit theorems also hold for the isotonized kernel estimator $\tilde{\lambda}_n^{GS}$, defined in Section~\ref{sec:GS},
when considering a Hellinger distance corresponding to the interval $(b^\gamma,1-b^\gamma)$ instead of $(0,1)$ in~\eqref{def:Hellinger}.}

\section{Testing}
\label{sec:testing}
In this section we investigate a possible application of {the results obtained in Section~\ref{sec:SG}} for testing monotonicity. 
{For example, Theorem~\ref{theo:as.distribution_I_n} could be used to construct a test for the single null hypothesis $H_0:\,\lambda=\lambda_0$, 
for some known monotone function $\lambda_0$.
Instead, we investigate a nonparametric test for monotonicity on the basis of the $L_p$-distance between the smoothed Grenander-type estimator and the kernel estimator,
see Theorem~\ref{theo:clt_SG_kernel}.}

The problem of testing a nonparametric null hypothesis of monotonicity has gained a lot of interest in the literature (see for example \cite{LK_2004} for the density setting, \cite{HK05}, \cite{GJ2012} for the hazard rate, \cite{ABD14}, \cite{BD07}, \cite{BN2013},\cite{GHJK2000} for the regression function).

We consider a regression model with deterministic design points
\begin{equation}
\label{def:reg_model}
Y_i=\lambda\left(\frac{i}{n}\right) +\epsilon_{i}, \qquad i\in\{1,\dots,n\},
\end{equation}
where the $\epsilon_i$'s are independent normal random variables with mean zero and variance $\sigma^2$. 
Such a model satisfies Assumption~(A2) with $q=+\infty$ and
{$\Lambda_n(t)=n^{-1}\sum_{i\leq nt}Y_i$, for $t\in[0,1]$} (see Theorem~5 in~\cite{durot2007}).

Assume we have a sample of $n$ obseravtions $Y_1,\dots,Y_n$.  Let $\mathcal{D}$ be the space of decreasing functions on $[0,1]$. 
We want to test $H_0\colon \lambda\in\mathcal{D}$ against $H_1\colon \lambda\notin\mathcal{D}$.
Under the null hypothesis we can estimate $\lambda$ by the smoothed Grenander-type estimator $\tilde{\lambda}_n^{SG}$ defined as in \eqref{def:SG}. On the other hand, under the alternative hypothesis we can estimate $\lambda$ by the kernel estimator with boundary corrections $\hat{\lambda}^s_n$ defined in \eqref{def:corrected_kernel_est}. Then, as a test statistics we take
\[
T_n=n^{2/3}\left(\int_{b}^{1-b} \left|\tilde{\lambda}^{SG}_n(t)-\hat{\lambda}^s_n(t)\right|^2\,\dd t\right)^{1/2},
\]
and at level $\alpha$, we reject the null hypothesis if $T_n>c_{n,\alpha}$ for some critical value $c_{n,\alpha}>0.$ 

In order to use the asymptotic quantiles of the limit distribution in Theorem \ref{theo:clt_SG_kernel}, we need to estimate the constant $C_0$ which depends on the derivatives of $\lambda$. 
To avoid this, we choose to determine the critical value by a bootstrap procedure. 
We generate $B=1000$ samples of size $n$ from the model \eqref{def:reg_model} with $\lambda$ replaced by its estimator $\tilde{\lambda}_n^{SG}$ under the null hypothesis. 
For each of these samples we compute the estimators $\tilde{\lambda}_n^{SG,*}$, $\hat{\lambda}_n^{s,*}$ and  the test statistics
\[
T_{n,j}^*=n^{2/3}\left(\int_{b}^{1-b} \left|\tilde{\lambda}^{SG,*}_n(t)-\hat{\lambda}^{s,*}_n(t)\right|^2\,\dd t\right)^{1/2},\qquad j=1,\dots,B.
\]
Then we take as a {critical value,} the $100\alpha$-th upper-percentile of the values $T_{n,1}^*,\dots,T_{n,B}^*$. We repeat this procedure $N=1000$ times and we count the percentage of rejections. This gives an approximation of the level (or the power) of the test if we start with a sample for which the true $\lambda$ is decreasing (or non-decreasing).

We investigate the performance of the test by comparing it to tests proposed in \cite{ABD14}, \cite{BHL05} and in \cite{GHJK2000}. For a power comparison,  \cite{ABD14} and \cite{BHL05} consider the following functions
\[
\begin{aligned}
\lambda_1(x)&=-15(x-0.5)^3\1_{\{x\leq 0.5\}}-0.3(x-0.5)+\exp\left(-250(x-0.25)^2\right),\\
\lambda_2(x)&=16\sigma x, \quad
\lambda_3(x)=0.2\exp\left(-50(x-0.5)^2\right), \quad
\lambda_4(x)=-0.1\cos(6\pi x),\\
\lambda_5(x)&=-0.2x+\lambda_3(x),\quad
\lambda_6(x)=-0.2x+\lambda_4(x),\\
\lambda_7(x)&=-(1+x)+0.45\exp\left(-50(x-0.5)^2\right),
\end{aligned}
\]
We denote by $T_B$ the local mean test of~\cite{BHL05} and~$S^{reg}_n$ the test proposed in~\cite{ABD14} on the basis of the distance between the least concave majorant of $\Lambda_n$ and $\Lambda_n$.  The result of the simulations for $n=100$, $\alpha=0.05$, {$b=0.1$,} are given in Table~\ref{tab:1}.
\begin{table}[h]
\begin{tabular}{cccccccc}
\toprule
Function  & $\lambda_1$ & $\lambda_2$ & $\lambda_3$ & $\lambda_4$ & $\lambda_5$ & $\lambda_6$ & $\lambda_7$ \\
\\[-5pt]
\cline{1-8}
\\[-5pt]
$\sigma^2$  & 0.01 & 0.01 & 0.01 & 0.01 & 0.004 & 0.006 & 0.01 \\
\\[-5pt]
\cline{1-8}
\\[-5pt]
$T_n$     & 1 & 1  & 1 & 1 & 1 & 1  & 0.99 \\
$T_B$   & 0.99 & 0.99 & 1  & 0.99 & 0.99 & 0.98 & 0.76 \\
$S^{reg}_n$   & 0.99 & 1 & 0.98 & 0.99 & 0.99 & 0.99  & 0.68  \\
\bottomrule\\
\end{tabular}
\caption{Simulated power of $T_n$, $T_B$ and $S^{reg}_n$ for $n=100$.}
\label{tab:1}
\end{table}
We see that, apart from the last case,  all the three tests perform very well and they are comparable. 
However, our test behaves much better for the function $\lambda_7$, which is more difficult {to detect} than the others.

The second model that we consider {is taken from~\cite{ABD14} and~\cite{GHJK2000}, which is a regression function given by}
\[
\lambda_a(x)=-(1+x)+a\exp\left(-50(x-0.5)^2\right), \qquad x\in[0,1].
\]
The results of the simulation, again for $n=100$, $\alpha=0.05$, {$b=0.1$} and various values of $a$ and $\sigma^2$ are given in Table~\ref{tab:2}. 
We denote by $S^{reg}_n$ the test of~\cite{ABD14} and by $T_{run}$ the test of~\cite{GHJK2000}.
\begin{table}[h]
\begin{tabular}{cccccccccccc}
\toprule
& \multicolumn{3}{c}{$a=0$} & & \multicolumn{3}{c}{$a=0.25$} & & \multicolumn{3}{c}{$a=0.45$}
\\
\\[-5pt]
\cline{2-4}  \cline{6-8} \cline{10-12}
\\[-5pt]
$\sigma$ & 0.025 & 0.05 & 0.1
&& 0.025 & 0.05 & 0.1
&& 0.025 & 0.05 & 0.1
\\
\\[-5pt]
\hline
\\[-5pt]
$T_n$
& 0.012   & 0.025 & 0.022  && 0.927 & 0.497 & 0.219 && 1 & 1 & 0.992 \\
$T_{run}$
& 0 & 0 & 0  && 0.106  & 0.037 & 0.014 && 1 & 1 & 0.805   \\
$S^{reg}_n$
& 0 & 0.002 & 0.013 && 0.404  & 0.053 & 0.007 && 1 & 1 & 0.683  \\
\bottomrule\\
\end{tabular}
\caption{Simulated power of $T_n$, $T_{run}$ and $S^{reg}_n$ for $n=100$.}
\label{tab:2}
\end{table}
Note that when $a=0$, the regression function is decreasing so $H_0$ is satisfied. We observe  that  our test rejects the null hypothesis more often than $T_{run}$ and $S^{reg}_n$ but, however, it has rejection probability smaller than $0.05$. As the value of $a$ increases, the monotonicity of $\lambda_a$ is perturbed. For $a=0.25$ our test performs significantly better than  the other two and, as expected, the power decreases as the variance of the errors increases. When  $a=0.45$ and $\sigma^2$ not to large, the three test show optimal power but, when $\sigma^2$ increases, $T_n$ outperforms $T_{run}$ and $S^{reg}_n$.

We note that the test performs the same way if, instead of the $L_2$-distance between $\tilde{\lambda}^{SG}_n$ and $\hat{\lambda}^{s}_n$, we use the $L_1$-distance on $(0,1)$. 
Indeed, in Remark~\ref{re:boundaries} we showed that , for $p=1$, the limit theorem holds on the whole interval $(0,1)$. 
Moreover, we did not investigate the choice of the bandwidth. 
We take {$b=0.1$}, which seems to be a reasonable one considering that the whole interval has length one.

\section{Auxiliary results and proofs}
\label{sec:proofs}
\subsection{Proofs for Section~\ref{sec:kernel}}
\label{subsec:proofs section kernel}

\begin{lemma}
	\label{lemma:Gamma2}
	Let {$L:[0,1]\to\R$ be strictly positive and twice differentiable,}
	such that $\inf_{{t\in[0,1]}}L'(t)>0$ and $\sup_{{t\in[0,1]}}|L''(t)|<\infty$.
	Let $\Gamma^{(2)}_n$, {$g_{(n)}$, and $m_n^c(p)$ be defined in~\eqref{def:Gamma2}, \eqref{eqn:def-g_n}, and~\eqref{eqn:def-m_n},
		respectively.}
	Assume that (A1) {and}~(A3) hold.
	\begin{enumerate}
		\item If $nb^5\to 0$, {then}
		\[
		(b\sigma^2(p))^{-1/2}\left\{\int_b^{1-b}\left|b^{-1/2}\Gamma^{(2)}_n(t)+g_{(n)}(t)\right|^p\,\dd \mu(t)-m_n^c(p) \right\}\xrightarrow{d} N(0,1),
		\]
		{where $\sigma^2(p)$ is defined in~\eqref{eqn:def-sigma}.}
		\item If $nb^5\to C_0^2$, {then}
		\[
		(b\tilde\theta^2(p))^{-1/2}\left\{\int_b^{1-b}\left|b^{-1/2}\Gamma^{(2)}_n(t)+g_{(n)}(t)\right|^p\,\dd \mu(t)-m_n^c(p) \right\}\xrightarrow{d} N(0,1),
		\]
		{where $\tilde\theta^2(p)$ is defined in~\eqref{eqn:def-theta-tilde}.}
	\end{enumerate}
\end{lemma}
\begin{proof}
	From the properties of the kernel function and $L$ we have
	\[
	\begin{split}
	\Gamma^{(2)}_n(t)
	&=
	\int k\left(\frac{t-u}{b}\right)\,\mathrm{d}W(L(u))-\frac{W(L(1))}{L(1)}\int k\left(\frac{t-u}{b}\right){L'(u)}\,\mathrm{d}u\\
	&=
	\int k\left(\frac{t-u}{b}\right)\,\mathrm{d}W(L(u))-b\frac{W(L(1))}{L(1)}L'(t)+O_P(b^3),
	\end{split}
	\]
	where the $O_P$ term is {uniformly for} $t\in[0,1]$.
	{Hence}, inequality \eqref{eqn:L_p-inequality} implies that
	\[
	\begin{split}
	&\int_b^{1-b}\left|b^{-1/2}\Gamma^{(2)}_n(t)+g_{(n)}(t) \right|^p\,\dd \mu(t)\\
	&=\int_b^{1-b}\left|b^{-1/2}\int k\left(\frac{t-u}{b}\right)\,\mathrm{d}W(L(u))+g_{(n)}(t)-b^{1/2}\frac{W(L(1))}{L(1)}L'(t) \right|^p\,\dd \mu(t)+O(b^3).
	\end{split}
	\]
	{Therefore}, it is sufficient to prove a CLT for
	\begin{equation}
	\label{eqn:CLT1}
	\int_b^{1-b}\left|b^{-1/2}\int k\left(\frac{t-u}{b}\right)\,\mathrm{d}W(L(u))+g_{(n)}(t)-b^{1/2}\frac{W(L(1))}{L(1)}L'(t) \right|^p\,\dd \mu(t).
	\end{equation}
	Let
	\begin{equation}
	\label{def:Xnt}
	X_{n,t}=b^{-1/2}\int k\left(\frac{t-u}{b}\right)\,\mathrm{d}W(L(u))+g_{(n)}(t).
	\end{equation}
	{Then $X_{nt}\sim N(g_{(n)}(t),\sigma^2_n(t))$, where
		\begin{equation}
		\label{def:sigman}
		\sigma^2_n(t)=\frac{1}{b}\int k^2\left( \frac{t-u}{b}\right)L'(u)\,\dd u.
		\end{equation}
		We can then write
		\begin{equation}
		\label{eqn:taylor}
		\begin{split}
		&
		b^{-1/2}
		\left\{
		\int_b^{1-b}\left|b^{-1/2}\Gamma^{(2)}_n(t)+g_{(n)}(t)\right|^p\,\dd \mu(t)-m_n^c(p)
		\right\}\\
		&=
		b^{-1/2}
		\left\{
		\int_b^{1-b}\left|X_{n,t}-b^{1/2}\frac{W(L(1))}{L(1)}L'(t)\right|^p\,\dd \mu(t)
		-
		m_n^c(p)
		\right\}
		+
		o(1)\\
		&=
		b^{-1/2}
		\left\{
		\int_b^{1-b}\left|X_{n,t}\right|^p\,\dd \mu(t)
		-
		m_n^c(p)
		\right\}\\
		&\qquad-
		p\frac{W(L(1))}{L(1)}\int_b^{1-b}\left|X_{n,t}\right|^{p-1}\sgn\left\{X_{n,t}\right\}L'(t)w(t)\,\dd t\\
		&\qquad+
		b^{-1/2}\int_b^{1-b}O\left(bW(L(1))^2 \right)\,\dd t
		+
		o(1),
		\end{split}
		\end{equation}
		where we use
		\begin{equation}
		\label{eq:diff p powers}
		|x|^p=|y|^p+p(x-y)|y|^{p-1}\sgn(y)+O((x-y)^2)
		\end{equation}
		for the first term in the integrand on the right hand side of the first equality in~\eqref{eqn:taylor}.}
	The {third} term {on the right hand side of~\eqref{eqn:taylor}} converges to zero in probability,
	so it suffices to deal with the first two terms.
	{To establish a central limit theorem for the first term}, one can mimic the approach in~\cite{CH88}
		using a big-blocks-small-blocks procedure.
		See Lemmas~\ref{lemma:Gamma1} and~\ref{lemma:asymptotic_normality} in {the supplemental material}~\cite{LM_L_p_2018} for details.
		It can be shown that
		\[
		b^{-1/2}\left\{
		\int_b^{1-b}\left|X_{n,t}\right|^p\,\dd \mu(t)
		-
		m_n^c(p)
		\right\}
		=
		b^{1/2}\sum_{i=1}^{M_3}\zeta_i+o_P(1),
		\]
		where $\zeta_i=\sum_{j=c_i}^{d_i} \xi_j$, with $c_i=(i-1)(M_2+2)+1$ and $d_i=(i-1)(M_2+2)+M_2$,
		$M_2=[(M_1-1)^\nu]$, for some $0<\nu<1$ and $M_1=[1/b-1]$, $M_3=[(M_1-1)/(M_2+2)]$,
		and
		\[
		\xi_i
		=
		b^{-1}
		\int_{ib}^{ib+b}
		\bigg\{\left|X_{n,t}\right|^p
		-
		\int_{-\infty}^{+\infty}\left|\sqrt{L'(t)} D x+g_{(n)}(t)\right|^p \phi(x)\,\dd x\bigg\} w(t)\,\dd t.
		\]
		The random variables $\zeta_i$ are independent and satisfy
	\begin{equation}
	\label{eq:CLT BM}
	b^{1/2}\sum_{i=1}^{M_3}\zeta_i\xrightarrow{d}N(0,{\gamma^2(p)}),
	\end{equation}
	{where $\gamma^2(p)$ is defined in~\eqref{eqn:def-gamma}.}
	
	{Next,} consider the second term in the right hand side of ~\eqref{eqn:taylor}.
	We have
	\[
	\begin{split}
	&
	\E\left[ \int_b^{1-b}\left|X_{n,t}\right|^{p-1}\sgn\left\{X_{n,t}\right\}L'(t)w(t)\,\dd t\right]\\
	&\quad=
	\int_b^{1-b}\int_{\R}\left|\sigma_n(t)x+g_{(n)}(t)\right|^{p-1}\sgn\left\{\sigma_n(t)x+g_{(n)}(t)\right\}\phi(x)\,\dd x\, L'(t)w(t)\,\dd t\\
	&\quad\to
	\int_0^{1}\int_{\R}\left|\sqrt{L'(t)}Dx+g(t)\right|^{p-1}\sgn\left\{\sqrt{L'(t)}Dx+g(t)\right\}\phi(x)\,\dd x\, L'(t)w(t)\,\dd t,
	\end{split}
	\]
	where {$D$ and $\sigma_n(t)$ are defined in~\eqref{eqn:r(s)} and~\eqref{def:sigman}, respectively, and $\phi$ denotes the standard normal density.}
	Note that
	\[
	\frac{d}{dx}\left|\sqrt{L'(t)}Dx+g(t)\right|^{p}
	=
	p\left|\sqrt{L'(t)}Dx+g(t)\right|^{p-1}\sgn\left\{\sqrt{L'(t)}Dx+g(t)\right\}.
	\]
	Hence, {integration by parts gives}
	\[
	\int_0^{1}\int_{\R}\left|\sqrt{L'(t)}Dx+g(t)\right|^{p-1}\sgn\left\{\sqrt{L'(t)}Dx+g(t)\right\}\phi(x)\,\dd x\, L'(t)w(t)\,\dd t=\frac{\theta_1(p)}{Dp},
	\]
	where $\theta_1$ is defined in \eqref{eqn:def-theta1}.
	{We conclude}
	\[
	\E\left[ \int_b^{1-b}\left|X_{n,t}\right|^{p-1}\sgn\left\{X_{n,t}\right\}L'(t)w(t)\,\dd t\right]\to \frac{\theta_1(p)}{Dp}.
	\]
	Moreover,
	{\[
		\begin{split}
		&
		\text{Var}\left( \int_b^{1-b}\left|X_{n,t}\right|^{p-1}\sgn\left\{X_{n,t}\right\}L'(t)w(t)\,\dd t\right)\\
		&=
		\int_b^{1-b}\int_b^{1-b}
		\text{Covar}\left(
		\left|X_{n,t}\right|^{p-1}\sgn\left\{X_{n,t}\right\},
		\left|X_{n,s}\right|^{p-1}\sgn\left\{X_{n,s}\right\}
		\right)
		L'(t)L'(s)w(t)w(s)\,\dd t\,\dd  s\\
		&=
		\int_b^{1-b}\int_b^{1-b}\1_{\{|t-s|\leq 2b\}}
		\text{Covar}\left(
		\left|X_{n,t}\right|^{p-1}\sgn\left\{X_{n,t}\right\},
		\left|X_{n,s}\right|^{p-1}\sgn\left\{X_{n,s}\right\}
		\right)\\
		&\qquad\qquad\qquad\qquad\qquad\qquad\qquad\qquad\cdot
		L'(t)L'(s)w(t)w(s)\,\dd t\,\dd  s,
		\end{split}
		\]}%
	because for $|t-s|>2b$, $X_{n,t}$ is independent of $X_{n,s}$.
	As a result, using that $X_{n,t}$ has bounded moments, we obtain
	\[
	\text{Var}\left( \int_b^{1-b}\left|X_{n,t}\right|^{p-1}\sgn\left\{X_{n,t}\right\}L'(t)w(t)\,\dd t\right)\to 0.
	\]
	This means that
	\[
	\int_b^{1-b}\left|X_{n,t}\right|^{p-1}\sgn\left\{X_{n,t}\right\}L'(t)w(t)\,\dd t\to \frac{\theta_1(p)}{Dp},
	\]
	{in probability}
	and
	\[
	-p\frac{W(L(1))}{L(1)}\int_b^{1-b}\left|X_{n,t}\right|^{p-1}\sgn\left\{X_{n,t}\right\}L'(t)w(t)\,\dd t=CW(L(1))+o_P(1),
	\]
	where
	\begin{equation}
	\label{def:C}
	C=-\frac{\theta_1(p)}{DL(1)}.
	\end{equation}
	Going back to \eqref{eqn:taylor}, {we conclude that}
	\begin{equation}
	\label{eq:expansion BM}
	\begin{split}
	&
	{b^{-1/2}
		\left\{
		\int_b^{1-b}\left|b^{-1/2}\Gamma^{(2)}_n(t)+g_{(n)}(t)\right|^p\,\dd \mu(t)-m_n^c(p)
		\right\}}\\
	&\quad=
	b^{1/2}\sum_{i=1}^{M_3}\zeta_i+CW(L(1))+o_P(1).
	\end{split}
	\end{equation}
	
	In the case $nb^5\to 0$, {we have $g(t)=0$ in the definition of $\theta_1(p)$ in~\eqref{eqn:def-theta1}}.
	Hence, by the symmetry of the  standard normal distribution, it follows that $\theta_1(p)=0$ and as a result $C=0$.
	{According to~\eqref{eq:CLT BM} and~\eqref{eq:expansion BM}, this means that
		\[
		b^{-1/2}
		\left\{
		\int_b^{1-b}\left|b^{-1/2}\Gamma^{(2)}_n(t)+g_{(n)}(t)\right|^p\,\dd \mu(t)-m_n^c(p)
		\right\}
		\]
		converges in distribution to a mean zero normal random variable with variance $\sigma^2(p)$.}
	
	{Then, consider} the case $nb^5\to C_0^2>0.$
	Note that $\zeta_i$ depends only on the Brownian motion on the interval {$[c_ib-b,c_ib+b]$.
		These} intervals are disjoint, because $c_{i+1}b-b=d_ib+b$.
	We write
	{\[
		W(L(1))=\sum_{i=1}^{M_3}\left[W(t_{i+1})-W(t_i) \right]+W(L(1))-W(t_{M_3}),
		\]
		where $t_i=L(c_ib-b)$, for $i=1,\ldots,M_3$.
		Moreover, $W(L(1))-W(t_{M_3})\to 0$,
		in probability, since $t_{M_3}\sim L(1+O(b))\to L(1)$.
		Hence, the left hand side of~\eqref{eq:expansion BM}, can be written as
		\[
		\sum_{i=1}^{M_3}Y_i+o_P(1),
		\quad
		Y_i=b^{1/2}\zeta_i+C\left[W(t_{i+1})-W(t_i) \right].
		\]
		Since now we have a sum of independent random variables, we apply the Lindeberg-Feller central limit theorem.}
	Using $\E[Y_i]=O(b^{5/2}M_2)$, it suffices to show that
	\begin{equation}
	\label{eqn:CLT_var}
	\E\left[\left(\sum_{i=1}^{M_3}Y_i\right)^2 \right]\to \tilde\theta^2(p){>0,}
	\end{equation}
	and that the Lyapounov condition
	\begin{equation}
	\label{eqn:CLT_Lyap}
	\sum_{i=1}^{M_3}\E[Y_i^4]\left( \sum_{i=1}^{M_3}\E[Y_i^2]\right)^{-2}\to 0.
	\end{equation}
	is satisfied.
	Once we have \eqref{eqn:CLT_var}, {condition~\eqref{eqn:CLT_Lyap} is equivalent to
		$\sum_{i=1}^{M_3}\E[Y_i^4]\to 0$.
		In order to prove this, we use that $\E[\zeta_i^4]=O(M_2^2)$, (see~\eqref{eqn:fourth_moment_zeta} in the proof of Lemma~\ref{lemma:asymptotic_normality}
		in {the supplemental material}~\cite{LM_L_p_2018}).
		Then,} we get
	\[
	\begin{split}
	\sum_{i=1}^{M_3}\E[Y_i^4]
	&\leq
	O(b^2)\sum_{i=1}^{M_3}\E[\zeta_i^4]+O(1)\sum_{i=1}^{M_3}\E[(W(t_{i+1})-W(t_i))^4]\\
	&\leq
	O(M_3b^2M_2^2)+O(M_3(t_{i+1}-t_i)^2)\\
	&=
	o(1)+O(M_3M_2^2b^2)=o(1).
	\end{split}
	\]
	Because $\E[Y_i]=O(b^{5/2}M_2)$, for~\eqref{eqn:CLT_var} we have
	\[
	\begin{split}
	\E\left[\left(\sum_{i=1}^{M_3}Y_i\right)^2 \right]
	&=
	{\sum_{i=1}^{M_3}\E\left[Y_i^2\right]}+o(1)\\
	&=
	{b\sum_{i=1}^{M_3}\E\left[\zeta_i^2\right]
		+
		C^2\sum_{i=1}^{M_3}(t_{i+1}-t_i)
		+
		2Cb^{1/2}\sum_{i=1}^{M_3}\E[\zeta_i\{W(t_{i+1})-W(t_i)\}]+o(1).}
	\end{split}
	\]
	{It can be shown that $b\sum_{i=1}^{M_3}\E\left[\zeta_i^2\right]\to0$, see Lemma~\ref{lemma:asymptotic_normality} in {the supplemental material}~\cite{LM_L_p_2018} for details.
		Moreover, $\sum_{i=1}^{M_3}(t_{i+1}-t_i)=L((M_3-1)(M_2+2)b)-L(0)=L(1)+o(1)$.
		Finally, since
		\[
		\zeta_i
		=
		b^{-1}
		\int_{c_ib}^{d_ib}
		\bigg\{\left|X_{n,t}\right|^p
		-
		\int_{-\infty}^{+\infty}\left|\sqrt{l(t)} D x+g_{(n)}(t)\right|^p \phi(x)\,\dd x\bigg\} w(t)\,\dd t,
		\]
		we can write
		\[
		2Cb^{1/2}\sum_{i=1}^{M_3}\E[\zeta_i\{W(t_{i+1})-W(t_i)\}]
		=
		2C\sum_{i=1}^{M_3}\int_{c_ib}^{d_ib}\E\left[|X_{n,t}|^pZ_{n,t} \right]w(t)\,\dd t,
		\]}%
	where $Z_{n,t}=b^{-1/2}\{W(t_{i+1})-W(t_i)\}$.
	Note that
	\[
	(X_{n,t},Z_{n,t})
	\sim
	N\left(
	\begin{bmatrix}
	g_{(n)}(t) \\ 0
	\end{bmatrix},
	\begin{bmatrix}
	\sigma^2_n(t) && \rho_n(t)\sigma_n(t)\tilde{\sigma}_n(t) \\
	\rho_n(t)\sigma_n(t)\tilde{\sigma}_n(t) && \tilde\sigma^2_n(t))
	\end{bmatrix}\right).
	\]
	where {$\sigma^2_n(t)$ is defined in~\eqref{def:sigman} and}
	\[
	\tilde{\sigma}^2_n(t)=b^{-1}[L(t+b)-L(t-b)],
	\qquad
	\rho_n(t)
	=
	\sigma_n(t)^{-1}\tilde{\sigma}_n(t)^{-1}b^{-1}\int k\left(\frac{t-u}{b}\right)l(u)\,\dd u.
	\]
	Using
	\[
	Z_{n,t}\mid X_{n,t}=x
	\sim N\left(\frac{\tilde{\sigma}_n(t)}{\sigma_n(t)}\rho_n(t)(x-g_{(n)}(t)),
	\left(1-\rho_n^2(t)\right)\tilde{\sigma}_n^2(t)\right)
	\]
	we obtain
	\[
	\begin{split}
	\E\left[
	|X_{n,t}|^pZ_{n,t}
	\right]
	&=
	\E\left[|X_{n,t}|^p\,\E[Z_{n,t}\mid X_{n,t}]\right]\\
	&=
	\E\left[|X_{n,t}|^p\frac{\tilde{\sigma}_n(t)}{\sigma_n(t)}\rho_n(t)
	\left(X_{n,t}-g_{(n)}(t)\right)\right]\\
	&=
	\frac{\tilde{\sigma}_n(t)}{\sigma_n(t)}\rho_n(t)
	\E\left[|X_{n,t}|^p\left(X_{n,t}-g_{(n)}(t)\right)\right]\\
	&=
	\frac{\tilde{\sigma}_n(t)}{\sigma_n(t)}\rho_n(t) \int_{\R}|g_{(n)}(t)+\sigma_n(t)x|^p\sigma_n(t) x\phi(x)\,\dd x\\
	&=
	\sigma_n(t)^{-1}b^{-1}\int k\left(\frac{t-u}{b}\right)l(u)\,\dd u\int_{\R}|g_{(n)}(t)+\sigma_n(t)x|^p x\phi(x)\,\dd x.
	\end{split}
	\]
	{Because $\sigma^2_n(t)\to D^2l(t)$, where $D$ is defined in~\eqref{eqn:r(s)},
		$g_{(n)}(t)\to g(t)$, as defined in~\eqref{eqn:def-g(u)}, and $b^{-1}\int k\left(\frac{t-u}{b}\right)l(u)\,\dd u\to l(t)$,
		we find that
		\[
		\E\left[
		|X_{n,t}|^pZ_{n,t}
		\right]
		\to
		\frac{\sqrt{l(t)}}{D}\int_{\R}|g(t)+D\sqrt{l(t)}x|^p x\phi(x)\,\dd x.
		\]}%
	Hence
	\[
	\begin{split}
	\E\left[\left(\sum_{i=1}^{M_3}Y_i\right)^2\right]
	&=
	\theta^2(p)+C^2L(1)\\
	&\qquad+
	2CD^{-1}\sum_{i=1}^{M_3}\int_{c_ib}^{d_ib}\int_{\R}|g(t)
	+
	\sqrt{l(t)}Dx|^p x\phi(x)\,\dd x\, \sqrt{l(t)} w(t)\,\dd t+o(1)\\
	&=
	\theta^2(p)+C^2L(1)\\
	&\qquad+
	2CD^{-1}
	\int_{0}^{1}\int_{\R}|g(t)+D\sqrt{l(t)}x|^p x\phi(x)\,\dd x\, \sqrt{l(t)} w(t)\,\dd t+o(1)\\
	&=
	\theta^2(p)+C^2L(1)+2CD^{-1}\theta_1(p)+o(1)\\
	&=
	\theta^2(p)-\frac{{\theta^2_1(p)}}{D^2L(1)}+o(1),
	\end{split}
	\]
	{applying the definitions of $C$ and $\theta_1(p)$
		in~\eqref{def:C} and~\eqref{eqn:def-theta1}, respectively.}
	It follows from {the} Lindeberg-Feller central limit theorem that $\sum_{i=1}^{M_3} Y_i\xrightarrow{d} N(0,\tilde{\theta}^2(p))$,
	{where $\tilde\theta(p)$ is defined in~\eqref{eqn:def-theta-tilde}.}
\end{proof}

\subsection{Proofs for Section~\ref{sec:SG}}
\label{subsec:proofs section SG}

\begin{lemma}
\label{lem:Yn-Yn1}
Let $Y_n$ and $Y_n^{(1)}$ be defined in~\eqref{def:Yn} and~\eqref{def:tilde_Y}, respectively.
Assume that $(A1)-(A2)$ hold.
If  $1\leq p<min(q,2q-7)$, $1/b=o\left(n^{1/3-1/q}\right)$ and $1/b=o\left(n^{(q-3)/(6p)}\right)$, then
\[
{b^{-p}}\int_{b}^{1-b} |Y_n(t)-{Y_n^{(1)}}(t)|^p\,\dd\mu(t)
=
o_P(1).
\]
\end{lemma}
\begin{proof}
{We follow} the same reasoning as in the proof of Lemma~8 in~\cite{lopuhaa-mustaSPL2017}.
Let $I_{nv}=[0,1]\cap[v-n^{-1/3}\log n,v+n^{-1/3}\log n]$ and for $J=E,W$, let
\begin{equation}
\label{def:NJnv}
N^J_{nv}
=
\left\{
[\CM_{[0,1]}\Lambda_n^W](s)=[\CM_{I_{nv}}\Lambda_n^W](s)\text{ for all }s\in I_{nv}
\right\}.
\end{equation}
Then according to Lemma~3 in~\cite{lopuhaa-mustaSPL2017},
there exists $C>0$, independent of $n,v,d$, such that
\begin{equation}
\label{eq:bounds prob}
\begin{split}
\p\left((N_{nv}^W)^c\right)
&=
O(\text{e}^{-Cd^3})\\
\p\left((N_{nv}^E)^c\right)
&=
O(n^{1-q/3}d^{-2q}+\text{e}^{-Cd^3}).
\end{split}
\end{equation}
Let $K_{nv}=N_{nv}^E\cap N_{nv}^W$
and write
\[
\begin{split}
\E
\left[
\left|
A_n^E(v)^p-A_n^W(v)
\right|
\right]
&=
\E
\left[
\left|
A_n^E(v)^p-A_n^W(v)
\right|
\1_{K_{nv}^c}
\right]\\
&\quad+
n^{2p/3}
\E
\left[
\left|
[\DM_{I_{nv}}\Lambda_n](t)^p
-
[\DM_{I_{nv}}\Lambda_n^W](t)^p
\right|
\1_{K_{nv}}
\right].
\end{split}
\]
From the proof of Lemma~8 in~\cite{lopuhaa-mustaSPL2017}, using~\eqref{eq:bounds prob} with $d=\log n$, we have
{
\[
\E
\left[
\left|
A_n^E(v)^p-A_n^W(v)
\right|
\1_{K_{nv}^c}
\right]
=
O_P(n^{1/2-q/6}(\log n)^{-q}+\text{e}^{-C(\log n)^3/2}/2)
\]
and
\[
n^{2p/3}
\E
\left[
\left|
[\DM_{I_{nv}}\Lambda_n](t)^p
-
[\DM_{I_{nv}}\Lambda_n^W](t)^p
\right|
\1_{K_{nv}}
\right]
=
O_p\left(n^{-1/3+1/q}\right).
\]
It follows that}
\[
\begin{split}
&
{b^{-p}}\int_{b}^{1-b} |Y_n(t)-{Y_n^{(1)}}(t)|^p\,\dd\mu(t)\\
&\leq
C{b^{-p}}\int_{-1}^1 |{A_n^E}(t-by)-A^W_n(t-by)|^p\,\dd y\\
&=
{b^{-p}}
O_P\left(n^{-p/3+p/q} \right)+
{b^{-p}}
O_P\left(n^{1/2-q/6} (\log n)^{-q}
+
{\text{e}^{-C(\log n)^3/2}}
\right).
\end{split}
\]
According to the assumptions on the order of $b^{-1}$, the right hand side
is of order $o_P(1)$.
\end{proof}
\begin{lemma}
\label{lem:Yn1-Yn2}
Let $Y_n^{(1)}$ and $Y_n^{(2)}$ be defined in~\eqref{def:tilde_Y} and~\eqref{def:Yn2}, respectively.
Assume that $(A1)-(A2)$ hold.
If {$b\to0$,} such that $nb\to\infty$, then
\[
{b^{-p}}\int_{b}^{1-b} |Y_n^{(1)}(t)-{Y_n^{(2)}}(t)|^p\,\dd\mu(t)
=
o_P(1).
\]
\end{lemma}
\begin{proof}
{We have}
\[
\begin{split}
&
\sup_{t\in(b,1-b)}
\E\left[\left|Y_n^{(1)}(t) -Y_n^{(2)}(t)  \right|^p\right]\\
&=
\sup_{t\in(b,1-b)}
\E\left[\left|\frac{1}{b}\int_{t-b}^{t+b} k'\left(\frac{t-v}{b}\right)
\1_{(N^W_{nv})^c}
\left(
A^W_n(v)-n^{2/3}[\DM_{I_{nv}}\Lambda^W_n](v)
\right)\,\dd v \right|^p\right]\\
&\leq
\sup_{u\in[0,1]}|k'(u)|^p\sup_{t\in(b,1-b)}
\E\left[\sup_{v\in[0,1]}\left|A^W_n(v)-n^{2/3}[\DM_{I_{nv}}\Lambda^W_n](v)\right|^p
\left( \frac{1}{b}\int_{t-b}^{t+b} \1_{(N^W_{nv})^c}\,\dd v\right)^p\right],
\end{split}
\]
where $N^W_{nv}$ is defined in~\eqref{def:NJnv}.
Moreover, since
\[
\sup_{v\in[0,1]}\left|A^W_n(v)-n^{2/3}[\DM_{I_{nv}}\Lambda^W_n](v)\right|\leq 4 n^{2/3}\left\{\Lambda(1)+n^{-1/2}\sup_{s\in[0,L(1)]}|W_n(s)|\right\},
\]
from the Cauchy-Schwartz inequality we obtain
\[
\begin{split}
&
\sup_{t\in(b,1-b)}\E\left[\sup_{v\in[0,1]}\left|A^W_n(v)-n^{2/3}[\DM_{I_{nv}}\Lambda^W_n](v)\right|^p
\left( \frac{1}{b}\int_{t-b}^{t+b} \1_{(N^W_{nv})^c}\,\dd v\right)^p\right]\\
&\qquad\leq
4^p
n^{2p/3}\E\left[\left\{\Lambda(1)+n^{-1/2}\sup_{s\in[0,L(1)]}|W_n(s)|\right\}^{2p}\right]^{1/2}\\
&\qquad\qquad\qquad\qquad\cdot
\sup_{t\in(b,1-b)}\E\left[\left( \frac{1}{b}\int_{t-b}^{t+b} \1_{(N^W_{nv})^c}\,\dd v\right)^{2p}\right]^{1/2}.
\end{split}
\]
For the last term on the right hand side, we can use Jensen's inequality:
\[
\left(\frac{1}{b-a}\int_a^bf(x)\,\dd x \right)^p\leq \frac{1}{b-a}\int_a^b f(x)^p\,\dd x,
\]
for all $a<b$, $p\geq1$, and $f(x)\geq 0$.
Because all the moments of $\sup_{s\in[0,L(1)]}|W_n(s)|$ are finite,
together with~\eqref{eq:bounds prob},
it follows that
\begin{equation}
\label{eqn:difference_local_approximation}
\begin{split}
\sup_{t\in(b,1-b)}
\E\left[\left|Y_n^{(1)}(t) -Y_n^{(2)}(t)\right|^p\right]
&\leq
Cn^{2p/3}\sup_{t\in(b,1-b)}\E\left[\frac{1}{b}\int_{t-b}^{t+b} \1_{(N^W_{nv})^c}\,\dd v\right]^{1/2}\\
&=
O\left(n^{2p/3}\exp\left(-C(\log n)^3/2\right) \right).
\end{split}
\end{equation}
Because $b^{-p}n^{2p/3}\exp\left(-C(\log n)^3/2\right)=(nb)^{2p/3-C(\log n)^2/2}b^{-p-2p/3+C(\log n)^2/2}\to0$,
this finishes the proof.
\end{proof}

\begin{lemma}
\label{lem:Yn2-Yn3}
Let $Y_n^{(2)}$ and $Y_n^{(3)}$ be defined in~\eqref{def:Yn2} and~\eqref{def:Yn3}, respectively.
Assume that $(A1)-(A2)$ hold.
If {$1/b=o\left(n^{1/3-1/q}\right)$,}
then
\[
{b^{-p}}\int_{b}^{1-b} |Y_n^{(2)}(t)-{Y_n^{(3)}}(t)|^p\,\dd\mu(t)
=
o_P(1).
\]
\end{lemma}
\begin{proof}
{Let $H_{nv}=[-n^{1/3}v,n^{1/3}(1-v)]\cap [-\log n,\log n]$ and $\Delta_{nv}=n^{2/3}[\DM_{I_{nv}}\Lambda^W_n](v)-[\DM_{H_{nv}}Y_{nv}](0)$.
By definition, we have
\[
\int_{b}^{1-b} |Y_n^{(2)}(t)-{Y_n^{(3)}}(t)|^p\,\dd\mu(t)=\int_{b}^{1-b}\left|\frac{1}{b}\int_{t-b}^{t+b} k'\left(\frac{t-v}{b}\right)\Delta_{nv}\,\dd v \right|^p\,\dd\mu(t).
\]}
Moreover, using
\[
\sup_{t\in(0,1)}\E\left[\left|\Delta_{nt}\right|^p\right]=O\left(n^{-p/3+p/q}\right)
\]
(see the proof of Lemma 6 in \cite{lopuhaa-mustaSPL2017}), we obtain
\begin{equation}
\label{eqn:delta_n}
\begin{split}
&
\sup_{t\in(b,1-b)}\E\left[\left|\frac{1}{b}
\int_{t-b}^{t+b} k'\left(\frac{t-v}{b}\right)\Delta_{nv}\,\dd v \right|^p\right]\\
&\leq
\sup_{u\in[-1,1]}|k'(u)|^p\sup_{t\in(b,1-b)}
\E\left[\left|\frac{1}{b}\int_{t-b}^{t+b} \Delta_{nv}\,\dd v \right|^p\right]\\
&\leq
C\sup_{t\in(b,1-b)}\frac{1}{b}\int_{t-b}^{t+b} \E\left[\left|\Delta_{nv}\right|^p\right]\,\dd v
\leq
{2C}\sup_{v\in(0,1)}\E\left[\left|\Delta_{nv}\right|^p\right]\\
&=O\left(n^{-p/3+p/q}\right).
\end{split}
\end{equation}
Because $1/b=o\left(n^{1/3-1/q}\right)$, this finishes the proof.
\end{proof}

\begin{lemma}
\label{lem:Yn3-Yn4}
Let $Y_n^{(3)}$ and $Y_n^{(4)}$ be defined in~\eqref{def:Yn3} and~\eqref{def:Yn4}, respectively.
Assume that $(A1)-(A2)$ hold.
If {$1/b=o\left(n^{1/3-1/q}\right)$, then}
\[
{b^{-p}}\int_{b}^{1-b} |Y_n^{(3)}(t)-{Y_n^{(4)}}(t)|^p\,\dd\mu(t)
=
o_P(1).
\]
\end{lemma}
\begin{proof}
Let $H_{nv}$ be defined as in the proof of Lemma~\ref{lem:Yn2-Yn3}
and let
$J_{nv} = [n^{1/3} (L(a_{nv}) - L(v)) /L'(v),
n^{1/3} (L(b_{nv} ) - L(v)) /L'(v)]$,
where $a_{nv} = \max(0, v - n^{-1/3} \log n)$
and $b_{nv} = \min(1, v + n^{-1/3} \log n)$.
As in (4.31) in \cite{kulikov-lopuhaa2008} we have
\begin{equation}
\label{eq:bound tilde Y - Z}
\sup_{v\in(0,1)}\E\left[\left|[\DM_{H_{nv}}\tilde{Y}_{nv}](0)-[\DM_{J_{nv}}Z_{nv}](0)\right|^p\right]
=
O(n^{-p/3}(\log n)^{3p}),
\end{equation}
where $\tilde{Y}_{nv}$ and $Z_{nv}$ are defined in~\eqref{def:tilde Ynv} and~\eqref{def:Znv}.
{This means that},
\begin{equation}
\label{eqn:Y_n-Z_n}
\begin{split}
&\sup_{t\in(b,1-b)}
\E\left[\left|
{Y_n^{(3)}(t)-Y_n^{(4)}(t)}
\right|^p\right]\\
&\leq
\sup_{u\in[-1,1]}|k'(u)|^p
\sup_{t\in(b,1-b)}
\E\left[\left|\frac{1}{b}\int_{t-b}^{t+b}
\left\{
[\DM_{H_{nv}}\tilde{Y}_{nv}](0)-[\DM_{J_{nv}}Z_{nv}](0)
\right\}\,\dd v \right|^p\right]\\
&\leq C
\sup_{t\in(b,1-b)}\frac{1}{b}\int_{t-b}^{t+b} \E\left[\left| [\DM_{H_{nv}}\tilde{Y}_{nv}](0)-[\DM_{J_{nv}}Z_{nv}](0)\right|^p\right]\,\dd v \\
&\leq
C\sup_{v\in(b,1-b)}\E\left[\left| [\DM_{H_{nv}}\tilde{Y}_{nv}](0)-[\DM_{J_{nv}}Z_{nv}](0)\right|^p\right]
=
O\left(n^{-p/3}(\log n)^{3p}\right).
\end{split}
\end{equation}
Since $1/b=o\left(n^{1/3-1/q}\right)$, this finishes the proof.
\end{proof}

\begin{lemma}
\label{lem:Yn4-Yn5}
Let $Y_n^{(4)}$ and $Y_n^{(5)}$ be defined in~\eqref{def:Yn4} and~\eqref{def:Yn5}, respectively.
Assume that $(A1)-(A2)$ hold.
If {$nb\to\infty$,} such that
$1/b={o(n^{1/6+1/(6p)}(\log n)^{-(1/2+1/(2p))})}$, then
\[
{b^{-p}}\int_{b}^{1-b} |Y_n^{(4)}(t)-{Y_n^{(5)}}(t)|^p\,\dd\mu(t)
=
o_P(1).
\]
\end{lemma}

\begin{proof}
{We argue as in the proof of Lemma~4.4 in~\cite{kulikov-lopuhaa2008}.
When $v\in(n^{-1/3}\log n, 1-n^{-1/3}\log n)$, there exists $M>0$, only depending $\lambda$,
such that $[-M\log n,M\log n]\subset I_{nv}$, and on the interval $[-M\log n,M\log n]$ we have
that $\CM_{[-M\log n,M\log n]}Z\leq \CM_{I_{nv}}Z\leq \CM_\R Z$.
Let $N_{nM}=N(M\log n)$, where $N(d)$ is the event that $[\CM_{[-d,d]}Z](s)$ is equal to
$[\CM_\R Z](s)$ for $s\in[-d/2,d/2]$.
According to Lemma~1.2 in~\cite{kulikov-lopuhaa2006SPL}, it holds that
\begin{equation}
\label{eq:bound prob Nd}
\p(N(d)^c)\leq \exp(-d^3/2^7).
\end{equation}
For convenience, write $\delta_n=n^{-1/3}\log n$.
Because $[\CM_{[-M\log n,M\log n]}Z](0)=[\CM_{I_{nv}}Z](0)=[\CM_\R Z](0)$
on the event $N_{nM}$,
we have
by means of Cauchy-Schwarz, we find that
\[
\begin{split}
\sup_{v\in(\delta_n,1-\delta_n)}
\E\left[
\left|
[\DM_{I_{nv}}Z](0)-[\DM_{\R}Z](0)\right|^p \right]
&=
\sup_{v\in(\delta_n,1-\delta_n)}
\E\left[
\left|
[\DM_{I_{nv}}Z](0)-[\DM_{\R}Z](0)\right|^p \right]
\1_{N_{nM}^c}\\
&\leq
2^p
\E\left[
\left(\sup_{s\in\R}|Z(s)|\right)^p
\1_{N_{nM}^c}
\right]\\
&\leq
2^p
\left(
\E\left[
\left(\sup_{s\in\R}|Z(s)|\right)^{2p}
\right]
\right)^{1/2}
\p(N_{nM}^c)^{1/2}.
\end{split}
\]
Because $\E[(\sup |Z|)^{2p}]<\infty$, together with~\eqref{eq:bound prob Nd}, we find that
\begin{equation}
\label{eq:diff D}
\sup_{v\in(\delta_n,1-\delta_n)}
\E\left[
\left|
[\DM_{I_{nv}}Z](0)-[\DM_{\R}Z](0)\right|^p \right]
=
O\left( \exp(-C(\log n)^3)\right).
\end{equation}
Note that
\begin{equation}
\label{eq:Yn4-Yn5}
Y_n^{(4)}(t)-Y_n^{(5)}(t)
=
\frac1b
\int_{t-b}^{t+b}
k'\left(\frac{t-v}{b}\right)\frac{1}{c_1(v)}
\left(
[\DM_{I_{nv}}Z](0)-[\DM_{\R}Z](0)
\right)
\dd v.
\end{equation}
When $t\in(b+\delta_n,1-b-\delta_n)$,
then $v\in(t-b,t+b)\subset (\delta_n,1-\delta_n)$,
so after change of variables, it follows that
\begin{equation}
\label{eqn:approximation_D_R_Z}
\begin{split}
&
\sup_{t\in(b+\delta_n,1-b-\delta_n)}
\E\left[\left|Y_n^{(4)}(t)-Y_n^{(5)}(t)\right|^p\right]\\
&\quad\leq
2^p\frac{\sup_{u\in[-1,1]}|k'(u)|^p}{\inf_{v\in(0,1)}c_1(v)^p}
\sup_{v\in(\delta_n,1-\delta_n)}
\E\left[\left|[\DM_{I_{nv}}Z](0)-[\DM_{\R}Z](0)\right|^p \right]\\
&\quad=
O\left( \exp(-C(\log n)^3)\right).
\end{split}
\end{equation}
Next, consider the case where $t\in(b,b+\delta_n)$.
In this case we split the integral on the right hand side of~\eqref{eq:Yn4-Yn5}
into an integral over $v\in(t-b,\delta_n)$ and an integral over $v\in(\delta_n,t+b)$.
The latter integral can be bounded in the same way as in~\eqref{eqn:approximation_D_R_Z},
whereas for the first integral we have
\[
\begin{split}
&\left|
\frac1b
\int_{t-b}^{\delta_n}
k'\left(\frac{t-v}{b}\right)\frac{1}{c_1(v)}
\left(
[\DM_{I_{nv}}Z](0)-[\DM_{\R}Z](0)
\right)
\dd v
\right|\\
&\quad\leq
b^{-1}\delta_n
\frac{\sup_{u\in[-1,1]}|k'(u)|}{\inf_{v\in(0,1)}c_1(v)}
\left|[\DM_{I_{nv}}Z](0)-[\DM_{\R}Z](0)\right|\\
&\quad\leq
b^{-1}\delta_n
\frac{\sup_{u\in[-1,1]}|k'(u)|}{\inf_{v\in(0,1)}c_1(v)}
[\DM_{\R}Z](0),
\end{split}
\]
where we also use that
$[\DM_{I_{nv}}Z](0)\leq [\DM_{\R}Z](0)$.
Furthermore, since $[\DM_{\R}Z](0)$ has bounded moments of any order,
for $t\in(b,b+\delta_n)$, we obtain
\begin{equation}
\label{eqn:approximation_D_R_Z2}
\begin{split}
&
\sup_{t\in(b,b+\delta_n)}
\E\left[\left|Y_n^{(4)}(t)-Y_n^{(5)}(t)\right|^p\right]\\
&\quad\leq
b^{-p}\delta_n^p
\frac{\sup_{u\in[-1,1]}|k'(u)|^p}{\inf_{v\in(0,1)}c_1(v)^p}
\E\left[[\DM_{\R}]Z](0)^p\right]
+
O\left( \exp(-C(\log n)^3)\right)\\
&\quad=
O_P\left( b^{-p}\delta_n^p\right)
+
O_P\left( \exp(-C(\log n)^3)\right).
\end{split}
\end{equation}
A similar bound can be obtained for $t\in(1-b-\delta_n,1-b)$.
Putting things together yields,
\[
\int_{b}^{1-b}\left|{Y_n^{(4)}(t)}-
{Y_n^{(5)}(t)}
\right|^p\,\dd\mu(t)
=
O_P\left(\exp(-C(\log n)^3)\right)+O_P\left( b^{-p}\delta_n^{p+1}\right).
\]
Because $nb\to\infty$ implies $b^{-p}\exp(-C(\log n)^3)\to0$ and
$o(n^{1/6+1/(6p)}(\log n)^{-(1/2+1/(2p))})$ yields $b^{-2p}\delta_n^{p+1}\to0$,
this finishes the proof.
}
\end{proof}
\section*{Supplementary Material}
Supplement to ''Central limit theorems for global errors of smooth isotonic estimators''.
\begin{itemize}
\item Supplement~\ref{sec:Supp_KE}: Kernel estimator of a decreasing function.
\item Supplement~\ref{subsec:proofs section GS}: Isotonized kernel estimator.
\item Supplement~\ref{supp_Hellinger}: CLT for the Hellinger loss.
\end{itemize}

\bibliographystyle{imsart-number}
\bibliography{shapeconstrained-estimation}
\newpage
\setcounter{page}{1}
\setcounter{equation}{0}
\renewcommand{\theequation}{S\arabic{equation}}

%
%

\centerline{\LARGE\bf Central limit theorems for global errors }
\smallskip
\centerline{\LARGE\bf of smooth isotonic estimators}
\bigskip
\centerline{\LARGE Supplementary Material}
\bigskip
\centerline{\Large Hendrik Paul Lopuha\"a$^\dagger$ and Eni Musta$^\dagger$}
\medskip
\centerline{\it Delft University of Technology$^\dagger$}
\bigskip
\centerline{\date{\today}}

\appendix

\section{Kernel estimator of a decreasing function}
\label{sec:Supp_KE}
\begin{lemma}
	\label{lemma:Gamma1}
	Let $l(t)$ be a differentiable function on $[0,1]$ such that $\inf_{[0,1]}l(t)>0$ and $\sup_{[0,1]}|l'(t)|<\infty$. Define $L(t)=\int_0^t l(u)\,\dd u$ and let $\Gamma^{(1)}_n$ be as in \eqref{def:Gamma1}. Assume that (A1) and (A3) hold. Then
	\[
	(b\gamma^2(p))^{-1/2}\left\{\int_b^{1-b}\left|b^{-1/2}\Gamma^{(1)}_n(t)+g_{(n)}(t) \right|^p\,\dd \mu(t)-m_n^c(p) \right\}\xrightarrow{d} N(0,1),
	\]
	where $\gamma^2(p)$, $g_{(n)}$ and $m_n^c(p)$ are defined respectively in~\eqref{eqn:def-gamma},~\eqref{eqn:def-g_n} and~\eqref{eqn:def-m_n}.
\end{lemma}
\begin{proof}
	With a change of variable we can write
	\begin{equation}
	\begin{split}
	&\int_b^{1-b}\left|b^{-1/2}\Gamma^{(1)}_n(t)+g_{(n)}(t) \right|^p\,\dd \mu(t)-m_n^c(l,p)\\
	&\quad=b\int_1^{(1-b)/b}\left\{\left|b^{-1/2}\int_{t-1}^{t+1}k(t-y)\,\dd W(L(by))+g_{(n)}(tb) \right|^pw(tb)-\int_{\R}\left|l(tb) D x+g_{(n)}(tb)\right|^p \phi(x)\,\dd x\right\}\,\dd t\\
	\label{eqn:sum_representation}
	&\quad=b\left\{\sum_{i=1}^{M_1-1}\xi_i+\eta\right\},
	\end{split}
	\end{equation}
	where $M_1=[1/b-1]$,
	\begin{equation}
	\label{eqn:xi_i}
	\xi_i=\int_i^{i+1}\left\{\left|b^{-1/2}\int_{t-1}^{t+1}k(t-y)\,\dd W(L(by))+g_{(n)}(tb) \right|^p -\int_{-\infty}^{+\infty}\left|l(tb) D x+g_{(n)}(tb)\right|^p \phi(x)\,\dd x\right\} w(tb)\,\dd t
	\end{equation}
	and
	\[
	\eta=\int_{M_1}^{(1-b)/b}\left\{\left|b^{-1/2}\int_{t-1}^{t+1}k(t-y)\,\dd W(L(by))+g_{(n)}(tb) \right|^p -\int_{-\infty}^{+\infty}\left|l(tb) D x+g_{(n)}(tb)\right|^p \phi(x)\,\dd x\right\} w(tb)\,\dd t.
	\]
	First, we show that $\eta$ has no effect on the asymptotic distribution, i.e. is negligible.
	Using Jensen inequality and $(a+b)^p\leq 2^p(a^p+b^p)$ and the fact that $l$ and $w$ are bounded, we obtain
	\[
	\begin{split}
	\eta^2&\leq \int_{M_1}^{(1-b)/b}\left\{\left|b^{-1/2}\int_{t-1}^{t+1}k(t-y)\,\dd W(L(by))+g_{(n)}(tb) \right|^{2p} +\left(\int_{\R}\left|l(tb) D x+g_{(n)}(tb)\right|^p \phi(x)\,\dd x\right)^2\right\} w(tb)\,\dd t\\
	&\leq C_1\int_{M_1}^{(1-b)/b}\left\{\left|b^{-1/2}\int_{t-1}^{t+1}k(t-y)\,\dd W(L(by))\right|^{2p}+\left|g_{(n)}(tb)\right|^{2p}\right\} \,\dd t+C_2,
	\end{split}
	\]
	for some positive constants $C_1$ and $C_2$. On the other hand,
	\[
	\begin{split}
	\int_{M_1}^{(1-b)/b}\left|g_{(n)}(tb)\right|^{2p} \,\dd t&=(nb)^{p}\int_{M_1}^{(1-b)/b}\left|\lambda_{(n)}(tb)-\lambda(tb)\right|^{2p} \,\dd t\\
	&=(nb)^{p}b^{-1}\int_{M_1b}^{1-b}\left|\lambda_{(n)}(t)-\lambda(t)\right|^{2p} \,\dd t\\
	&=(nb)^{p}b^{-1}\int_{M_1b}^{1-b}\left|\int k(y)[\lambda(t-by)-\lambda(t)]\,\dd y\right|^{2p} \,\dd t\\
	&\leq(nb)^{p}b^{4p}\sup_{t\in[0,1]}|\lambda''(t)|^{2p}\left|\int k(y)y^2\,\dd y\right|^{2p} 
	\end{split}
	\]
	Hence,
	\begin{equation}
	\begin{split}
	\E[\eta^2]&\leq C_1\int_{M_1}^{(1-b)/b}\E\left[\left|b^{-1/2}\int_{t-1}^{t+1}k(t-y)\,\dd W(L(by))\right|^{2p}\right]+2C_3(nb)^{p}b^{4p}+C_2\\
	\label{eqn:bound_eta}
	&=O\left((nb)^{p}b^{4p}\right)=O(1).
	\end{split}
	\end{equation}
	This means that $b\eta=o_P(1)$. The statement follows immediately from Lemma~\ref{lemma:asymptotic_normality}.
\end{proof}
\begin{lemma}
	\label{lemma:asymptotic_normality}
	Let $l(t)$ be a differentiable function on $[0,1]$ such that $\inf_{[0,1]}l(t)>0$ and $\sup_{[0,1]}|l'(t)|<\infty$. Define $L(t)=\int_0^t l(u)\,\dd u$ . Assume that (A1) and (A3) hold. Let $\xi_i$, for  $i=1,\ldots ,M_1-1$, be defined as in~\eqref{eqn:xi_i}. Then we have
	\[
	b^{1/2}\gamma(p)^{-1}\sum_{i=1}^{M_1-1}\xi_i\to N(0,1),
	\]
	where $\gamma^2(p)$ is defined in \eqref{eqn:def-gamma}.
\end{lemma}
\begin{proof}
	Let  $\gamma\in(0,1)$ and $M_2=[(M_1-1)^\gamma]$, $M_3=[(M_1-1)/(M_2+2)]$. Define
	\[
	\zeta_i=\sum_{j=(i-1)(M_2+2)+1}^{(i-1)(M_2+2)+M_2} \xi_j,\qquad i=1,\ldots,M_3
	\]
	\[
	\gamma_i=\xi_{iM_2+2i-1}+\xi_{iM_2+2i},\qquad\gamma^*=\sum_{j=M_3(M_2+2)+1}^{M_1-1}\xi_j.
	\]
	With this notation we can write
	\[
	\sum_{i=1}^{M_1-1}\xi_i=\sum_{i=1}^{M_3}\zeta_i+\sum_{i=1}^{M_3}\gamma_i+\gamma^*
	\]
	and we aim at showing that the first term in the right hand side of the previous equation determines the asymptotic distribution of $\sum_{i=0}^{M_1-1}\xi_i$. 
	
	Note that 
	\[
	b^{-1/2}\int_{t-1}^{t+1}k(t-y)\,\dd W(L(by))\sim N\left(0,\sigma_t^2\right)
	\]
	where
	\[
	\sigma_t^2=\int_{t-1}^{t+1} k^2(t-y)l(by)\,\dd y=D^2l(bt)+O(b^2)
	\]
	and 
	\[
	\begin{split}
	\E\left[\left|b^{-1/2}\int_{t-1}^{t+1}k(t-y)\,\dd W(L(by))+g_{(n)}(tb) \right|^p \right]&=\int_{-\infty}^{+\infty}\left|\sigma_t x+g_{(n)}(tb)\right|^p \phi(x)\,\dd x\\
	&=\int_{-\infty}^{+\infty}\left|D\sqrt{l(tb)} x+g_{(n)}(tb)\right|^p \phi(x)\,\dd x+O(b^2).
	\end{split}
	\] 
	Hence, we get $\E[\xi_i]=O(b^2)$ and $\E[\gamma_i]=O(b^2)$. Furthermore,  and, as we did for $\eta$, it can be seen that $\E[\xi_i^2]=O(1)$ and $\E[\gamma_i^2]=O(1)$. 
	
	Since $\gamma_i$ depends only on the Brownian motion on the interval $[L(b(iM_2+2i-2)),L(b(iM_2+2i+2))]$, it follows that $\gamma_i$ are independent (note that $M_2>2$).
	Moreover, $\gamma^*$ is independent of $\gamma_i$, $i=1,\ldots,M_3-1$ and $\E[\gamma^*]=O(M_2b^2)$. In addition, since $\xi_i$ is independent of $\xi_j$ for $|i-j|\geq 3$, we also have $\E[(\gamma^*)^2]\leq CM_2$. As a result
	\begin{equation}
	\label{eqn:bound_gamma_gamma*}
	\E\left[\left(\sum_{i=1}^{M_3}\gamma_i+\gamma^* \right)^2 \right]\leq c(M_3+M_2)=o(1/b)
	\end{equation}
	because $bM_2\to 0$ and $bM_3\to 0$. Indeed $M_2\leq (T/b)^\gamma$ and 
	\[
	b\left[\frac{[(1-b)/b]}{\left[[(1-b)/b]^\gamma \right]+2}\right]\leq\frac{1-b}{[(1-b)/b]^\gamma+1}\leq \frac{1-b}{1+\frac{(1-2b)^\gamma}{b^\gamma}}=\frac{b^\gamma }{(1-2b)^\gamma+b^\gamma}\to 0.
	\]
	Consequently
	\[
	b^{1/2}\left(\sum_{i=1}^{M_3}\gamma_i+\gamma^*\right)\xrightarrow{\p}0.
	\]
	Next, since $\zeta_i$, $i=1,\ldots,M_3$ are independent, we apply the central limit theorem to conclude that 
	\[
	b^{1/2}\gamma(p)^{-1}\sum_{i=0}^{M_3}\zeta_i\to N(0,1)
	\]
	It suffices to show that
	\begin{equation}
	\label{eqn:limit_variance}
	b\E\left[\left(\sum_{i=1}^{M_3}\zeta_i\right)^2 \right]=b\sum_{i=0}^{M_3}\E[\zeta_i^2]\to\gamma^2(p).
	\end{equation}
	and that they satisfy the Lyapunov's condition
	\[
	\frac{\sum_i\E[\zeta_i^4]}{\left(\sum_i\E[\zeta_i^2] \right)^2}\to 0.
	\]
	Note that, once we have~\eqref{eqn:limit_variance}, the Lyapunov's condition is equivalent to $b^2\sum_i\E[\zeta_i^4]\to 0$. 
	Using
	\[
	\E[\zeta_i^4]=4!\sum_{\underset{k\leq l\leq m\leq r }{k,l,m,r\in I_i}}\E[\xi_k\xi_l\xi_m\xi_r],\qquad I_i=\{(i-1)(M_2+2)+1,\ldots,(i-1)(M_2+2)+M_2\},
	\]
	the fact that  
	\[
	\E[\xi_k\xi_l\xi_m\xi_r]=O(b^2)^4 \qquad\text{if}\qquad l\geq k+3\text{ or } r\geq m+3
	\]
	and that all the moments of the $\xi_i$'s are finite, we obtain that 
	\begin{equation}
	\label{eqn:fourth_moment_zeta}
	\E[\zeta_i^4]=O(M_2^2),\qquad\text{(uniformly w.r.t. $i$).}
	\end{equation}
 Consequently $b^2\sum_i\E[\zeta_i^4]=O(b^2M_3M^2_2)\to 0$ because $bM_2\to 0$ and $bM_3M_2=O(1)$. Indeed
	\[
	bM_2M_3\leq bM_2\frac{M_1-1}{M_2+2}\leq bM_1\leq 1.
	\]
	In particular, it also follows that
	\begin{equation}
	\label{eqn:bound_second_moment_zeta}
	b\sum_i\E[\zeta_i^2]=b\E\left[\left(\sum_{i=0}^{M_3}\zeta_i\right)^2 \right]+bO(M_3^2M_2^2b^4)=O(bM_3M_2)=O(1).
	\end{equation}
	Now we prove~\eqref{eqn:limit_variance}. From~\eqref{eqn:sum_representation}, it follows that
	\[
	\mathrm{Var}\left(\int_b^{1-b}\left|b^{-1/2}\Gamma^{(1)}_n(t)+g_{(n)}(t) \right|^p\,\dd \mu(t) \right)=b^2 \mathrm{Var}\left(\sum_{i=1}^{M_1-1}\xi_i+\eta \right).
	\]
	Moreover, since $\E[\xi_i]=O(b^2)$ for $i=1,\dots,M_1-1$ and $\E[\eta]=0$, we get
	\[
	\begin{split}
	&b^{-1}\mathrm{Var}\left(\int_b^{1-b}\left|l(t)b^{-1/2}\Gamma^{(1)}_n(t)+g_{(n)}(t) \right|^p\,\dd \mu(t) \right)\\
	&=b\E\left[\left(\sum_{i=1}^{M_1-1}\xi_i+\eta \right)^2\right]+o(1)\\
	&=b\E[\eta^2]+2b\E\left[\left(\sum_{i=1}^{M_1-1}\xi_i\right)\eta\right]+b\E\left[\left(\sum_{i=1}^{M_1-1}\xi_i\right)^2\right]+o(1)
	\end{split}
	\] 
	We have already shown in the proof of the previous lemma that $\E[\eta^2]=O(1)$, so the first term in the right hand side of the previous equation converges to zero. Furthermore,
	\[
	\begin{split}
	b\E\left[\left(\sum_{i=1}^{M_1-1}\xi_i\right)^2\right]&=b\E\left[\left(\sum_{i=1}^{M_3}\zeta_i+\sum_{i=1}^{M_3}\gamma_i+\gamma^*\right)^2\right]\\
	&=b\E\left[\left(\sum_{i=1}^{M_3}\zeta_i\right)^2\right]+b\E\left[\left(\sum_{i=1}^{M_3}\gamma_i+\gamma^*\right)^2\right]+b\E\left[\left(\sum_{i=1}^{M_3}\zeta_i\right)\left(\sum_{i=1}^{M_3}\gamma_i+\gamma^*\right)\right].
	\end{split}
	\]  
	Now, making use of~\eqref{eqn:bound_gamma_gamma*},~\eqref{eqn:bound_second_moment_zeta}
	and the fact that, by Cauchy-Schwartz, 
	\[
	\E\left[\left(\sum_{i=1}^{M_3}\zeta_i\right)\left(\sum_{i=1}^{M_3}\gamma_i+\gamma^*\right)\right]\leq \E\left[\left(\sum_{i=1}^{M_3}\zeta_i\right)^2\right]^{1/2}\E\left[\left(\sum_{i=1}^{M_3}\gamma_i+\gamma^*\right)^2\right]^{1/2}
	\]
	we obtain
	\[
	b\E\left[\left(\sum_{i=1}^{M_1-1}\xi_i\right)^2\right]=b\E\left[\left(\sum_{i=1}^{M_3}\zeta_i\right)^2\right]+o(1).
	\]
	Similarly,
	\[
	\begin{split}
	b\E\left[\left(\sum_{i=1}^{M_1-1}\xi_i\right)\eta\right]&=b\E\left[\left(\sum_{i=1}^{M_3}\zeta_i+\sum_{i=1}^{M_3}\gamma_i+\gamma^*\right)\eta\right]\\
	&\leq b\E[\eta^2]^{1/2}\left\{\E\left[\left(\sum_{i=1}^{M_3}\zeta_i\right)^2\right]^{1/2}+\E\left[\left(\sum_{i=1}^{M_3}\gamma_i+\gamma^*\right)^2\right]^{1/2}\right\}\to 0.
	\end{split}
	\]  
	This means that 
	\[
	\begin{split}
	b\E\left[\left(\sum_{i=1}^{M_3}\zeta_i\right)^2 \right]=b^{-1}\mathrm{Var}\left(\int_b^{1-b}\left|b^{-1/2}\Gamma^{(1)}_n(t)+g_{(n)}(t) \right|^p\,\dd \mu(t) \right)+o(1).
	\end{split}
	\]
	Moreover, from Lemma~\ref{lemma:p-moment}, it follows that
	{\small \[
		\begin{split}
		&b^{-1} \mathrm{Var}\left(\int_b^{1-b}\left|b^{-1/2}\Gamma^{(1)}_n(t)+g_{(n)}(t) \right|^p\,\dd \mu(t) \right)\\
		&=\frac{1}{b}\int_b^{1-b}\int_b^{1-b}\left\{\E\left[\left|b^{-1/2}\int_{t-b}^{t+b}k\left(\frac{t-y}{b}\right)\,\dd W(L(y))+g_{(n)}(t) \right|^p\left|b^{-1/2}\int_{u-b}^{u+b}k\left(\frac{u-y}{b}\right)\,\dd W(L(y))+g_{(n)}(u) \right|^p\right] \right.\\
		&\quad\left.-\int_{-\infty}^{+\infty}\left|\sigma_n(t) x+g_{(n)}(t)\right|^p \phi(x)\,\dd x\int_{-\infty}^{+\infty}\left|\sigma_n(u) y+g_{(n)}(u)\right|^p \phi(y)\,\dd y\right\} w(t)w(u)\,\dd t\,\dd u\\
		&=\frac{1}{b}\int_b^{1-b}\int_b^{1-b}\int_{\R}\int_{\R}\left\{\left|\sigma_n(u) y+g_{(n)}(u) \right|^p\left| g_{(n)}(t)+\sigma_n(t)\rho_n(t,u)y+\sqrt{1-\rho^2_n(t,u)}\sigma_n(t)x \right|^p \right.\\
		&\quad-\left|\sigma_n(t) x+g_{(n)}(t)\right|^p \left|\sigma_n(u) y+g_{(n)}(u)\right|^p \bigg\} w(t)w(u)\phi(x)\phi(y)\,\dd x\,\dd y\,\dd t\,\dd u\\
		&=\frac{1}{b}\int_b^{1-b}\int_b^{1-b}\int_{\R}\int_{\R}\left\{\left|\sqrt{L'(u)}D y+g_{(n)}(u) \right|^p\left| g_{(n)}(t)+\sigma_n(t)\rho_n(t,u)y+\sqrt{1-\rho^2_n(t,u)}\sigma_n(t)x \right|^p \right.\\
		&\quad-\left|\sqrt{L'(t)}D x+g_{(n)}(t)\right|^p \left|\sqrt{L'(u)}D y+g_{(n)}(u)\right|^p \bigg\} w(t)w(u)\phi(x)\phi(y)\,\dd x\,\dd y\,\dd t\,\dd u\\
		\end{split}
		\]}
	where $\rho_n(t,u)$ and $\sigma_n(t)$ are defined respectively in~\eqref{eqn:def-rho} and~\eqref{eqn:def-sigma(t)}.
	
	First we consider the case $nb^5\to 0$ and show that we can remove the $g_{(n)}$ functions from the previous integral. Indeed, since
	\[
	\left|\left|\sqrt{L'(u)}Dy+g_{(n)}(u) \right|^p-|\sqrt{L'(u)}Dy|^p\right|\leq p2^{p-1}|g_{(n)}(u)|^p+p2^{p-1}|\sqrt{L'(u)}Dy|^{p-1}|g_{(n)}(u)|
	\]
	we obtain
	\[
	\begin{split}
	&\left|A_n-\frac{1}{b}\int_{b}^{1-b}\int_{b}^{1-b}\int_{\R}\int_{\R}\left|\sqrt{L'(u)}Dy \right|^pB_n(t,u,x,y)w(t)w(u)\phi(x)\phi(y)\,\dd x\,\dd y\,\dd t\,\dd u\right|\\
	&\leq \frac{c}{b}\int_{b}^{1-b}\int_{b}^{1-b}\int_{\R}\int_{\R}\left|g_{(n)}(u) \right|^p|B_n(t,u,x,y)|w(t)w(u)\phi(x)\phi(y)\,\dd x\,\dd y\,\dd t\,\dd u\\
	&\qquad+\frac{c}{b}\int_{b}^{1-b}\int_{b}^{1-b}\int_{\R}\int_{\R}\left|\sqrt{L'(u)}Dy\right|^{p-1}\left|g_{(n)}(u) \right||B_n(t,u,x,y)|w(t)w(u)\phi(x)\phi(y)\,\dd x\,\dd y\,\dd t\,\dd u,
	\end{split}
	\]
	where
	\[
	A_n=b^{-1} \mathrm{Var}\left(\int_b^{1-b}\left|b^{-1/2}\Gamma^{(1)}_n(t)+g_{(n)}(t) \right|^p\,\dd \mu(t) \right)
	\]
	and
	\[
	B_n(t,u,x,y)=\left| g_{(n)}(t)+\sigma_n(t)\rho_n(t,u)y+\sqrt{1-\rho^2_n(t,u)}\sigma_n(t)x \right|^p -\left|\sqrt{L'(t)} D x+g_{(n)}(t)\right|^p.
	\]
	Note that, if $|t-u|\geq2b$, then $\rho_n(t,u)=0$ and the previous integrands are equal to zero. Hence, a sufficient condition for the left hand side of the previous inequality to converge to zero is to have
	\[
	b^{-1}\int_{b}^{1-b}\int_{b}^{1-b}\1_{\{|t-u|<2b\}}\left|g_{(n)}(u) \right|^p\left|g_{(n)}(t) \right|^p\,\dd u\,\dd t\to 0.
	\]
	and
	\[
	b^{-1}\int_{b}^{1-b}\int_{b}^{1-b}\1_{\{|t-u|<2b\}}\left|g_{(n)}(u) \right|^p\,\dd u\,\dd t\to 0.
	\]
	This is indeed the case because $g_n(u)=O\left((nb)^{1/2}b^2\right)$ uniformly w.r.t. $u$ and $(nb)^{1/2}b^2\to 0$. In the same way we can remove also the other $g_{(n)}$ functions from the integrand, i.e.
	\[
	A_n=\frac{1}{b}\int_{b}^{1-b}\int_{b}^{1-b}\int_{\R}\int_{\R}\left|\sqrt{L'(u)}Dy \right|^pB'_n(t,u,x,y) w(t)w(u)\phi(x)\phi(y)\,\dd x\,\dd y\,\dd t\,\dd u+o(1)
	\]
	where
	\[
	B'_n(t,u,x,y)=\left| \sigma_n(t)\rho_n(t,u)y+\sqrt{1-\rho^2_n(t,u)}\sigma_n(t)x \right|^p -\left|\sqrt{L'(t)} D x\right|^p 
	\]
	With the change of variable $t=u+sb$, we get
	\[
	\begin{split}
	A_n&=\int_{b}^{1-b}\int\limits_{\substack{1-u/b \\ |s|\leq 2}}^{(1-b-u)/b}\int_{\R}\int_{\R}\left|\sqrt{L'(u)}\sqrt{L'(u+sb)}D^2y\right|^p\left\{\left| yr(s)+\sqrt{1-r^2(s)}x \right|^p-\left|x\right|^p \right\}\\
	&\qquad w(u)w(u+sb)\phi(x)\phi(y)\,\dd x\,\dd y\,\dd s\,\dd u+o(1),
	\end{split}
	\]
	where $r(s)$ is defined in~\eqref{eqn:r(s)}.
	The continuity of the functions $l$ and $w$  and the dominated convergence theorem yield
	\[
	A_n=\int_{b}^{1-b}\int\limits_{\substack{|s|\leq 2}}\int_{\R^2}\left|\sqrt{L'(u)}\right|^{2p}D^{2p}|y|^p\left\{\left| yr(s)+\sqrt{1-r^2(s)}x \right|^p-\left|x\right|^p \right\} w(u)^2\phi(x)\phi(y)\,\dd x\,\dd y\,\dd s\,\dd u+o(1).
	\] 
	Then, with the change of variable $yr(s)+\sqrt{1-r^2(s)}x=z$ we can write equivalently
	\[
	\begin{split}
	A_n&=D^{2p}\int_{b}^{1-b}\left|\sqrt{L'(u)}\right|^{2p}w(u)^2\,\dd u \frac{1}{2\pi}\\
	&\qquad\int_{\R^3}|y|^p\left\{\left|z\right|^p-\left|\frac{z-r(s)y}{\sqrt{1-r^2(s)}}\right|^p \right\} e^{-\frac{z^2+y^2-2rzy}{2(1-r^2(s))}}\frac{1}{\sqrt{1-r^2(s)}}\,\dd z\,\dd y\,\dd s+o(1)\\
	&=\sigma_1D^{2p}\int_{0}^{1}\left|\sqrt{L'(u)}\right|^{2p}w(u)^2\,\dd u +o(1)
	\end{split}
	\]
	where $\sigma^1$ is defined in~\eqref{eqn:sigma_1}.
	
	Let us now consider the case $nb^5\to c_0^2>0.$ First we show that the $g_{(n)}(u)$ functions can be replaced by $g(u)$ defined in \eqref{eqn:def-g(u)}. Indeed, $g_{(n)}(u)=g(u)+o((nb)^{1/2}b^2)$, where the big O term is uniform w.r.t. $u$ and similar calculations to those of the previous case allow us to conclude that
	\[
	A_n=\frac{1}{b}\int_{b}^{1-b}\int_{b}^{1-b}\int_{\R}\int_{\R}\left|\sqrt{L'(u)}Dy+g(u) \right|^pB'_n(t,u,x,y) w(t)w(u)\phi(x)\phi(y)\,\dd x\,\dd y\,\dd t\,\dd u+o(1)
	\]
	where
	\[
	B'_n(t,u,x,y)=\left|g(t)+ \sqrt{L'(t)}D\left[\rho_n(t,u)y+\sqrt{1-\rho^2_n(t,u)}x\right] \right|^p -\left|\sqrt{L'(t)} D x+g(t)\right|^p.
	\]
	With the change of variable $t=u+sb$, we get
	{\small \[
		\begin{split}
		A_n&=\int_{b}^{1-b}\int\limits_{\substack{(b-u)/b \\ |s|\leq 2}}^{(1-b-u)/b}\int_{\R}\int_{\R}\left|g(u)+\sqrt{L'(u)}Dy\right|^p\left\{\left|g(u+sb)+\sqrt{L'(u+sb)}D[ yr(s)+\sqrt{1-r^2(s)}x ]\right|^p\right.\\
		&\quad\left.-\left|g(u+sb)+\sqrt{L'(u+sb)}Dx\right|^p \right\} w(u)w(u+sb)\phi(x)\phi(y)\,\dd x\,\dd y\,\dd s\,\dd u+o(1).
		\end{split}
		\]}
	Again, by the continuity of the functions $l$, $w$ and $g$  and the dominated convergence theorem we obtain that $A_n$ converges to 
	 \[
	 \begin{split}
		&\int_{0}^{1}\int_{\R^3}\left|g(u)+\sqrt{L'(u)}Dy\right|^{p}\left\{\left|g(u)+\sqrt{L'(u)}D[ yr(s)+\sqrt{1-r^2(s)}x ]\right|^p\right.\\
		&\qquad-\left.\left|g(u)+\sqrt{L'(u)}Dx\right|^p \right\} w(u)^2\phi(x)\phi(y)\,\dd x\,\dd y\,\dd s\,\dd u,
	\end{split}
		\] 
	which is exactly $\theta^2(p)$ defined in~\eqref{eqn:def-theta}.
\end{proof}
\begin{lemma}
	\label{lemma:p-moment}
	Let $l(t)$ be a differentiable function on $[0,1]$ such that $\inf_{[0,1]}l(t)>0$ and $\sup_{[0,1]}|l'(t)|<\infty$. Define $L(t)=\int_0^t l(u)\,\dd u$. For $t\in[0,1]$, define
	\[
	X_{n,t}=b^{-1/2}\int_{t-b}^{t+b}k\left(\frac{t-y}{b}\right)\,\dd W(L(y))+g_{(n)}(t).
	\]
	It holds
{\small	\[
	\E\left[ |X_{n,t}X_{n,u}|^p\right]=\int_{\R}\int_{\R}\left| \sigma_n(u) y+g_{(n)}(u)\right|^p\left|g_{(n)}(t)+\sigma_n(t)\rho_n(t,u)y+\sqrt{1-\rho_n^2(t,u)}\sigma_n(t)x\right|^p\phi(x) \phi(y)\,\dd x\,\dd y,
	\]}
	where
	\begin{equation}
	\label{eqn:def-sigma(t)}
	\sigma^2_n(t)=l(t)D^2+O(b^2),\qquad\sigma_n(t,u)=b^{-1}\int k(t-y)k(u-y)l(y)\,\dd y.
	\end{equation}
	and
	\begin{equation}
	\label{eqn:def-rho}
	\rho_n(t,u)=\frac{\int k\left(\frac{t-y}{b}\right)k\left(\frac{u-y}{b}\right)l(y)\,\dd y}{b\sqrt{D^2l(t)+O(b^2)}\sqrt{D^2l(u)+O(b^2)}}
	\end{equation}
\end{lemma}
\begin{proof}
	First, note that
	\[
	(X_{n,t},X_{n,u})\sim N\left(\begin{bmatrix}
	g_{(n)}(t) \\ g_{(n)}(u)\end{bmatrix},\begin{bmatrix}
	\sigma^2_n(t) && \sigma_n(t,u) \\ \sigma_n(t,u) && \sigma^2_n(u)\end{bmatrix}\right).
	\]
	Hence, we have
	\[
	X_{n,t}|X_{n,u}=x_2\sim N\left(g_{(n)}(t)+\frac{\sigma_n(t)}{\sigma_n(u)}\rho_n(t,u)\left(x_2-g_{(n)}(u)\right),(1-\rho^2_n(t,u))\sigma^2_n (t)\right).
	\]
	Consequently, we obtain
	\[
	\begin{split}
	&\E\left[ |X_{n,t}X_{n,u}|^p\right]\\
	&=\E\left[\E\left[ |X_{n,t}X_{n,u}|^p|X_{n,u}\right]\right]\\
	&=\E\left[ |X_{n,u}|^p\int_{\R}\left|g_{(n)}(t)+\frac{\sigma_n(t)}{\sigma_n(u)}\rho_n(t,u)\left(X_{n,u}-g_{(n)}(u)\right)+\sqrt{1-\rho^2_n(t,u)}\sigma_n(t)x\right|^p\phi(x)\,\dd x \right]\\
	&=\int_{\R}\left| \sigma_n(u)y+g_{(n)}(u)\right|^p\int_{\R}\left|g_{(n)}(t)+\sigma_n(t)\rho_n(t,u)y+\sqrt{1-\rho^2_n(t,u)}\sigma_n(t)x\right|^p\phi(x)\,\dd x\, \phi(y)\,\dd y\\
	&=\int_{\R}\int_{\R}\left| \sigma_n(u)y+g_{(n)}(u)\right|^p\left|g_{(n)}(t)+\sigma_n(t)\rho_n(t,u)y+\sqrt{1-\rho^2_n(t,u)}\sigma_n(t)x\right|^p\phi(x) \phi(y)\,\dd x\,\dd y.
	\end{split}
	\]
\end{proof}

\begin{proof}[Proof of Proposition~\ref{prop:boundaries}]
	We first prove (i).
	For each $t\in[0,b)$, we have
	\[
	\begin{split}
	&
	\tilde{\lambda}^s_n(t)-\lambda(t)
	=
	\int_0^{t+b}k_b(t-u)\,\dd \Lambda_n(u)-\lambda(t)\\
	&\quad=
	\int_0^{t+b}k_b(t-u)\,\dd (\Lambda_n-\Lambda)(u)+\int_0^{t+b} k_b(t-u)\,\dd \Lambda(u)-\lambda(t)\\
	&\quad=
	\int_0^{t+b}k_b(t-u)\,\dd (\Lambda_n-\Lambda)(u)+\int_{-1}^{t/b}k(y)[\lambda(t-by)-\lambda(t)]\,\dd y
	-
	\lambda(t)\int_{t/b}^{1}k(y)\,\dd y.
	\end{split}
	\]
	Note that
	{
		\begin{equation}
		\label{eq:bound integral Mn}
		\begin{split}
		&
		\left|\int_0^{t+b}k_b(t-u)\,\dd(\Lambda_n-\Lambda)(u)\right|\\
		&\quad=
		\frac{1}{b^2}\left|\int_0^{t+b}(\Lambda_n-\Lambda)(u)k'\left(\frac{t-u}{b}\right)\,\dd u\right|\\
		&\quad\leq
		cb^{-1}\sup_{u\leq 2b}\left|M_n(u)-M_n(0)\right|\\
		&\quad\leq
		cb^{-1}\left\{\sup_{u\leq 2b}\left|M_n(u)-n^{-1/2}B_n\circ L(u)\right|+n^{-1/2}\left|B_n\circ L(u)-B_n\circ L(0) \right|\right\}\\
		&\quad=
		O_P\left(b^{-1}n^{-1+1/q}\right)+n^{-1/2}b^{-1}\sup_{y\leq cb}|B_n(y)|
		=
		O_P\left((nb)^{-1/2}\right),
		\end{split}
		\end{equation}
		uniformly in $t\in[0,b]$, and that according to~\eqref{eq:order term1},
		\[
		\left|\int_{-1}^{t/b}k(y)[\lambda(t-by)-\lambda(t)]\,\dd y\right|
		=O(b),
		\]
		Moreover, for} $t\leq b/2$,
	\[
	\lambda(t)\int_{t/b}^{1}k(y)\,\dd y\geq \inf_{t\in[0,1]}\lambda(t)\int_{1/2}^{1}k(y)\,\dd y=C>0.
	\]
	{Now, define} the event
	\[
	\begin{split}
	\A_n
	=
	\Bigg\{
	\sup_{t\in[0,b]}
	\bigg(
	\bigg|\int_0^{t+b}&k_b(t-u)\,\dd (\Lambda_n-\Lambda)(u)\bigg|\\
	&+
	\bigg|
	\int_{-1}^{t/b}k(y)[\lambda(t-by)-\lambda(t)]\,\dd y
	\bigg|\bigg)
	\leq C/2
	\Bigg\}.
	\end{split}
	\]
	Then, $\p(\A_n)\to 1$ and on the event $\A_n$, $|\tilde{\lambda}^s_n(t)-\lambda(t)|\geq C/2$.
	Consequently we obtain
	\begin{equation}
	\label{eqn:expectation}
	\begin{split}
	\E\left[\int_0^{b}\left|\tilde{\lambda}^s_n(t)-\lambda(t)\right|^p\,\dd\mu(t)\right]&\geq \E\left[\int_0^{b/2}\left|\tilde{\lambda}^s_n(t)-\lambda(t)\right|^p\,\dd\mu(t)\right]\\
	&\geq\E\left[\1_{\A_n}\int_0^{b/2}\left|\hat{\lambda}^s_n(t)-\lambda(t)\right|^p\,\dd\mu(t)\right]
	\geq
	c\p(\A_n)b,
	\end{split}
	\end{equation}
	{for some $c>0$.}
	Hence
	\[
	(nb)^{p/2}\E\left[\int_0^{b}\left|\tilde{\lambda}^s_n(t)-\lambda(t)\right|^p\,\dd\mu(t)\right]\geq cb(nb)^{p/2}\p(\A_n)\to\infty,
	\]
	{because $b(nb)^{p/2}\geq b(nb)^{1/2}=(nb^3)^{1/2}\to \infty$.}
	
	In order to prove (ii), due to~\eqref{eqn:L_p-inequality}, we can bound
	\[
	b^{-1/2}
	\left|
	\int_0^b(nb)^{p/2}\left|\tilde{\lambda}^s_n(t)-\lambda(t)\right|^p\,\dd \mu(t)-\int_0^b\left| g_{(n)}(t)\right|^p\,\dd \mu(t)
	\right|
	\]
	by
	\[
	\begin{split}
	&
	p2^{p-1}b^{-1/2}(nb)^{p/2}\int_0^b \left|\int_0^{t+b}k_b(t-u)\,\dd(\Lambda_n-\Lambda)(u)\right|^p\,\dd \mu(t)\\
	&\quad+
	p2^{p-1}b^{-1/2}
	\left(
	\int_0^b (nb)^{p/2}\left|\int_0^{t+b}k_b(t-u)\,\dd(\Lambda_n-\Lambda)(u)\right|^p\,\dd \mu(t)
	\right)^{1/p}\\
	&\qquad\qquad\qquad\qquad\qquad\qquad\cdot
	\left(
	\int_0^b\left|g_{(n)}(t) \right|^p\,\dd \mu(t)
	\right)^{1-1/p}.
	\end{split}
	\]
	{According to~\eqref{eq:bound integral Mn}
		\[
		\left|\int_0^{t+b}k_b(t-u)\,\dd(\Lambda_n-\Lambda)(u)\right|\\
		=
		O_P\left((nb)^{-1/2}\right),
		\]
		uniformly in $t\in[0,b]$.
		Furthermore, using~\eqref{eqn:bias_g_n}, \eqref{eq:order term1}, and~\eqref{eq:order term2},
		we have
		\begin{equation}
		\label{eq:order gn}
		g_{(n)}(t)
		=
		O\left((nb)^{1/2}\right),
		\end{equation}
		uniformly for $t\in[0,b]$.
		Hence,} we obtain
	\[
	\begin{split}
	&
	b^{-1/2}
	\left|
	\int_0^b (nb)^{p/2}\left|\tilde{\lambda}^s_n(t)-\lambda(t)\right|^p\,\dd \mu(t)
	-
	\int_0^b\left|g_{(n)}(t) \right|^p\,\dd \mu(t)
	\right|\\
	&\quad\leq
	O_P\left(b^{1/2}\right)+O_P\left(n^{(p-1)/2}b^{p/2}\right)\to 0,
	\end{split}
	\]
	{because $n^{(p-1)/2}b^{p/2}=(bn^{1-1/p})^{p/2}\to0$.}
	
	Next we deal with (iii).
	{Again by means of~\eqref{eqn:L_p-inequality}, we can bound
		\[
		b^{-1/2}\left|\int_0^b(nb)^{p/2}\left|\tilde{\lambda}^s_n(t)-\lambda(t)\right|^p\,\dd \mu(t)
		-
		\int_0^b\left|Y_n(t)+g_{(n)}(t)\right|^p\,\dd \mu(t) \right|
		\]
		by
		\[
		\begin{split}
		&
		p2^{p-1}b^{-1/2}\int_0^b \left|(nb)^{1/2}\int_0^{t+b}k_b(t-u)\,\dd(\Lambda_n-\Lambda-n^{-1/2}B_n\circ L)(u)\right|^p\,\dd \mu(t)\\
		&\qquad+
		p2^{p-1}b^{-1/2}
		\left(\int_0^b \left|(nb)^{1/2}\int_0^{t+b}k_b(t-u)\,
		\dd(\Lambda_n-\Lambda-n^{-1/2}B_n\circ L)(u)\right|^p\,\dd \mu(t)
		\right)^{1/p}\\
		&\qquad\qquad\qquad\qquad\qquad\qquad\cdot
		\left( \int_0^b\left|Y_n(t) +g_{(n)}(t) \right|^p\,\dd \mu(t) \right)^{1-1/p}
		\end{split}
		\]
		Note that
		\[
		\sup_{t\in[0,b]}|Y_n(t)|
		=
		\sup_{t\in[0,b]}\left|b^{1/2}
		\int_0^{t+b}
		k_b(t-u)\,\dd B_n(L(u))\right|
		=O_P(1),
		\]
		and, as in~\eqref{eq:bound integral Mn},
		\[
		\begin{split}
		&
		\left|\int_0^{t+b}k_b(t-u)\,\dd(\Lambda_n-\Lambda-n^{-1/2}B_n\circ L)(u)\right|\\
		&\qquad\leq
		\frac{1}{b}\sup_{u\leq 2b}\left|(\Lambda_n-\Lambda-n^{-1/2}B_n\circ L)(u) \right|
		=
		O_P\left(b^{-1} n^{-1+1/q}\right),
		\end{split}
		\]
		uniformly for $t\in[0,b]$.
		Together with~\eqref{eq:order gn}, we obtain}
	\[
	\begin{split}
	&
	b^{-1/2}\left|\int_0^b(nb)^{p/2}\left|\tilde{\lambda}^s_n(t)-\lambda(t)\right|^p\,\dd \mu(t)
	-
	\int_0^b\left|Y_n(t)+g_{(n)}(t)\right|^p\,\dd \mu(t) \right|\\
	&\quad\leq
	O_P\left(
	b^{-1/2}(nb)^{p/2}b
	n^{-p+p/q}b^{-p}
	\right)
	+
	O_P\left(b^{-1/2}b(nb)^{p/2}n^{-1+1/q}b^{-1}\right)\\
	&\quad=
	{b^{-1/2}(nb)^{p/2}n^{-1+1/q}
		\left\{
		O_P\left((n^{-1+1/q}b^{-1})^{p-1}\right)
		+
		O_P(1)
		\right\}.}
	\end{split}
	\]
	{Because $n^{-1+1/q}b^{-1}=O(1)$, the term within the brackets is of order $O_P(1)$,
		and since $b^{p-1}n^{p-2+2/q}\to0$, the right hand side tends to zero.
		This proves~\eqref{eqn:approx_boundaries}.}
	
	Then, by Jensen's inequality, we get
	\begin{equation}
	\label{eqn:infinite_variance}
	\begin{split}
	&
	b^{-1}\mathrm{Var}\left(\int_0^{cb}|Y_n(t)+g_{(n)}(t)|^p\,\dd \mu(t)\right)\\
	&\quad=
	b^{-1}\E\left[\left(\int_0^{cb}|Y_n(t)+g_{(n)}(t)|^p\,\dd \mu(t)-\int_0^{cb}\E\left[|Y_n(t)+g_{(n)}(t)|^p\right]\,\dd \mu(t) \right)^2 \right]\\
	&\quad\geq
	b^{-1}\E\left[\left|\int_0^{cb}|Y_n(t)+g_{(n)}(t)|^p\,\dd \mu(t)-\int_0^{cb}\E\left[|Y_n(t)+g_{(n)}(t)|^p\right]\,\dd \mu(t) \right| \right]^2.
	\end{split}
	\end{equation}
	Note that $Y_n(t)\sim N(0,\sigma_n^2(t))$,
	{where,
		\[
		\sigma_n^2(t)=b^{-1}\int_0^{t+b} k^2\left(\frac{t-u}{b}\right)L'(u)\,\dd u=\int_{-1}^{t/b}k^2(y)L'(t-by)\,\dd y,
		\]
		if $B_n$ is a Brownian motion, and
		\[
		\sigma_n^2(t)=\int_{-1}^{t/b}k^2(y)L'(t-by)\,\dd y+O(b),
		\]
		if $B_n$ is a Brownian bridge.
		Now, choose $\epsilon>0$.
		Then}
	\[
	\liminf_{{n\to\infty}}
	\p\left(\epsilon\leq Y_n(0)\leq 2\epsilon\right)>0
	\qquad\text{and}\qquad
	\liminf_{{n\to\infty}}
	\p\left(-2\epsilon\leq Y_n(0)\leq -\epsilon\right)>0.
	\]
	{For $c>0$, define the events
		\[
		\begin{split}
		\A_{n1}
		&=
		\left\{\epsilon/2\leq Y_n(t)\leq 3\epsilon,\text{ for all } t\in[0,cb]\right\},\\
		{\A_{n2}}
		&=
		\left\{-3\epsilon\leq Y_n(t)\leq -\epsilon/2,\text{ for all } t\in[0,cb]\right\},
		\end{split}
		\]
		and let
		\[
		\B_n=\left\{\int_0^{cb}\E\left[|Y_n(t)+g_{(n)}(t)|^p\right]\,\dd \mu(t)> \int_0^{cb}|g_{(n)}(t)|^p\,\dd \mu(t) \right\}.
		\]}
	{Then,} since $Y_n$ has continuous paths, we have
	\[
	\liminf_{{n\to\infty}}\p({\A_{n1}})>0
	\qquad\text{and}\qquad
	\liminf_{{n\to\infty}}\p({\A_{n2}})>0.
	\]
	Moreover, {$Y_n(t)>0$ on the event ${\A_{n1}}$, and from~\eqref{eq:gn negative},
		it follows that $Y_n(t)+g_{(n)}(t)<0$, for~$n$ sufficiently large.
		Therefore, for $n$ sufficiently large, we have on $\A_{n1}$,
		\begin{equation}
		\label{eq:ineq An1}
		\int_0^{cb}|Y_n(t)+g_{(n)}(t)|^p\,\dd \mu(t)\leq \int_0^{cb}|\epsilon/2+g_{(n)}(t)|^p\,\dd \mu(t).
		\end{equation}
		Similarly, $Y_n(t)<0$ on the event $\A_{n2}$ and $Y_n(t)+g_{(n)}(t)<0$, for large $n$,
		so that on $\A_{n2}$,
		\begin{equation}
		\label{eq:ineq An2}
		\int_0^{cb}|Y_n(t)+g_{(n)}(t)|^p\,\dd \mu(t)\geq \int_0^{cb}|-\epsilon/2+g_{(n)}(t)|^p\,\dd \mu(t).
		\end{equation}
		Next, write
		\begin{equation}
		\label{eq:bound}
		\begin{split}
		&
		\E\left[\left|
		\int_0^{cb}|Y_n(t)+g_{(n)}(t)|^p\,\dd \mu(t)
		-
		\int_0^{cb}\E\left[|Y_n(t)+g_{(n)}(t)|^p\right]\,\dd \mu(t)
		\right|
		\right]\\
		&\geq
		\E\left[
		\left|
		\int_0^{cb}|Y_n(t)+g_{(n)}(t)|^p\,\dd \mu(t)
		-
		\int_0^{cb}
		\E\left[|Y_n(t)+g_{(n)}(t)|^p\right]\,\dd \mu(t)
		\right|\1_{\A_{n1}}
		\right]\1_{\B_n}\\
		&\qquad+
		\E\left[
		\left|\int_0^{cb}|Y_n(t)+g_{(n)}(t)|^p\,\dd \mu(t)
		-
		\int_0^{cb}\E\left[|Y_n(t)+g_{(n)}(t)|^p\right]\,\dd \mu(t)
		\right| \1_{\A_{n2}}
		\right]\1_{\B^c_n}.
		\end{split}
		\end{equation}
		Consider the first term on the right hand side.
		Because for $n$ large, $Y_n(t)+g_{(n)}(t)<0$ on the event $\A_{n1}$, we have $|Y_n(t)+g_{(n)}(t)|\leq |g_{(n)}(t)|$.
		It follows that on the event $\A_{n1}\cap \B_n$:
		\[
		\int_0^{cb}|Y_n(t)+g_{(n)}(t)|^p\,\dd \mu(t)
		\leq
		\int_0^{cb}|g_{(n)}(t)|^p\,\dd \mu(t)
		<
		\int_0^{cb}
		\E\left[|Y_n(t)+g_{(n)}(t)|^p\right]\,\dd \mu(t).
		\]
		This means that we can remove the absolute value signs in the first term on the right hand side of~\eqref{eq:bound}.
		Similarly, $Y_n(t)+g_{(n)}(t)<0$, for $n$ sufficiently large on the event $\A_{n2}$, so that on the event
		$\A_{n2}\cap \B^c_n$:
		\[
		\int_0^{cb}|Y_n(t)+g_{(n)}(t)|^p\,\dd \mu(t)
		\geq
		\int_0^{cb}|g_{(n)}(t)|^p\,\dd \mu(t)
		\geq
		\int_0^{cb}
		\E\left[|Y_n(t)+g_{(n)}(t)|^p\right]\,\dd \mu(t),
		\]
		so that we can also remove the absolute value signs in the second term on the right hand side of~\eqref{eq:bound}.
		It follows that the right hand of~\eqref{eq:bound} is equal to
		\[
		\begin{split}
		&
		\E\left[\int_0^{cb}
		\left(
		\E\left[|Y_n(t)+g_{(n)}(t)|^p\right]\,\dd \mu(t)
		-
		\int_0^{cb}|Y_n(t)+g_{(n)}(t)|^p\,\dd \mu(t)\right)\1_{\A_{n1}} \right]\1_{\B_n}\\
		&\qquad+
		\E\left[\left(\int_0^{cb}|Y_n(t)+g_{(n)}(t)|^p\,\dd \mu(t)
		-
		\int_0^{cb}\E\left[|Y_n(t)+g_{(n)}(t)|^p\right]\,\dd \mu(t) \right) \1_{\A_{n2}}\right]\1_{\B^c_n}\\
		&\geq
		\left(\int_0^{cb}|g_{(n)}(t)|^p\,\dd \mu(t)-\int_0^{cb}|\epsilon/2+g_{(n)}(t)|^p\,\dd \mu(t) \right)
		\p(\A_{n1})\1_{\B_n}\\
		&\qquad
		+\left(\int_0^{cb}|-\epsilon/2+g_{(n)}(t)|^p\,\dd \mu(t)-\int_0^{cb}|g_{(n)}(t)|^p\,\dd \mu(t) \right)
		\p(\A_{n2})\1_{\B^c_n},
		\end{split}
		\]
		by using~\eqref{eq:order term1} and~\eqref{eq:order term2}.
		Furthermore, for the first term on the right hand side
		\[
		|g_{(n)}(t)|^p-|\epsilon/2+g_{(n)}(t)|^p
		=
		|g_{(n)}(t)|^p
		\left(
		1-\left|\epsilon_n(t)+1\right|^p
		\right),
		\]
		where $\epsilon_n(t)=\epsilon/(2g_{(n)}(t))=O((nb)^{-1/2})\to0$,
		due to~\eqref{eqn:bias_g_n},
		\eqref{eq:order term1} and~\eqref{eq:order term2}, where the big-O term is uniformly for $t\in[0,b]$.
		This means that, for $n$ large, $1+\epsilon_n(t)>0$, and
		by a Taylor expansion $|1+\epsilon_n(t)|^p=1+p\epsilon_n(t)+O((nb)^{-1})$.
		It follows that
		\[
		\begin{split}
		&
		\int_0^{cb}|g_{(n)}(t)|^p\,\dd \mu(t)-\int_0^{cb}|\epsilon/2+g_{(n)}(t)|^p\,\dd \mu(t)\\
		&=
		\int_0^{cb}|g_{(n)}(t)|^p\left\{1-\left|\epsilon_n(t)+1\right|^p\right\}\,\dd \mu(t)\\
		&=
		-p\int_0^{cb}|g_{(n)}(t)|^p\epsilon_n(t)\,\dd \mu(t)
		+
		cb\sup_{t\in[0,cb]}|g_{(n)}(t)|^pO((nb)^{-1})\\
		&=
		p(\epsilon/2)\int_0^{cb}|g_{(n)}(t)|^{p-1}\,\dd \mu(t)+O\left(b(nb)^{(p-1)/2}\right)\\
		&=
		O\left(b(nb)^{(p-1)/2}\right)
		\end{split}
		\]
		due to~\eqref{eq:order gn}.
		Similarly
		\[
		\int_0^{cb}|-\epsilon/2+g_{(n)}(t)|^p\,\dd \mu(t)-\int_0^{cb}|g_{(n)}(t)|^p\,\dd \mu(t)
		=
		O\left(b(nb)^{(p-1)/2}\right).
		\]
		Going back to~\eqref{eqn:infinite_variance},
		since $\p(\A_{n1})\to1$ and $\p(\A_{n2})\to1$, we conclude that
		\[
		b^{-1}\mathrm{Var}\left(\int_0^{cb}|Y_n(t)+g_{(n)}(t)|^p\,\dd \mu(t)\right)
		\geq
		b^{-1}O\left(b(nb)^{(p-1)/2}\right)^2.
		\]
		The} statement follows from the fact that $b^{-1}(nb)^{p-1}b^2=n^{p-1}b^p\to \infty$.
	
	{Finally}, one can deal in the same way with the $L_p$-error on the interval $(1-b,1]$.
\end{proof}
\begin{proof}[Proof of Proposition \ref{prop:clt_boundaries}]
	By definition we have
	\[
	\begin{split}
	&
	(nb)^{p/2}\int_0^b\left|\hat{\lambda}^s_n(t)-\lambda(t) \right|^p\dd\mu(t)\\
	&\quad=
	\int_0^b\left|(nb)^{1/2}\int_0^{t+b}k^{(t)}_b(t-u)\,\dd (\Lambda_n-\Lambda)(u)+{\bar g_{(n)}}(t)\right|^p\dd\mu(t),
	\end{split}
	\]
	where
	\begin{equation}
	\label{def:bar_g_n}
	\bar{g}_{(n)}(t)=(nb)^{1/2}\left(\int k^{(t)}_b(t-u)\lambda(u)\,\dd u-\lambda(t) \right).
	\end{equation}
	{When $B_n$ in assumption (A2) is a Brownian motion, we can argue} as in the proof of Theorem~\ref{theo:as.distribution_J_n}.
	{By means of~\eqref{eqn:L_p-inequality} we can bound
		\[
		\begin{split}
		&
		b^{-1/2}
		\bigg|
		(nb)^{p/2}\int_0^b\left|\hat{\lambda}^s_n(t)-\lambda(t) \right|^p\dd\mu(t)\\
		&\qquad\qquad\qquad\qquad-
		\int_0^b\left|b^{-1/2}\int_0^{t+b}k^{(t)}\left(\frac{t-u}{b}\right)\,\dd B_n(L(u))+\bar g_{(n)}(t) \right|^p\dd\mu(t)
		\bigg|,
		\end{split}
		\]
		from above by
		\begin{equation}
		\label{eqn:approximation_Brownian functional}
		\begin{split}
		&
		p2^{p-1}b^{-1/2}b^{-p/2}
		\int_0^b\left|\int_0^{t+b}k^{(t)}\left(\frac{t-u}{b}\right)\,\dd (B_n\circ L-n^{1/2}M_n)(u)\right|^p\dd\mu(t)\\
		&\quad+
		p2^{p-1}b^{-1/2}
		\left(
		b^{-p/2}\int_0^b\left|\int_0^{t+b}k^{(t)}\left(\frac{t-u}{b}\right)\,\dd (B_n\circ L-n^{1/2}M_n)(u)\right|^p\dd\mu(t)
		\right)^{1/p}\\
		&\qquad\qquad\qquad\cdot
		\left(
		\int_0^b\left|b^{-1/2}\int_0^{t+b}k^{(t)}\left(\frac{t-u}{b}\right)\,\dd B_n(L(u))+\bar{g}_{(n)}(t) \right|^p\dd\mu(t)
		\right)^{1-1/p}.
		\end{split}
		\end{equation}
		Similar to~\eqref{eqn:embedding_approximation},}
	\begin{equation}
	\label{eqn:embedding_approximation1}
	\begin{split}
	&
	\sup_{t\in[0,b]}
	\left|\int_0^{t+b}k^{(t)}\left(\frac{t-u}{b}\right)\,\dd {(B_n\circ L-n^{1/2}M_n)(u)} \right|\\
	&\quad\leq
	\left|
	\int_{-1}^{t/b}
	{\left\{
		\psi_1\left(\frac{t}{b}\right)k(y)
		+
		\psi_2\left(\frac{t}{b}\right)y k(y)
		\right\}
		\,\dd (B_n\circ L-n^{1/2}M_n)(t-by)}
	\right|\\
	&\quad\leq
	C\sup_{t\in[0,1]}\left|B_n\circ L(t)-n^{1/2}M_n(t)\right|\\
	&\quad=O_P(n^{-1/2+1/q}).
	\end{split}
	\end{equation}
	Note that here we used the boundedness of the coefficients
	{$\psi_1$ and $\psi_2$.
		Similar to the proof of Theorem~\ref{theo:as.distribution_J_n}},
	the idea is to show that
	\begin{equation}
	\label{eq:to show}
	b^{-1/2}\int_0^{b}\left|b^{-1/2}\int_0^{t+b}k^{(t)}\left(\frac{t-u}{b}\right)\,\dd B_n(L(u))+\bar{g}_{(n)}(t)\right|^p\,\dd \mu(t)\to 0,
	\end{equation}
	{in probability.
		We first bound the left hand side of~\eqref{eq:to show} by
		\[
		Cb^{-1/2} \int_0^{b}\left\{|\bar{g}_{(n)}(t) |^p+b^{-p/2}\left|\int_0^{t+b}k^{(t)}\left(\frac{t-u}{b}\right)\,\dd B_n(L(u))\right|^p\right\}\,\dd \mu(t).
		\]
		According to~\eqref{eq:unbiased boundary kernel},
		a Taylor expansion gives
		\[
		\begin{split}
		\sup_{t\in[0,b]}|\bar{g}_{(n)}(t)|&=(nb)^{1/2}\sup_{t\in[0,b]}\left|\int_0^{t+b}k^{(t)}_b(t-u)\lambda(u)\,\dd u-\lambda(t) \right|\\
		&=(nb)^{1/2}\sup_{t\in[0,b]}\left|\int_{-1}^{t/b}k^{(t)}(y)\left[\lambda(t-by)-\lambda(t) \right]\,\dd y\right|\\
		&=(nb)^{1/2}b^2\sup_{t\in[0,b]}\left|\frac{1}{2}\int_{-1}^{t/b}k^{(t)}(y)y^2\lambda''(\xi_{t,y})\,\dd y\right|\\
		&=O_P\left((nb^5)^{1/2} \right)=O_P(1).
		\end{split}
		\]
		Furthermore,
		\[
		\begin{split}
		&\E\left[\left|\int_0^{t+b}k^{(t)}\left(\frac{t-u}{b}\right)\,\dd B_n(L(u))\right|^p\right]\\
		&=\int_{\R}\left(\int_0^{t+b}\left(k^{(t)}\left(\frac{t-u}{b}\right)\right)^2L'(u)\,\dd u\right)^{p/2}|x|^p\phi(x)\,\dd x\\
		&=b^{p/2}\int_{\R}\left(\int_0^{t+b}\left(k^{(t)}\left(\frac{t-u}{b}\right)\right)^2L'(u)\,\dd u\right)^{p/2}|x|^p\phi(x)\,\dd x\\
		&=O(b^{p/2}),
		\end{split}
		\]
		where $\phi$ denotes the standard normal density.
		This proves~\eqref{eq:to show} for the case that $B_n$ is a Brownian motion.
		
		When $B_n$ in (A2)} is a Brownian bridge,
	then {we use the representation} $B_n(u)=W_n(u)-uW_n(L(1))/L(1)$, for some Brownian motion $W_n$.
	In this case, {by means of~\eqref{eqn:L_p-inequality}, we can bound
		\[
		\begin{split}
		&
		b^{-1/2}
		\bigg|
		\int_0^{b}
		\left|
		b^{-1/2}\int_0^{t+b}k^{(t)}\left(\frac{t-u}{b}\right)\,
		\dd B_n(L(u))+\bar{g}_{(n)}(t)
		\right|^p\,\dd \mu(t)\\
		&\qquad\qquad\qquad-
		\int_0^{b}
		\left|
		b^{-1/2}\int_0^{t+b}k^{(t)}\left(\frac{t-u}{b}\right)\,\dd W_n(L(u))
		+
		\bar{g}_{(n)}(t)
		\right|^p\,\dd \mu(t)
		\bigg|
		\end{split}
		\]
		by
		\[
		\begin{split}
		&
		p2^{p-1}b^{-1/2}\int_0^{b}\left|b^{-1/2}\frac{W_n(L(1)}{L(1)}
		\int_0^{t+b}k^{(t)}\left(\frac{t-u}{b}\right)L'(u)\,\dd u+\bar{g}_{(n)}(t)\right|^p\,\dd \mu(t)\\
		&\quad+
		p2^{p-1}b^{-1/2}\left(\int_0^{b}\left|b^{-1/2}\frac{W_n(L(1))}{L(1)}
		\int_0^{t+b}k^{(t)}\left(\frac{t-u}{b}\right)L'(u)\,\dd u+\bar{g}_{(n)}(t)\right|^p\,\dd \mu(t)\right)^{1/p}\\
		&\qquad\qquad\qquad\cdot
		\left(\int_0^{b}\left|b^{-1/2}\int_0^{t+b}k^{(t)}\left(\frac{t-u}{b}\right)\,\dd W_n(L(u))+\bar{g}_{(n)}(t)\right|^p\,\dd \mu(t) \right)^{1-1/p},
		\end{split}
		\]
		which tends to zero in probability, due to~\eqref{eq:to show}.
	}
\end{proof}

\section{Isotonized kernel estimator}
\label{subsec:proofs section GS}

\begin{lemma}
	\label{lemma:monotone}
	{Assume (A1)-(A2) and let $\tilde{\lambda}_n^{s}$ be defined in~\eqref{def:kernel_est}.
		Let $k$ satisfy~\eqref{def:kernel} and let $p\geq 1$.}
	If $b\to0$, {$nb\to\infty$,} and $1/b=o(n^{1/4})$, {then}
	\[
	\p
	\left(
	\tilde{\lambda}_n^{s}\text{ is decreasing on } [b,1-b]
	\right)\to 1.
	\]
\end{lemma}
\begin{proof}
	The proof is completely similar to that of Lemma A.7 in \cite{lopuhaa-mustaJSPI2017}.
	Note that condition~(8) in that paper follows from our Assumption (A2) and that here $\lambda$ is a decreasing function.
	
	We use the fact that on $[b,1-b]$, $\tilde{\lambda}_n^{s} $ is the standard kernel estimator of $\lambda$ given by~\eqref{eqn:naive_kernel} and we get
	\begin{equation}
	\label{eqn:der.naive}
	\frac{\mathrm{d}}{\mathrm{d}t}\tilde{\lambda}_n^{s}(t)=
	\int_{t-b}^{t+b} \frac{1}{b^2}k'\left(\frac{t-u}{b}\right)\,\dd \left(\Lambda_n-\Lambda\right)(u)+
	\int_{t-b}^{t+b} \frac{1}{b^2} k'\left(\frac{t-u}{b}\right)\lambda(u)\,\dd u.
	\end{equation}
	The first term on the right hand side of~\eqref{eqn:der.naive} converges to zero because {in absolute value it is bounded from above by}
	\[
	\frac{1}{b^2}\sup_{x\in[0,1]}\left|\Lambda_n(x)-\Lambda(x)\right|
	\sup_{y\in[-1,1]}|k''(y)|
	=
	O_p(b^{-2}n^{-1/2})=o_p(1),
	\]
	according to Assumption (A2) and the fact that $1/b=o(n^{-1/4})$.
	Moreover, {integration by parts gives}
	\[
	\int\frac{1}{b^2} k'\left(\frac{t-u}{b}\right)\lambda(u)\,\dd u=\int_{-1}^1k(y)\lambda'(t-by)\,\dd y.
	\]
	Hence, the second term on the right hand side of~\eqref{eqn:der.naive} is bounded from above by a strictly negative constant because of Assumption (A1). We conclude that $\tilde{\lambda}_n^{s}$ is decreasing on $[b,1-b]$ with probability tending to one.
\end{proof}

\begin{cor}
	\label{cor:asymptotic_equivalence}
	{Assume (A1)-(A2) and let $\tilde{\lambda}_n^{s}$ and $\tilde{\lambda}_n^{GS}$ be defined in~\eqref{def:kernel_est}
		and Section~\ref{sec:GS}, respectively.
		Let $k$ satisfy~\eqref{def:kernel}.
		Let $0<\gamma<1$ and $p\geq 1$.
		If $b\to0$, {$nb\to\infty$,} and $1/b=o(n^{1/4})$, then}
	\[
	\p
	\left(
	\tilde{\lambda}_n^{s}(t)=\tilde{\lambda}_n^{GS}(t)	\text{ for all } t\in[b^{\gamma},1-b^{\gamma}]
	\right)\to 1.
	\]
\end{cor}
\begin{proof}
	The proof is completely similar to that of {Lemma~3.2} in~\cite{lopuhaa-mustaJSPI2017},
	but now we want to extend the interval to $[b^{\gamma},1-b^{\gamma}]$, which is not fixed but approaches the boundaries as $n\to \infty$.
	In this case we define the linearly extended  version of $\Lambda_n^s$ by
	\[
	\hat\Lambda_n^*(t)
	=
	\begin{cases}
	\Lambda_n^s(b^{\gamma})+\big(t-b^{\gamma}\big)\tilde{\lambda}_n^{s}(b^{\gamma}), &\text{ for } t\in[0,b^{\gamma}),\\
	\Lambda_n^s(t), &\text{ for } t\in[b^{\gamma},1-b^{\gamma}],\\
	\Lambda_n^s(1-b^{\gamma})+\big(t-1+b^{\gamma}\big)\tilde{\lambda}_n^{s}(1-b^{\gamma}), &\text{ for } t\in(1-b^{\gamma},1].
	\end{cases}
	\]
	{Choose $0<\delta<2$.}
	It suffices to prove that, for sufficiently large $n$,
	\begin{equation}
	\label{eq:cor prop1 GS}
	\p
	\left(
	\hat\Lambda_n^*\text{ is concave on }[0,1]
	\right)\geq 1-\delta/2,
	\end{equation}
	and
	\begin{equation}
	\label{eq:cor prop2 GS}
	\p
	\left(
	\hat\Lambda_n^*(t)\geq \Lambda_n^s(t),
	\text{ for all }t\in [0,1]
	\right)\geq 1-\delta/2.
	\end{equation}
	To prove~\eqref{eq:cor prop1 GS}, define the event
	\[
	A_n=\left\{\tilde{\lambda}_n^{s} \text{ is decreasing on } [b,1-b]\right\}.
	\]
	On the event $A_n$ the curve {$\hat\Lambda_n^*$} is concave {on $[0,1]$}, so
	{\[
		\p
		\left(
		\hat\Lambda_n^*\text{ is concave on }[0,1]
		\right)
		\geq
		\p(A_n),
		\]
		and} the result follows from Lemma~\ref{lemma:monotone}.
	To prove~\eqref{eq:cor prop2 GS}, we split the interval $[0,1]$ in five intervals
	$I_1=[0,b),$ $I_2=[b,b^{\gamma})$, $I_3=[b^{\gamma},1-b^{\gamma}]$, $I_4=(1-b^{\gamma},1-b]$ and $I_5=(1-b,1]$.
	Then, as in {Lemma~3.2} in \cite{lopuhaa-mustaJSPI2017}, we show that
	\begin{equation}
	\label{eqn:C GS}
	\p
	\left(
	\hat\Lambda_n^*(t)\geq \Lambda_n^s(t),\text{ for all }t\in I_i
	\right)\geq 1-\delta/10,
	\quad
	i=1,\dots,5.
	\end{equation}
	For $t\in I_3$, $\hat\Lambda_n^*(t)=\Lambda_n^s(t)$, so~\eqref{eqn:C GS} is trivial. For $t\in I_2$, by the mean value theorem,
	\[
	\begin{split}
	\hat\Lambda_n^*(t)-\Lambda_n^s(t)
	=
	\Lambda_n^s(b^{\gamma})+\big(t-b^{\gamma}\big)\tilde{\lambda}_n^{s}(b^{\gamma})-\Lambda_n^s(t)
	=
	(b^{\gamma}-t)\left[\tilde{\lambda}_n^{s}(\xi_t)-\tilde{\lambda}_n^{s}(b^{\gamma})\right],
	\end{split}
	\]
	for some $\xi_t\in(t,b^{\gamma})\subset(b,b^{\gamma})$. Thus,
	\[
	\p
	\left(
	\hat\Lambda_n^*(t)\geq \Lambda_n^s(t),\text{ for all }t\in I_2
	\right)\geq \p(A_n)\geq 1-\delta/10,
	\]
	for $n$ sufficiently large, according to Lemma~\ref{lemma:monotone}. 
	The argument for $I_4$ is exactly the same.
	
	{Next,} we consider $t\in I_1$.
	We have
	\begin{equation}
	\label{eqn:I_1}
	\begin{split}
	&
	\hat\Lambda_n^*(t)-\Lambda_n^s(t)\\
	&\quad=
	\Lambda_n^s(b^{\gamma})+\big(t-b^{\gamma}\big)\tilde{\lambda}_n^{s}(b^{\gamma})-\Lambda_n^s(t)\\
	&\quad=
	\left[\Lambda_n^s(b^{\gamma})-\Lambda^s(b^{\gamma})\right]
	+
	\left[\Lambda^s(t)-\Lambda_n^s(t)\right]
	+
	\Lambda^s(b^{\gamma})-\Lambda^s(t)-\big(b^{\gamma}-t\big)\tilde{\lambda}_n^{s}(b^{\gamma})\\
	&\quad\geq-
	2\sup_{t\in[0,1]}\left|\Lambda_n^s(t)-\Lambda^s(t)\right|
	+
	\Lambda^s(b^{\gamma})-\Lambda^s(t)-(b^{\gamma}-t)\lambda(b^{\gamma})\\
	&\qquad\qquad\qquad\qquad\qquad\qquad\,+
	\big(b^{\gamma}-t\big)\left[\lambda(b^{\gamma})-\tilde{\lambda}_n^{s}(b^{\gamma})\right],
	\end{split}
	\end{equation}
	where $\Lambda^s$ is the deterministic version of $\Lambda_n^s$,
	\[
	\Lambda^s(t)=\int_{(t-b)\vee 0}^{(t+b)\wedge 1 } k^{(t)}_b(t-u)\Lambda(u)\,\dd u.
	\]
	{For the first term on right hand side of~\eqref{eqn:I_1}},
	note that
	\begin{equation}
	\label{eqn:I_2}
	\begin{split}
	\sup_{t\in[0,1]}\left|\Lambda_n^s(t)-\Lambda^s(t)\right|&=\sup_{t\in[0,1]}\left|\int_{(t-b)\vee 0}^{(t+b)\wedge 1 } k^{(t)}_b(t-u)\left[\Lambda_n(u)-\Lambda(u)\right]\,\dd u \right|\\
	&=\sup_{t\in[0,1]}\left|\int k^{(t)}(y)\left[\Lambda_n(t-by)-\Lambda(t-by)\right]\,\dd y \right|\\
	&\leq\sup_{t\in[0,1]}\left|\Lambda_n(t-by)-\Lambda(t-by)\right|\int \sup_{t\in[0,1]}\left|k^{(t)}(y)\right|\,\dd y \\
	&=O_P\left(n^{-1/2}\right),
	\end{split}
	\end{equation}
	{due to Assumption~(A2).}
	Moreover, {for the third term on right hand side of~\eqref{eqn:I_1}},
	for $t\in(b,1-b)$, we have
	\begin{equation}
	\label{eqn:I_3}
	\begin{split}
	\left|\lambda(t)-\tilde{\lambda}_n^{s}(t)\right|&\leq\left|\lambda(t)-\int k_b(t-u)\lambda(u)\,\dd u\right|+\left| \int k_b(t-u)\,\dd(\Lambda-\Lambda_n)(u)\right|\\
	&=\left|\int k(y)[\lambda(t)-\lambda(t-by)]\,\dd y\right|+b^{-1}\left| \int k'(y)(\Lambda-\Lambda_n)(t-by)\,\dd y\right|\\
	&=
	O(b^2)+O_P(b^{-1}n^{-1/2}).
	\end{split}
	\end{equation}
	{For the second term on right hand side of~\eqref{eqn:I_1}},
	for $t\in[0,b)$, {we write}
	{\begin{equation}
		\label{eqn:I_4}
		\begin{split}
		&
		\Lambda^s(b^{\gamma})-\Lambda^s(t)-(b^{\gamma}-t)\lambda(b^{\gamma}) \\
		&\quad=
		\int_{b^{\gamma}-b}^{b^{\gamma}+b } k_b(b^{\gamma}-u)\Lambda(u)\,\dd u-\int_{0}^{t+b} k^{(t)}_b(t-u)\Lambda(u)\,\dd u-(b^{\gamma}-t)\lambda(b^{\gamma}) \\
		&\quad=
		\int_{b^{\gamma}-b}^{b^{\gamma}+b } k_b(b^{\gamma}-u)[\Lambda(u)-\Lambda(b^{\gamma})]\,\dd u
		-
		\int_{0}^{t+b} k^{(t)}_b(t-u)[\Lambda(u)-\Lambda(t)]\,\dd u\\
		&\qquad\qquad\qquad\qquad+
		[\Lambda(b^{\gamma})-\Lambda(t)-(b^{\gamma}-t)\lambda(b^{\gamma})] \\
		&\quad=
		\int_{-1}^{1} k(y)[\Lambda(b^{\gamma}-by)-\Lambda(b^{\gamma})]\,\dd y
		-
		\int_{-1}^{t/b} k^{(t)}(y)[\Lambda(t-by)-\Lambda(t)]\,\dd y\\
		&\qquad\qquad\qquad\qquad-
		\frac{1}{2}(b^{\gamma}-t)^2\lambda'(\xi_t)\\
		&\quad\geq
		\int_{-1}^{1} k(y)[\Lambda(b^{\gamma}-by)-\Lambda(b^{\gamma})]\,\dd y
		-
		\int_{-1}^{t/b} k^{(t)}(y)[\Lambda(t-by)-\Lambda(t)]\,\dd y\\
		&\qquad\qquad\qquad\qquad-
		\inf_{t\in[0,1]}|\lambda'(t)|b^{1+\gamma}
		+
		\frac12\inf_{t\in[0,1]}|\lambda'(t)|b^{2\gamma}
		\end{split}
		\end{equation}
		where $\xi_t\in(t,b^{\gamma})$.
		Furthermore, the first two integrals on the right hand side can be written as
		\[
		\begin{split}
		&\frac{b^2}{2}\int_{-1}^{1} k(y)y^2\lambda'(\xi_{1,y})\,\dd y
		-
		\frac{b^2}{2}\int_{-1}^{t/b} k^{(t)}(y)y^2\lambda'(\xi_{2,y})\,\dd y\\
		&\quad\geq-
		\frac{b^2}{2}\left|\int_{-1}^{1} k(y)y^2\lambda'(\xi_{1,y})\,\dd y
		-
		\int_{-1}^{t/b} k^{(t)}(y)y^2\lambda'(\xi_{2,y})\,\dd y \right|\\
		&\quad\geq-
		\frac{b^2}{2}\left|\int_{-1}^{1} k(y)y^2\lambda'(\xi_{1,y})\,\dd y-\int_{-1}^{t/b} k^{(t)}(y)y^2\lambda'(\xi_{2,y})\,\dd y \right|
		=
		O(b^2),
		\end{split}
		\]
		with $\xi_t\in(t,b^{\gamma})$, $|\xi_{1,y}-b^{\gamma}|\leq by$ and $|\xi_{2,y}-t|\leq by$.
		This means that
		\[
		\p\left(\hat\Lambda_n^*(t)-\Lambda_n^s(t)\geq 0, \text{ for all } x\in I_1\right)
		\geq
		\p\left(Y_n\leq \frac12\inf_{t\in[0,1]}|\lambda'(t)|b^{2\gamma}\right),
		\]
		where
		\[
		Y_n
		=
		O_P(n^{-1/2})
		+
		O(b^\gamma)
		\left\{
		O(b^2)
		+
		O_P(b^{-1}n^{-1/2})
		\right\}
		+
		O(b^2)
		-
		\inf_{t\in[0,1]}|\lambda'(t)|b^{1+\gamma}
		=
		O_P(b^{1+\gamma}).
		\]
	}
	Hence, for $n$ large enough, this probability is greater than $1-\delta/10$, because $\gamma<1$.
\end{proof}

\section{CLT for the Hellinger loss}
\label{supp_Hellinger}
\begin{lemma}Assume (A1)-(A3) hold.	If $\lambda$ is strictly positive, we have
	\label{lemma:Hellinger_approximation}
	\[
	\int_0^1\left(\sqrt{\hat{\lambda}_n^s(t)}-\sqrt{\lambda(t)}\right)^2\,\dd\mu(t)=\int_0^1\left(\hat{\lambda}_n^s(t)-\lambda(t)\right)^2(4\lambda(t))^{-1}\,\dd\mu(t) +O_P\left((nb)^{-3/2}\right).
	\]
	The previous results holds also if we replace $\hat{\lambda}_n^s$ with the smoothed Grenander-type estimator $\tilde{\lambda}_n^{SG}$.
\end{lemma}
\begin{proof}
	As in the proof of Lemma 2.1 in \cite{LM_Hellinger} we get
	\[
	\int_0^1\left(\sqrt{\hat{\lambda}_n^s(t)}-\sqrt{\lambda(t)}\right)^2\,\dd\mu(t)=\int_0^1\left(\hat{\lambda}_n^s(t)-\lambda(t)\right)^2(4\lambda(t))^{-1}\,\dd\mu(t)+R_n,
	\]
where 
\[
|R_n|\leq C\int_0^1\left|\hat{\lambda}_n^s(t)-\lambda(t)\right|^3\,\dd\mu(t)
\]
for some positive constant $C$ only depending on $\lambda(0)$ and $\lambda(1)$. Then, from Corollary \ref{cor:as.distribution_J_n}, it follows that $R_n=O_P\left((nb)^{-3/2}\right)$.
When dealing with the smoothed Grenander-type estimator, the result follows from Theorem \ref{theo:as.distribution_I_n}.
\end{proof}
\begin{theo}
	\label{theo:clt_Hellinger_squared}
	Assume (A1)-(A3) hold and that $\lambda$ is strictly positive.
\begin{enumerate}[label={\roman*)}]
	\item If $nb^5\to0$, then it holds
	\[
	(b\sigma^{2,*}(2))^{-1/2}\left\{2nbH(\hat{\lambda}_n^s,\lambda)^2-m_n^*(2) \right\}\xrightarrow{d} N(0,1).
	\]
	\item If $nb^5\to C_0^2>0$ and $B_n$ in Assumption (A2) is a Brownian motion, then it holds
	\[
	(b\theta^{2,*}(2))^{-1/2}\left\{2nbH(\hat{\lambda}_n^s,\lambda)^2-m_n^*(2) \right\}\xrightarrow{d} N(0,1),
	\]
	\item If $nb^5\to C_0^2>0$ and $B_n$ in Assumption (A2) is a Brownian bridge, then it holds
	\[
	(b\tilde\theta^{2,*}(2))^{-1/2}\left\{2nbH(\hat{\lambda}_n^s,\lambda)^2-m_n^*(2) \right\}\xrightarrow{d} N(0,1),
	\]
\end{enumerate}
where $\sigma^{2,*}$, $\theta^{2,*}$, $\tilde\theta^{2,*}$ and $m_n^*$ are defined, respectively, as in \eqref{eqn:def-sigma}, \eqref{eqn:def-theta}, \eqref{eqn:def-theta-tilde} and \eqref{eqn:def-m_n} by replacing $w(t)$ with $w(t)(4\lambda(t))^{-1}$.

If $p<min(q,2q-7)$ and $1/b=o\left(n^{(1/3-1/q)\min(q/(2p),1)}\right)$, the same results hold also when replacing $\hat{\lambda}_n^s$ by the smoothed Grenander-type estimator $\tilde{\lambda}_n^{SG}.$
\end{theo}
\begin{proof}
According to Lemma \ref{lemma:Hellinger_approximation}, it is sufficient to show that the results hold if we replace $2H(\hat{\lambda}_n^s,\lambda)^2$ by 
\[
\int_0^1\left(\hat{\lambda}_n^s(t)-\lambda(t)\right)^2(4\lambda(t))^{-1}\,\dd\mu(t)=\int_0^1\left(\hat{\lambda}_n^s(t)-\lambda(t)\right)^2\,\dd\tilde\mu(t),
\]	
where
\[
d\tilde\mu(t)=\frac{1}{4\lambda(t)}\,\dd\mu(t)=\frac{w(t)}{4\lambda(t)}\,\dd t.
\]
It suffices to apply Corollary \ref{cor:as.distribution_J_n} with a weight $\tilde{\mu}$ instead of $\mu$.
 
For the smoothed Grenander estimator the result would follow from Theorem \ref{theo:as.distribution_I_n}. 
\end{proof}
\end{document}

%% file: L_p-frontmatter_AOS_RL.tex
\begin{frontmatter}
\title{Central limit theorems for the $L_p$-error of smooth isotonic estimators}
\runtitle{CLT for the $L_p$-error of smooth isotonic estimators}

\begin{aug}
\author{\fnms{Hendrik P.} \snm{Lopuha\"a}\ead[label=e1]{h.p.lopuhaa@tudelft.nl}},
\author{\fnms{Eni} \snm{Musta}\ead[label=e2]{e.musta@tudelft.nl}}

\runauthor{H.P.~Lopuha\"a and E.~Musta}

\affiliation{Delft University of Technology}

\address{Hendrik P. Lopuha\"a and Eni Musta\\
DIAM, Delft University of Technology, \\
van Mourik Broekmanweg 6, 2628 XE, Delft, Netherlands\\
\printead{e1}\\
\printead{e2}
}
\end{aug}

\begin{abstract}
We investigate the asymptotic behavior of the $L_p$-distance between a monotone function on a compact interval and a smooth estimator of this function. 
Our main result is a central limit theorem for the $L_p$-error of smooth isotonic estimators obtained by smoothing a Grenander-type estimator 
or isotonizing the ordinary kernel estimator.  
As a preliminary result we establish a similar result for ordinary kernel estimators.
Our results are obtained in a general setting, which includes estimation of a monotone density, regression function and hazard rate. 
We also perform a simulation study for testing monotonicity on the basis of the $L_2$-distance between the kernel estimator and the smoothed Grenander-type estimator.

\end{abstract}

\begin{keyword}[class=MSC]
\kwd[Primary ]{62G20}
\kwd[; secondary ]{62G10}
\end{keyword}

\begin{keyword}
\kwd{kernel estimator}
\kwd{$L_p$ loss}
\kwd{central limit theorem}
\kwd{smoothed Grenander-type estimator}
\kwd{isotonized kernel estimator}
\kwd{boundary corrections}
\kwd{Hellinger loss}
\kwd{testing monotonicity}
\end{keyword}
\end{frontmatter}